\documentclass[a4paper]{amsart}

\usepackage[utf8]{inputenc}
\usepackage[english]{babel}
\usepackage[T1]{fontenc}
\usepackage{lmodern}
\usepackage{dirtytalk}
\usepackage{comment}

\usepackage{babelbib}

\usepackage{amsmath, amssymb, amsthm}   
\usepackage{amstext, amsfonts, mathrsfs}
\usepackage[svgnames]{xcolor}

\usepackage{thmtools}

\usepackage{tikz-cd}

\usepackage[all]{xy} 
\usepackage{centernot}
\usepackage{array}
\usepackage{enumerate}
\usepackage{graphics,graphicx}
\usepackage{pinlabel}
\usepackage{float}
\usepackage{ae,aecompl}
\usepackage{eso-pic}
\usepackage{pst-all}
\usepackage{caption}
\usepackage{subcaption}

\usepackage{tabularx}
\usepackage{cases}
\usepackage[shortlabels]{enumitem}

\usepackage{hyperref}

\usepackage[normalem]{ulem}

\usepackage{thm-restate}

\newtheorem{thm}{Theorem}[section]
\newtheorem{prop}[thm]{Proposition}
\newtheorem{fact}[thm]{Fact}
\newtheorem{claim}[thm]{Claim}

\newtheorem{lem}[thm]{Lemma}

\newtheorem{coro}[thm]{Corollary}

\newtheorem{mainthm}{Theorem}    
 
\newtheorem{ques}[mainthm]{Question} 
\newtheorem{thmintro}{Theorem}

\theoremstyle{definition}
\newtheorem{defi}[thm]{Definition}
\newtheorem{nota}[thm]{Notation}
\newtheorem{example}[thm]{Example}

\newtheorem{convention}[thm]{Convention}

\theoremstyle{remark}
\newtheorem{rmk}[thm]{Remark}

\numberwithin{equation}{section}

\def\nn{\noindent}

\def\qq{\nn\quad}

\def\ssm{\ensuremath{{\smallsetminus}}}
\def\inv{\ensuremath{{-1}}}

\newcommand{\res}[2]
{\ensuremath{#1}{}_{\, {\vrule height 0.3cm} \, \ensuremath{#2}}}

\def\N{{\mathbb N}}    
\def\Z{{\mathbb Z}}    
\def\R{{\mathbb R}}

\def\T{{\mathbb T}}
\def\D{{\mathbb D}}

  \def\cG{{\mathcal G}}      \def\cN{{\mathcal N}} \def\cT{{\mathcal T}} \def\cC{{\mathcal C}}   \def\cO{{\mathcal O}}     \def\cP{{\mathcal P}} \def\cV{{\mathcal V}}   \def\cK{{\mathcal K}} \def\cQ{{\mathcal Q}}  \def\cF{{\mathcal F}}  \def\cL{{\mathcal L}}

\def\intr{\operatorname{int}}

\def\per{{\mathrm{per}}}

\def\pap{pseudo-Anosov piece}
\def\paps{pseudo-Anosov pieces}
\def\pa{pseudo-Anosov}
\def\pagp{pseudo-Anosov gluing map}

\def\paf{pseudo-Anosov flow}
\def\pafs{pseudo-Anosov flows}

\def\qms{quasi-Morse-Smale}

\def\cc{connected component}
\def\ccs{connected components}
\def\iin{\ensuremath{\mathrm{in}}}
\def\out{\ensuremath{\mathrm{out}}}
\def\Pin{\ensuremath{P^\iin}}
\def\Pout{\ensuremath{P^\out}}

\def\pP{\ensuremath{\partial P}}

\def\intr{\ensuremath{\mathrm{int}}}

\def\Per{\ensuremath{\mathscr{P}}}

\def\mm{{\ensuremath{**}}}

\PassOptionsToPackage{unicode}{hyperref}
\hypersetup{ 
linktocpage=true,
colorlinks=true,
breaklinks=true, 
pdfencoding=auto,
pdftitle={}, pdfauthor={}, pdfsubject={}
}

\definecolor{jauneclair}{RGB}{255,245,190}
\definecolor{orangeclair}{RGB}{255,225,180}
\definecolor{roseclair}{RGB}{250,220,245}

\definecolor{monviolet}{RGB}{125,0,255}

\title[Uniqueness of gluings]{Uniqueness of gluings and virtual finiteness of pseudo-Anosov flows on graph manifolds}

\author[Thomas Barthelm\'e]{Thomas Barthelm\'e}
\address{Queen's University, Kingston, Ontario}
\email{thomas.barthelme@queensu.ca}
\urladdr{sites.google.com/site/thomasbarthelme}

\author[Neige Paulet]{Neige Paulet}
 \address{Queen's University, Kingston, Ontario}
 \email{neige.paulet@queensu.ca}
\urladdr{ }

\date{\today}

\begin{document}
\maketitle

\begin{abstract}
    In this article, we give a characterization of when two pseudo-Anosov flows obtained via gluings of pieces of pseudo-Anosov flows are orbit equivalent. As an application of this work, and the description of pseudo-Anosov flows in Seifert pieces due to Barbot and Fenley, we prove a ``virtual'' version of the Finiteness Conjecture for transitive pseudo-Anosov flows on graph-manifolds.
\end{abstract}

\section{Introduction}

One approach to study and classify pseudo-Anosov flows on toroidal $3$-manifolds that was successfully developed over many years by Barbot and Fenley (\cite{barbotMisePositionOptimale1995,Bar96,barbotPseudoAnosovFlowsToroidal2013a,barbotClassificationRigidityTotally2015,barbotFreeSeifertPieces2021}) is to use a JSJ decomposition of the manifold that is adapted to the flow, and then describe the restriction to the pieces. This strategy has been particularly successful on Seifert pieces, as they obtained a complete description of the restriction of the flow inside such pieces (\cite{barbotClassificationRigidityTotally2015,barbotFreeSeifertPieces2021}).

What had been missing so far with this strategy was general conditions allowing us to uniquely recover the original flow from the ``model'' obtained in each piece. More precisely, while the topology of the graph manifold $M$ is completely determined by the topology of each piece together with the isotopy class of the gluing maps, there was no general understanding of which gluings (in a fixed isotopy class) give the same flows up to orbit equivalence, i.e., the following question was left open:

\begin{ques}\label{ques_graph_manifold_same_model}
    Let $\phi_1,\phi_2$ be two flows on a graph-manifold $M$ such that for any piece $P$ of the JSJ decomposition the ``models'' of $\phi_1|_P$ and $\phi_2|_P$ are orbit equivalent. 
    What are necessary and sufficient conditions on the gluing maps that ensure that $\phi_1$ and $\phi_2$ are orbit equivalent?
\end{ques}

The first main result of this article is to give an answer to this question, on any toroidal $3$-manifold, when the flows are assumed to be transitive: 

\begin{thmintro}\label{thmintro_advertisement}
    Let $\phi_1,\phi_2$ be two transitive pseudo-Anosov flows on an orientable (toroidal) $3$-manifold $M$ such that for any piece $P$ of the JSJ decomposition the restriction of $\phi_1|_P$ and $\phi_2|_P$ are orbit equivalent. Then $\phi_1$ and $\phi_2$ are orbit equivalent if and only if they have the same \emph{complete bar code}.
\end{thmintro}
     Informally, the complete bar code for $\phi$ is a \emph{finite} combinatorial data that is determined by, and recovers, the traces of the stable and unstable foliations of $\phi$ on each of the tori of the (adapted) JSJ decomposition of $M$. See below and Definition \ref{def: bar code} for the precise definition.

As an application of this result, and thanks to the work of Barbot and Fenley classifying pseudo-Anosov flows on Seifert pieces, we are able to prove a version of the \emph{Finiteness Conjecture} for pseudo-Anosov flows on graph-manifolds.
The Finiteness Conjecture, now listed as Problem 2.31 in Kirby's third list \cite{K3_list}, has attracted quite a lot of attention in the recent past, with many partial results (see, e.g., \cite{BM24,BSZ25,BTZ26,CP26,LT26}). It states that for a fixed $3$-manifold $M$, there should only be finitely many distinct orbit-equivalent classes of pseudo-Anosov flows on $M$.\footnote{Two independent proofs of the Finiteness Conjecture on \emph{atoroidal} manifolds were recently announced in upcoming works by Gabai and Li, and by Bowden and Colin.}

Here, we obtain a ``virtual finiteness'' version of the conjecture:
\begin{restatable}{thmintro}{finitenessgraph}
     \label{thmintro: finiteness graph}
    Let $M$ be a graph manifold and $\{\phi_i\}$ a family of transitive pseudo-Anosov flows on $M$. There exists a finite cover $\hat M$ and an integer $n$, both depending only on the topology of $M$, such that the lifted flows $\{\hat \phi_i\}$  belong to at most $n$ distinct orbit-equivalent classes.
\end{restatable}

In order to describe more precisely the general versions of the above results, we need to introduce some notations.

\subsection{Precise statement of results}

The idea behind Theorem \ref{thmintro_advertisement} is to be able to decide whether two flows that are obtained by gluing the same pieces of flow along tori are orbit equivalent. Note that aside from the work of Barbot and Fenley mentioned above that aimed for a classification of pseudo-Anosov flows by classifying the restriction of the flow to pieces, there has also been a lot of \emph{constructions} of Anosov flows by gluing distinct pieces \cite{bonattiExempleFlotAnosov1994,beguinBuildingAnosovFlows2017,pauletAnosovFlowsDimension2025,barbotPseudoAnosovFlowsToroidal2013a}, and the question of which gluings lead to orbit equivalent flows was only resolved for some cases (see \cite{beguinUniquenessTheoremTransitive2023}). 
Here, we will define what we mean by a piece of a pseudo-Anosov flow in the following way:

Given a manifold $M$, carrying a pseudo-Anosov flow $\phi$, a collection of disjoint tori $\cT$ that are \emph{quasi-transverse} (i.e., transverse except along finitely many periodic orbits, see Definition \ref{def: qt surface}), we will call a \emph{pseudo-Anosov piece} the couple $(P,\varphi)$, where $P$ is the completion of a (disjoint union of) component(s) of $M\smallsetminus \cT$, and $\varphi = \phi|_P$ is the restricted flow (see Definition \ref{def: pseudo anosov piece}). Note that when we consider a pseudo-Anosov piece $(P, \varphi)$, the original manifold $M$, and original flow $\phi$ are irrelevant. Instead, what is particularly relevant for this article are all the ways that one may glue the boundary components of $\partial P$. To study this, we define a \textit{pseudo-Anosov triple}, $(P, \varphi, f)$, as the data of a pseudo-Anosov piece $(P,\varphi)$ together with a map $f\colon \partial P \to \partial P$ such that the flow $\varphi$ induces a pseudo-Anosov flow $\varphi_f$ on the quotient manifold $M_f = P/f$ (see Definition \ref{def: gluing map}).

With this terminology, we address the following generalization of Question \ref{ques_graph_manifold_same_model}:
\begin{ques}
   Let $(P_0,\varphi_0,f_0)$ and $(P_1,\varphi_1,f_1)$ be two pseudo-Anosov triples. What are necessary and sufficient conditions for the flow $(\varphi_0)_{f_0}$ on $M_{f_0} = P_0/f_0$ and $(\varphi_1)_{f_1}$ on $M_{f_1} = P_1/f_1$ to be orbit equivalent?
\end{ques}

The first condition we will need is that the pseudo-Anosov pieces themselves are orbit equivalent, and that the manifolds $M_{f_0}$ and $M_{f_1}$ are diffeomorphic: We say that two pseudo-Anosov triples $(P_0,\varphi_0,f_0)$ and $(P_1,\varphi_1,f_1)$ are \emph{equivalent} if the pieces $(P_0, \phi_0)$ and $(P_1, \phi_1)$ are orbit equivalent via a map $H\colon P_0 \to P_1$ and the gluing maps $f_0$ and $H^{-1}\circ f_1 \circ H$ are homotopic. See Definition \ref{def: equivalent triple}.

The second condition we will require is that the complete bar code of the two triples are the same. As its definition is technical, we leave it for later (see Definition \ref{def: bar code}). For this introduction, it is enough to know that the complete bar code is just a finite combinatorial data that uniquely determines the traces of the stable and unstable foliations of $(\varphi_i)_{f_i}$ on (the images in $M_{f_i}$ of) $\partial P_i$.

In \cite{BM24,BFrM25}, an often complete invariant of the orbit equivalence class of \emph{transitive} pseudo-Anosov flows was introduced: the \emph{free homotopy data} of a flow.
Formally, given $\phi$ a flow on a manifold $M$, the free homotopy data of $\phi$ is the set $\cP(\phi)$ of elements $\gamma\in \pi_1(M)$ such that either $\gamma$ or $\gamma^{-1}$ are freely homotopic to a periodic orbit of the flow $\phi$. For two flows $\phi_0$ on $M_0$ and $\phi_1$ on $M_1$, we say that their free homotopy data are equivalent if there exists a homeomorphism $H\colon M_0 \to M_1$ such that $H_\ast \cP(\phi_0) = \cP(\phi_1)$.

The work that we do in this article is to show that, under the right conditions on the pseudo-Anosov triple, the resulting flows have equivalent free homotopy data, i.e., the more general version of Theorem \ref{thmintro_advertisement} is:
\begin{restatable}{thmintro}{uniquenessgluing}
 \label{thmintro: uniqueness of gluing}
    Two equivalent triples with same complete bar codes induce \pa{} flows with equivalent free homotopy data.
Moreover, if one of them is transitive, then they are orbit equivalent.
\end{restatable}

Notice that while complete bar codes are \emph{not} uniquely determined by a pseudo-Anosov triple, they are finitely determined:
\begin{restatable}{propintro}{finitenessbarcode}
\label{propintro: finiteness of bar code}
            There are finitely many complete bar codes in the equivalent class of a \pa{} triple $(P, \phi, f)$.
\end{restatable}
This result together with Theorem \ref{thmintro: uniqueness of gluing} implies that, starting from a single pseudo-Anosov piece, one can only build finitely many distinct flows on the same manifold (see Corollary \ref{coro: finiteness of free homotopy data}), and is essential for our application to the Finiteness Conjecture.

During our proof of Theorem \ref{thmintro: uniqueness of gluing}, we will work in the \emph{orbit space} $\cQ_\phi$ of a pseudo-Anosov flow $\phi$. The orbit space, introduced in works of Barbot \cite{barbotCharacterizationAnosovFlows1995} and Fenley \cite{fenleyAnosovFlows3manifolds1994}, is a topological plane obtained as the quotient of $\widetilde M$ by the orbits of $\widetilde \phi$, together with the projection of the stable and unstable foliations and the induced action of $\pi_1(M)$. (We recall below some background on the orbit space and refer to \cite{barthelmePseudoAnosovFlowsPlane2025} for a comprehensive treatment.)

Given a pseudo-Anosov triple, $(P,\varphi,f)$, the images of $\partial P$ in $M_f$, lifted to $\widetilde M_f$ and projected to the orbit space, give structures called \emph{chains of lozenges}, and the ``combinatorial'' type of these chains (i.e.,  the data of whether consecutive lozenges share a stable side, an unstable side, or a corner, see section \ref{subsec: trace qt tori}) is another characterization of the complete bar code. Moreover, the intersection pattern of the different chains of lozenges defines a partial order $\prec$ on these chains (see section \ref{subsec_lines_order}).

The key result which allows us to prove Theorem \ref{thmintro: uniqueness of gluing}, and may be of more general interest, gives a condition for the free homotopy data of two flows to be the same in terms of this order:
\begin{restatable}{thmintro}{toolthm}
\label{thmintro: tool thm for free homotopy data}
    Let $\phi_1$ and $\phi_2$ be pseudo-Anosov flows on a closed $3$-manifold $M$.
For $i=1,2$, let $\cT_i$ be good collections of quasi-transverse tori for $\phi_i$.
Assume that:
\begin{enumerate}
    \item\label{thm:tool_assumption_piecewise}
    $\Per\bigl(\res{\phi_1}{M\setminus \cT_1}\bigr)=\Per\bigl(\res{\phi_2}{M\setminus \cT_2}\bigr)$;
    \item\label{thm:tool_assumption_homeo}
    there exists a homeomorphism $\xi\colon M\to M$, homotopic to the identity, such that $\xi(\cT_1)=\cT_2$
    and $\xi$ maps compact leaves of $\cF^{s/u}_1\cap \cT_1$ to compact leaves of $\cF^{s/u}_2\cap \cT_2$;
    \item\label{thm:tool_assumption_order}
    for the induced chain isomorphisms $\bar \xi$ (Lemma~\ref{lem: isomorphism of chain}), $\bar \xi$ preserves the order relation~$\prec$ on maximal lines of lozenges in the traces.
\end{enumerate}
Then
$\Per(\phi_1)=\Per(\phi_2)$.
\end{restatable}

In the above result, $\cP(\res{\phi_i}{M \setminus \cT_i})$ denotes the set of free homotopy classes of periodic orbits that do \emph{not} cross any torus in $\cT_i$. 
The condition of Item \ref{thm:tool_assumption_order} can be explained in the following way. Up to conjugating the flow $\phi_2$ by $\xi$, one may assume that $\cT_1$ and $\cT_2$ are equal, with exactly the same compact leaves. Now, given a connected component $P$ of $M\smallsetminus \cT_1$,  each torus $T$ in $\partial P$ can be decomposed into maximal annuli with the property that each annulus contains compact leaves of only one foliation. Given two such annuli $A,B$, there may or may not exist an orbit of $\phi_i$ going from $A$ to $B$. The requirement of item \ref{thm:tool_assumption_order} is that if there exists an orbit $\alpha_1$ of $\phi_1$ between $A$ and $B$, then there also exists an orbit $\alpha_2$ of $\phi_2$ from $A$ to $B$ \emph{and} $\alpha_1$ and $\alpha_2$ are freely homotopic relative to $A \cup B$.

Note that versions of Theorem \ref{thmintro: tool thm for free homotopy data} were contained in \cite{BM24} (in the case when the flows $\phi_i$ are \emph{skew}) and \cite{BFM25} (in the case where all the tori in $\cT_i$ are \emph{scalloped}) and inspired the above generalization.

One application of Theorem \ref{thmintro: tool thm for free homotopy data} that we mention in passing (Corollary \ref{coro: 3D dehn twist}) is that one can use it to determine when a 3-dimensional Dehn twist is a self-orbit equivalence of a transitive pseudo-Anosov flow, generalizing some results of \cite{BM24}.

Finally, our application to the Finiteness Conjecture is slightly more general than Theorem \ref{thmintro: finiteness graph} stated above: We allow the manifold to have mixed JSJ decomposition, as long as the restriction to each hyperbolic piece satisfies an additional condition. We will say that a \pap{} $(P,\varphi)$ is \textit{skew} if it is orbit equivalent to a \pap{} $(P', \varphi')$ cut in a skew Anosov flow (see Definition \ref{def: skew pap}).

\begin{restatable}{thmintro}{finitenessgeneral}
 \label{thmintro: finiteness orbit equivalence}
 Let $M$ be a closed $3$-manifold and $\{\phi_i\}$ a family of transitive pseudo-Anosov flows on $M$ that are skew on each hyperbolic piece of the JSJ decomposition of $M$. Then there exists a finite cover $\hat M$ and an integer $n$, both depending only on the topology of $M$, such that the lifted flows $\{\hat \phi_i\}$  belongs to at most $n$ distinct orbit-equivalent classes.
\end{restatable}

Note that to deduce the Finiteness Conjecture from the virtual version we obtained, one could hope to answer positively the following question, which is interesting in and of itself:

\begin{ques}\label{que_oe_in_lift}
    Let $M$ be a closed $3$-manifold and $\{\phi_i\}$ a family of (transitive) pseudo-Anosov flows on $M$. Suppose that there exists a finite cover $\hat M$ such that all the lifts $\hat \phi_i$ are orbit equivalent in $\hat M$. Does there exist $n$ such that $\{\phi_i\}$ belongs to at most $n$ distinct orbit equivalent classes? 
\end{ques}
Note that if the flows $\hat \phi_i$ are orbit equivalent via a homeomorphism that is homotopic to the identity, then the $\phi_i$ themselves are orbit equivalent (see \cite[Proposition 2.5]{BTZ26}). It implies that the answer to Question \ref{que_oe_in_lift} is positive on hyperbolic manifolds as they have finite mapping class group, but the toroidal case seems to be wide open. We do not even know of an example where $n\neq 1$.

The reason we need to pass to a finite cover in the preceding theorems is somewhat technical and amounts to the fact that, for our argument to work with what is currently in the literature, we need the flows to satisfy some additional conditions. If these conditions are satisfied in $M$, then we do not need to pass to the finite cover. More precisely:

\begin{restatable}{thmintro}{finitenessextra}
    \label{thmintro: finiteness extra conditions}
 Let $M$ be a closed $3$-manifold and $\{\phi_i\}$ a family of transitive pseudo-Anosov flows on $M$. Suppose that:
 \begin{enumerate}[label=(\roman*)]
     \item $M$ is orientable;
     \item Each Seifert piece $P$ of the JSJ decomposition of $M$ has base an orientable surface and orientable fibers;
     \item Each $\phi_i$ is skew on each hyperbolic piece of the JSJ decomposition of $M$;
     \item Each torus $T$ of the JSJ decomposition is isotopic to an \emph{embedded} quasi-transverse torus for $\phi_i$.
 \end{enumerate}
 Then there exists $n$ depending only on $M$ such that the flows $\{\phi_i\}$ belong to at most $n$ distinct orbit equivalent classes.
\end{restatable}

\subsection*{Outline}
In Section \ref{sec: preli}, we recall some background on pseudo-Anosov flows, orbit spaces, quasi-transverse tori, and the modified JSJ decomposition of Barbot--Fenley. We also prove a result that has not yet appeared in the literature (Proposition \ref{prop: cover embedded jsj tori}) which states that one may always take a finite cover so that all the modified JSJ tori are embedded.

In Section \ref{sec: foliation QT tori}, we cover quasi-Morse-Smale laminations and foliations, and introduce the notion of bar code.

In Section \ref{sec: pA pieces}, we focus on the definitions and prove some properties of pseudo-Anosov pieces, pseudo-Anosov triples, and their associated bar code and return maps.

After these preliminary sections, the proof of Theorem \ref{thmintro: tool thm for free homotopy data} is done in 
Section \ref{sec: free data block}. In Section \ref{sec: block gluing}, we deduce Theorems \ref{thmintro_advertisement}, \ref{thmintro: uniqueness of gluing}, and Proposition \ref{propintro: finiteness of bar code}. Finally, in Section \ref{sec: finiteness flow}, we prove Theorem \ref{thm: finiteness free homotopy data}, which is a more general version of the virtual finiteness results stated above (Theorems \ref{thmintro: finiteness graph}, \ref{thmintro: finiteness orbit equivalence}, and \ref{thmintro: finiteness extra conditions}).

\subsection*{Acknowledgements}
We thank Katie Mann and François Béguin for many discussions about the content of this article.
Both authors were partially supported by the NSERC Discovery (RGPIN-2024-04412) and Alliance International programs (ALLRP 598447 - 24).

\section{Preliminaries}
\label{sec: preli}

In this section, we will quickly recall a number of results on pseudo-Anosov flows and the orbit space. We refer the reader to \cite{barthelmePseudoAnosovFlowsPlane2025} for a comprehensive introduction to this subject.

\subsection{Pseudo-Anosov flows}
\label{sec: preli; subsec: pa flows}

As there are many definitions of pseudo-Anosov flows in the literature, we recall the one we will be working with. Mostly, we will be working with \emph{topological} pseudo-Anosov flows, and sometimes we will need to consider smooth versions.

\begin{defi}[Pseudo-Anosov flows]
A (topological) pseudo-Anosov flow is a flow $\phi^t \colon M \to M$ with no fixed point, satisfying
the following conditions:

\begin{enumerate}
\item There are two continuous, topologically transverse, 2-dimensional
foliations $\cF^s$ and $\cF^u$ whose leaves are saturated by orbits of $\phi$
and intersect along orbits of $\phi$.  
Singularities lie along a finite (possibly empty) collection of periodic orbits
$\alpha_1,\ldots,\alpha_n$.  
Each $\alpha_i$ is locally homeomorphic to the singular locus $0\times(0,1)$
of a model $p_i$-prong, for some $p_i \ge 3$.

\item For any $\varepsilon>0$ sufficiently small, if $y\in \cF^s_\varepsilon(x)$
(resp.\ $y\in \cF^u_\varepsilon(x)$), then there is a standard reparameterization
$\sigma$ such that
\[
d\bigl(\phi_t(x),\, \phi_{\sigma(t)}(y)\bigr)\longrightarrow 0
\quad\text{as } t\to +\infty\ \text{(resp.\ } t\to -\infty\text{).}
\]

\item For any $\varepsilon>0$ sufficiently small, there exists $\delta>0$ such that
if $y\in \cF^s_\varepsilon(x)$ (resp.\ $y\in \cF^u_\varepsilon(x)$) is not in the same
$\varepsilon$-local orbit as $x$, then for any standard reparameterization $\sigma$
there exists $t<0$ (resp.\ $t>0$) such that
\[
d\bigl(\phi_t(x),\, \phi_{\sigma(t)}(y)\bigr)>\delta .
\]
\end{enumerate}
\end{defi}

\textit{Topological Anosov flows} are pseudo-Anosov flows without singular periodic orbits.

In Section \ref{sec: finiteness flow}, we will have to work with \emph{smooth} pseudo-Anosov flows instead of just topological ones. The notion of smooth Anosov flows is classical (see, e.g., \cite{FH_hyperbolicflows} or \cite{barthelmePseudoAnosovFlowsPlane2025}), as for smooth pseudo-Anosov flows they admit the usual hyperbolic splitting definition of smooth Anosov flows outside the singular orbits, and have Lipschitz regularity at the singular orbits. Since the exact definition is not essential in this work, we refer to \cite[Definition 5.9]{AT24} for details.

Note that thanks to a result of Shannon \cite{Sha25}, extended to pseudo-Anosov flows by Agol--Tsang \cite{AT24}, any \emph{transitive} topological pseudo-Anosov flow is orbit equivalent to a smooth one.

\begin{defi}[Orbit space]\label{def: orbit space}
Let $\phi$ be a (topological) pseudo-Anosov flow on a closed $3$-manifold $M$, and let
$\widetilde\phi$ denote its lift to the universal cover $\widetilde M$.
The \emph{orbit space} of $\phi$, denoted $\cQ_\phi$, is the quotient
\[
\cQ_\phi \ :=\ \widetilde M \big/ \widetilde\phi,
\]
i.e., the space of $\widetilde\phi$-orbits in $\widetilde M$, endowed with the quotient topology.
We denote by $p\colon \widetilde M \to \cQ_\phi$ the quotient map.

The stable and unstable foliations $\widetilde \cF^s,\widetilde \cF^u$ on $\widetilde M$ project
to one-dimensional (possibly singular) foliations on $\cQ_\phi$, denoted
$\cQ^s,\cQ^u$, obtained by collapsing each $\widetilde\phi$-orbit inside a leaf of
$\widetilde \cF^{s/u}$ to a point (so that $p$ maps each leaf of $\widetilde \cF^{s/u}$ to a leaf of
$\cQ^{s/u}$).
\end{defi}

\begin{defi}[Perfect fit]\label{def: perfect fit leaves}
Let $l^s$ be a leaf of $\cQ^s$ and $l^u$ be a leaf of $\cQ^u$.
We say that $l^s$ and $l^u$ \emph{make a perfect fit} if
$l^s\subset \partial\big(\cQ^s(l^u)\big)$
and
$l^u\subset \partial\big(\cQ^u(l^s)\big)$.
\end{defi}

Here $\partial\big(\cQ^s(l^u)\big)$ denotes the boundary of the saturation by the foliation $\cQ^s$ of the leaf $l^u$ in $\cQ_\phi$. Equivalently, it is the boundary in the leaf space $\Lambda(\cQ^s)$.

\begin{defi}[Lozenge]\label{def: lozenge}
A \emph{lozenge} in $\cQ_\phi$ is an open set $L\subset \cQ_\phi$ of the following form.
There exist two points (the \emph{corners}) $a,b\in \cQ_\phi$ and four half-leaves
\[
l^s_a\subset \cQ^s(a),\quad l^u_a\subset \cQ^u(a),\quad
l^s_b\subset \cQ^s(b),\quad l^u_b\subset \cQ^u(b)
\]
with $l^s_a$ and $l^u_b$ forming a perfect fit, and $l^u_a$ and $l^s_b$ forming a perfect fit.
The boundary of $L$ is exactly the union of these four half-leaves
and $L$ is the component of $\cQ_\phi\setminus \partial L$ whose closure contains both $a$ and $b$.
The two half-leaves $l^s_a,l^s_b$ are the \emph{stable sides} of $L$, and
$l^u_a,l^u_b$ are the \emph{unstable sides}.
\end{defi}

\begin{defi}[Chain of lozenges]\label{def: chain lozenges}
A \emph{chain of lozenges} in $\cQ_\phi$ is a union of closed lozenges that satisfies the following connectedness property: for any two lozenges $L$, $L'$ in the chain, there exist lozenges $L_0, L_1, \ldots L_k$ of the chain such that $L = L_0, L'= L_k$ and for all $i$ the lozenges $L_i$ and $L_{i+1}$ share a corner.
\end{defi}

\subsection{The geometric JSJ decomposition}

In this article we will often be working with the Jaco-Shalen-Johannson (JSJ) decomposition: 

\begin{thm}[JSJ decomposition, \cite{johannsonHomotopyEquivalences3manifolds1979}, \cite{JacoSeifertFiberedSpace1979}] \label{thm: JSJ}
Let $M$ be a closed, orientable, irreducible $3$-manifold.
There exists a finite family $\cT$ of disjoint incompressible tori and Klein bottles
such that every component of $M\setminus \cT$ is either Seifert fibered or atoroidal.
A minimal such family is unique up to isotopy.
\end{thm}

The closures of the connected components of $M\setminus \cT$ are called the \emph{JSJ pieces}.

We will need the fact that there are only finitely many ways one can reconstruct a manifold $M$ from its decomposition into JSJ pieces. While this result is well-known among 3-manifold topologists, we could not find a reference with that precise statement, so we provide a proof.

\begin{prop}[Finiteness of gluing along a JSJ torus]\label{prop: finiteness jsj}
Let $T$ be a JSJ torus embedded in an orientable, irreducible $3$-manifold $M$, and let $P$ be the closure of $M \setminus T$. 
Then there exist only finitely many isotopy classes of involutions
$f \colon \partial P \rightarrow \partial P$ such that the quotient $P / f$ is diffeomorphic to $M$.
\end{prop}

\begin{proof}
Assume by contradiction that there exist infinitely many pairwise non-isotopic gluing maps 
$f_i \colon \partial P \to \partial P$ such that 
$M_i := P / f_i \simeq M$
for all $i$.  

Fix a reference gluing $f_0$ producing $M_0 \simeq M$, and let $H_i \colon M_i \to M_0$ be diffeomorphisms. 
Denote by $q_i \colon P \to M_i$ the quotient maps, and by $T_i = q_i(\partial P)$ the images of the boundary tori.
Each $H_i$ carries the JSJ decomposition of $M_i$ to that of $M_0$.
Up to composing $H_i$ with a diffeomorphism of $M_0$, we may assume that
$H_i(T_i) = T_0 = T$.  
Cutting $M_i$ and $M_0$ along $T_i$ and $T_0$ respectively, the restriction of $H_i$ induces a diffeomorphism
$h_i \colon P \to P$
which satisfies the boundary relation
\begin{equation}\label{eq:conjugacy}
    h_i \circ f_i = f_0 \circ h_i \quad\text{on } \partial P.
\end{equation}

Let $T^-$ be one boundary component of $\partial P$, and let $T^+ = f_i(T^-)$ be its paired component.  
Denote by $Q^\pm$ the JSJ pieces of $P$ adjacent to $T^\pm$, and set
\[
h_i^\pm := h_i|_{Q^\pm} \colon Q^\pm \to Q^\pm.
\]

For a compact $3$-manifold $Q$ with incompressible torus boundary, denote by $\mathrm{MCG}(Q;\partial)$ the mapping class group fixing each boundary component, and for a boundary component $T\subset\partial Q$ define $E(T)$ the image of this group inside $\mathrm{MCG}(T) \subset \mathrm{SL}(2,\Z)$.

Suppose first that both $Q^+$ and $Q^-$ are hyperbolic.
For a compact hyperbolic $3$-manifold with incompressible torus boundary, the mapping class group $\mathrm{MCG}(Q;\partial)$ is finite by Mostow--Prasad rigidity.  
Consequently the image $E(T)$ is finite.  
Since \eqref{eq:conjugacy} implies
$f_i = (h_i^+)^{-1} \circ f_0 \circ h_i^-$,
and there are only finitely many isotopy classes of possible $h_i^\pm|_{\partial Q^\pm}$, it follows that the set of isotopy classes of $f_i$ is finite.
This proves the claim in the hyperbolic case.

Suppose now that one piece is Seifert fibered with a unique Seifert fibration.
Let $Q=Q^\pm$ be such a Seifert piece, and let $T = T^\pm$ be the boundary torus.  
Choose a basis $(k,s)$ of $H_1(T)\cong\mathbb Z^2$, where $k$ is the oriented fiber slope and $s$ a transverse section slope.  
Any boundary diffeomorphism $f$ extending over $Q$ must preserve the fiber direction, hence is a $n$-Dehn twist along $k$ (possibly reversing orientation).  
If $n\ne 0$, such a twist extends uniquely to the interior of $Q$ as a fiberwise Dehn twist supported in a neighborhood $N\simeq A\times I$ of a vertical annulus $A$ connecting two boundary components $T$ and $T'$, and is the identity outside $N$.
Now let $T'$ be the other boundary component of $A$, adjacent to another JSJ piece $Q'$.  
If $Q'$ is hyperbolic, then $\mathrm{MCG}(Q';\partial)$ is finite, hence the slope of the twist that can be realized on $T'$ is determined by finitely many possibilities for $n$.  
It follows that only finitely many $n$ can occur on~$T$ as well.
If $Q'$ is Seifert fibered with a unique fibration, we may choose the bases $(k,s)$ and $(k',s')$ on $T$ and $T'$ so that $k'$ corresponds to the fiber of $Q$ and $s'$ to the fiber of $Q'$.  
Then $f$ must preserve both $k'$ and $s'$, forcing $n=0$.

In all cases, there are finitely many isotopy classes for the restrictions $h_i^\pm|_{\partial Q^\pm}$, hence for the maps $f_i$.

If the Seifert fibration is not unique,
then by \cite[Lemma~VI.20]{jacoLecturesThreemanifoldTopology1997} the piece $Q$ is homeomorphic to one of the following:
\[
\D^2\times S^1, \quad \T^2\times I,\quad \text{or}\quad K \widetilde{\times} I,
\]
where $K \widetilde{\times} I$ is the twisted $I$-bundle over the Klein bottle.
Since $Q$ is a piece of the torus decomposition of an irreducible manifold, $Q\neq \D^2\times S^1$ because its boundary would be compressible.  
If $Q=\T^2\times I$, then it would lie inside the peripheral region of an adjacent Seifert piece, contradicting the maximality of the JSJ decomposition, or the whole manifold $M$ is a mapping torus, in which case the proposition is true.
Hence the only remaining possibility is that $Q$ is a twisted $I$-bundle over the Klein bottle.  
In this case, $Q$ admits exactly two non-isotopic Seifert fibrations, and the above arguments give the finiteness as in the previous cases.
\end{proof}

\subsection{The modified JSJ decomposition}

Here we recall the modified JSJ decomposition of Barbot and Fenley, and refer to \cite[Chapter 6]{barthelmePseudoAnosovFlowsPlane2025} for more details.

\begin{defi}[Quasi-transverse surface] \label{def: qt surface}
A closed surface $S$ in $M$ is said to be \emph{quasi-transverse} for $\phi$
if it is immersed, incompressible and:
\begin{enumerate}
\item There are finitely many (possibly zero) periodic orbits
$O_1,\ldots,O_n$ of $\phi$ contained in $S$;
\item The complement $S \setminus \{O_1,\ldots,O_n\}$ of the periodic orbits is
transverse to $\phi$.
\end{enumerate}
\end{defi}

\begin{defi}[Weakly embedded] \label{def: weakly embedded quasi transverse surface}
A quasi-transverse surface $S$ is \emph{weakly embedded} if the complement
$S \setminus \{O_1,\ldots,O_n\}$ of periodic orbits
is embedded.
\end{defi}

In this article, we will only consider quasi-transverse tori.
Let $T$ be a quasi-transverse torus and let
$O_1,\ldots,O_n$ be its periodic orbits.
Let $A_i$ denote the open annuli bounded by $O_{i-1}$ and $O_i$ (with the convention that $O_{n+1}=O_1$).
The interior of $A_i$ is transverse to the flow, hence the foliation $\cF^s$ and $\cF^u$ induce 1-dimensional foliation $f^s_i$ and $f^u_i$ on each annulus $A_i$.
Both foliations admit $O_{i-1}$ and $O_i$ as closed leaves.
They extend to a pair of 1-dimensional foliations $(\cF_T^s, \cF_T^u)$ on the quasi-transverse torus $T$.

\begin{defi}[Normal forms of quasi-transverse tori]\label{def: qt normal forms}
A quasi-transverse torus $T$ is:
\begin{enumerate}
    \item \emph{alternating} if $n$ is even and the transverse orientation of $\phi$
    alternates from $A_i$ to $A_{i+1}$;
    \item \emph{maximally transverse} if $n=0$ or if $A_i$ and $A_{i+1}$ lie in
    non-adjacent quadrants near $O_i$;
    \item \emph{maximally periodic} if the only closed leaves of the induced foliations on $T$
    are the curves $O_1,\ldots,O_n$.
\end{enumerate}
\end{defi}

Maximally transverse tori minimize the number of periodic orbits they contain,
while maximally periodic ones maximize it.
Any alternating quasi-transverse surface is automatically maximally transverse, and the converse also holds except sometimes when some of the $O_i$ are singular orbits.

Given a quasi-transverse torus $T$, the lift of $T$ to the universal cover
projects to a chain of lozenges in the orbit space $\cQ_\phi$ (see Section \ref{sec: preli; subsec: pa flows}).

The following result summarizes this correspondence.

\begin{prop}[Trace of a quasi-transverse torus, {\cite[Proposition 5.3.5]{barthelmePseudoAnosovFlowsPlane2025}}]
Let $T$ be a quasi-transverse torus or Klein bottle, $\widetilde{T}$ a lift of $T$ to $\widetilde{M}$,
and $\widehat{T}$ its projection to the orbit space $\cQ_\phi$.
Then the closure of $\widehat{T}$ is a chain of lozenges $C$ with the property that consecutive lozenges share sides, and no three lozenges share a corner.
Furthermore, for any two ``diagonal'' lozenges of $C$ sharing a corner but not a side,
the shared corner is contained in $\widehat{T}$, and:
\begin{enumerate}
    \item $T$ is maximally transverse if and only if the only corners in $\widehat{T}$ are
    shared by diagonal lozenges;
    \item $T$ is alternating if and only if $T$ is maximally transverse and, for each
    corner $c$ contained in $\widehat{T}$, there are an odd number of quadrants separating
    the two lozenges sharing $c$;
    \item $T$ is maximally periodic if and only if each corner of $C$ is in $\widehat{T}$.
\end{enumerate}
\end{prop}

Barbot and Fenley show that the embedded JSJ tori can be homotoped into quasi-transverse
position relative to a pseudo-Anosov flow,
at the cost of allowing intersections along periodic orbits only.

The following result was proved by Barbot and Fenley in \cite{barbotPseudoAnosovFlowsToroidal2013a}, see also \cite[Theorem 6.3.14]{barthelmePseudoAnosovFlowsPlane2025}.

\begin{thm}[Barbot-Fenley modified JSJ decomposition] \label{thm: modified JSJ}
Let $\phi$ be a pseudo-Anosov flow on a $3$-manifold $M$,
and let $J=\{T_1,\dots,T_n\}$ be the (geometric) JSJ tori and Klein bottles.
There exists a finite family
\[
J_\phi=\{T_1',\ldots,T_n'\}
\]
of weakly embedded quasi-transverse surfaces such that:
\begin{enumerate}
    \item each $T_i'$ is homotopic to $T_i$;
    \item the union $\bigcup_i T_i'$ is embedded except possibly along periodic orbits contained in the $T_i'$;
    \item each connected component of $M\setminus J_\phi$ is either a JSJ piece,
    or has arbitrarily small neighborhoods that are JSJ pieces.
\end{enumerate}
\end{thm}

The family $J_\phi$ is called the \emph{modified JSJ decomposition} for the flow $\phi$.

\subsection{Periodic Seifert pieces and spines}\label{subsec: spine}
Barbot and Fenley, in their work on pseudo-Anosov flows in Seifert pieces \cite{barbotPseudoAnosovFlowsToroidal2013a,barbotClassificationRigidityTotally2015,barbotFreeSeifertPieces2021}, introduced the following terminology.

\begin{defi} \label{def_periodic_Seifert}
Let $P$ be a Seifert fibered JSJ piece of $M$.  The piece $P$ is called {\em periodic} for a pseudo-Anosov flow $\phi$ if a regular fiber of some Seifert fibration of $P$ is freely homotopic to a periodic orbit of $\phi$, and is called a {\em free piece} otherwise.  
\end{defi} 

They then describe the dynamics of the flow restricted to periodic pieces in \cite{barbotPseudoAnosovFlowsToroidal2013a}, and to free pieces in \cite{barbotFreeSeifertPieces2021}. It turns out that in periodic pieces, the dynamics of the flow is completely characterized by the \emph{spine}:

\begin{thm}[\cite{barbotPseudoAnosovFlowsToroidal2013a} Theorem B, and \cite{BFM25} Theorem 3.1] \label{thm_spines_exist}
Let $P$ be a periodic Seifert fibered piece for a pseudo-Anosov flow on $M$. Then there exists a connected, finite union $Z$ of elementary Birkhoff annuli in $M$, embedded along the complement of their periodic orbit boundaries,  such that any sufficiently small neighborhood of this union is a representative for the JSJ piece $P$. 
The union $Z$ is called the \emph{spine} of the piece. 
\end{thm} 

One can always isotope the spine $Z$ in a periodic Seifert piece $P$ so that it is foliated by Seifert fibers. 
When $P$ is topologically a $\Sigma\times S^1$, with $\Sigma$ a compact orientable surface (with boundary), then the projection of the spine $Z$ to the base surface $\Sigma$ gives a \emph{fatgraph} on $\Sigma$ (i.e., a graph $\cG$ embedded in $\Sigma$ such that $\Sigma$ deformation retracts to $\cG$).

In that case, the spine, and hence the flow in the piece, is determined by the fatgraph up to diffeomorphisms, together with a choice of orientation for one of the periodic orbits in the piece. More precisely, we have the following result, essentially due to Barbot and Fenley \cite{barbotPseudoAnosovFlowsToroidal2013a, barbotClassificationRigidityTotally2015}:

\begin{thm}\label{thm_same_spine_implies_same_flow}
    Let $\phi_1, \phi_2$ be two pseudo-Anosov flows on an orientable manifold $M$. Let $P$ be a Seifert piece which is periodic for both flows, and call $P_i$, $i=1,2$, the associated modified pieces for $\phi_i$.
    Assume that $P$ is homeomorphic to $\Sigma \times S^1$ with $\Sigma$ an orientable surface, and that the tori in $\partial P_i$ are all embedded.
 Suppose that the fatgraphs $\cG_1,\cG_2\subset \Sigma$ associated to $\phi_1$ and $\phi_2$ respectively are homeomorphic by a map $h\colon \Sigma \to \Sigma$ and there exists a vertex $v_0\in \cG_1$ such that the orientation in $P$ of the periodic orbit of $\phi_1$ corresponding to $v_0$ is the same as the orientation of the periodic orbit of $\phi_2$ corresponding to $h(v_0)\in \cG_2$.
 Then the restrictions $\phi_1|_{P_1}$ and $\phi_2|_{P_2}$ are orbit equivalent. Furthermore, if $h\colon \Sigma \to \Sigma$ is isotopic to the identity, then $\phi_1|_{P_1}$ and $\phi_2|_{P_2}$ are isotopically equivalent\footnote{By isotopically equivalent, we mean that the two flows are orbit equivalent via a map that is isotopic to the identity.}.
\end{thm}

\begin{proof}
    This result follows from the fact that, from the data of the fatgraph together with one choice of orientation, one can build a \emph{model flow} by gluing together copies of a standard neighborhood of an elementary Birkhoff annulus (see \cite[Section 4.1]{BFM25} or \cite[Section 6.5.1]{barthelmePseudoAnosovFlowsPlane2025}), according to the combinatorics of the fatgraph. 
    By construction that model flow will be isotopically equivalent on the piece to the original flow. This construction was first done in \cite{barbotPseudoAnosovFlowsToroidal2013a}, then extended in \cite[Proposition 4.7]{BFM25}, without any assumption about the topological type of $P$ but under the assumption that all the boundary surfaces in $\partial P_i$ are \emph{transverse} to the flow.

    This same construction can be easily extended to the current setting thanks to our assumption that $P\simeq \Sigma \times S^1$, which prevents the existence of singular Seifert fibers (so that one does not have to do Dehn surgeries to recover the piece $P$) and having to deal with non-orientability considerations (so that the only two types of gluing one needs are gluing model blocks along half-faces or just along a periodic orbit). See the discussion after Theorem 6.5.6 in \cite{barthelmePseudoAnosovFlowsPlane2025}.
\end{proof}

The following lemma on self-orbit equivalences of periodic Seifert pieces will be used later.

\begin{lem}
\label{lem: vertical annular twists orbit equivalences}
Let $(Q,\varphi)$ be a periodic Seifert pseudo-Anosov piece homeomorphic to $\Sigma\times S^1$ with $\Sigma$ orientable.
Let $A\subset Q$ be a vertical annulus in $Q$ joining two boundary curves in $\partial Q$.
Then a
three-dimensional Dehn twist along $A$ in the direction of the Seifert
fiber can be represented by a self-orbit equivalence of $(Q,\varphi)$.
\end{lem}

\begin{proof}
Call $D\colon \Sigma\times S^1 \to \Sigma\times S^1$ a Dehn twist along the vertical annulus $A$. Note that the projection of $D$ onto $\Sigma$ is the identity. Therefore, calling $\varphi'= D \circ \varphi \circ D^{-1}$, the fatgraph associated to the spine of $\varphi'$ is the same as that of $\varphi$. Moreover, the direction of the periodic orbits of $\varphi'$ are also unchanged compared to those of $\varphi$, so Theorem \ref{thm_same_spine_implies_same_flow} implies that $\varphi'$ and $\varphi$ are isotopically equivalent, so $D$ is isotopic to a self-orbit equivalence of $\varphi$.\qedhere

\end{proof}

\subsection{Finite cover with embedded modified JSJ tori}

For the technical details of the proof of Theorem \ref{thmintro: finiteness orbit equivalence}, it will be important for us to work with embedded quasi-transverse tori, not just weakly embedded, as well as ``good'' Seifert pieces. In this section, we show that this is achievable in a finite cover:

\begin{prop}\label{prop: cover embedded jsj tori}
   There exists an orientable finite cover $\hat M$ of $M$ such that all Seifert pieces are the product of an orientable surface with $S^1$, and for all pseudo-Anosov flows $\phi$ on $M$, the modified JSJ tori of the lift $\hat \phi$ on $\hat M$ are embedded.
\end{prop}
Note that the degree of the finite cover $\hat M \to M$ only depends on $M$: it is the same for each pseudo-Anosov flow $\phi$ on $M$.

This result contains two parts: one is a purely topological classical fact that one can always take a finite cover so that the Seifert pieces are as nice as one wants, and the second is that up to taking an additional cover, we can ensure that all the lifted flows have embedded quasi-transverse JSJ tori. For that second part, we will need some preliminary work on fatgraphs.

\begin{lem}\label{lem:finite_number_fat_graph}
Let $\Sigma$ be a compact orientable surface.  
Then there are only finitely many fatgraphs $\cG\subset\Sigma$, up to diffeomorphism of~$\Sigma$, having no vertex of valence~$2$.
\end{lem}

\begin{proof}
Since $\cG$ is a deformation retract of $\Sigma$, their Euler characteristics coincide:
\[
\chi(\Sigma)=\chi(\cG)=V-E,
\]
where $V$ and $E$ denote the numbers of vertices and edges of $\cG$.
By the Degree Sum Formula,
\[
\sum_{v\in\mathrm{Vert}(\cG)} (\deg(v)-2)
  = 2E-2V = -2\chi(\Sigma).
\]
Hence the sum of the quantities $\deg(v)-2$ over all vertices is a fixed integer depending only on $\Sigma$.
In particular, the number of vertices with valence $\ge3$ and the valence of each vertex are uniformly bounded in terms of $\chi(\Sigma)$.
Since $E=V-\chi(\Sigma)$, the number of edges is also uniformly bounded.
Consequently, there exist only finitely many combinatorial types of such graphs, and therefore finitely many fatgraphs in~$\Sigma$ up to diffeomorphism.
\end{proof}

\begin{lem}\label{lem:finite_cover_for_all_graph}
Let $\Sigma$ be an orientable, compact surface with non-empty boundary.
Then there exists a finite covering $\hat\Sigma \to \Sigma$ such that, for every fatgraph $\cG \subset \Sigma$, 
each boundary component $\hat c$ of $\hat\Sigma$ retracts onto a simple closed curve contained in the lift $\hat\cG \subset \hat\Sigma$ of $\cG$.
\end{lem}

\begin{proof}

We identify $\pi_1(\Sigma)$ with a free group $F_n$.
Let $\cG \subset \Sigma$ be a fatgraph.

\smallskip
\noindent
\textbf{Step 1.}
We will call a \emph{local sector} of a vertex $v$ in $\cG$ a connected component of $D\smallsetminus\cG$, where $D$ is a small disk on $\Sigma$ centered at $v$. A vertex of valence $p$ has $p$ local sectors.
A boundary component $c$ of $\Sigma$ retracts to a curve $\gamma$ in $\cG$ which is \emph{not} simple 
if and only if there exists a vertex $v \in \cG$ such that two distinct local sectors of $v$ intersect $c$. 

We claim that this vertex is never of valence 2: By contradiction, choose a small closed disk $D$ centered at $v$ such that $D \cap \cG$ is the image
of a single embedded segment $e \cap D$.
Consider a small embedded arc
$a \colon [0,1] \longrightarrow \Sigma$, with
$a(0) \in c$, $a((0,1]) \subset \mathrm{int}(\Sigma)$,
$a$ is transverse to $e$ and meets $e$ exactly once, in a point $p \in e \cap D$, and
$a$ does not meet $\cG$ elsewhere.
We orient $a$ from $c$ into the interior of~$\Sigma$.
By the choice of the boundary orientation on~$c$, the algebraic intersection number
$\langle c,a \rangle$ is equal to $+1$.
Denote $\gamma$ the retract of $c$ to $\cG$. We can assume that the homotopy retract fixes $a$.
We have
$\langle \gamma, a \rangle = \langle c, a \rangle = +1.$
By contradiction, suppose that $c$ meets both sides of $e$ near $v$,
then the retracted curve $\gamma$ travels along the edge $e$ in both orientations: once on each side.
These two contributions to the algebraic intersection with $a$ cancel each other, so that
$\langle \gamma, a \rangle = 0$. Contradiction.

In particular, we can replace the fatgraph $\cG$ by the fatgraph with no vertex of valence 2 by collapsing such a vertex to an edge.
We abuse notation and still denote this fatgraph $\cG$.
Choose a basepoint $x_0$ corresponding to one such vertex $v$.
We have
\[
[\gamma] = \beta_1 \beta_2 \dots \beta_k
\]
in generators $\beta_i$ represented by simple closed curves in $\cG$.

\smallskip
\noindent
\textbf{Step 2.}
We apply a version of Marshall Hall’s Theorem for free groups given by Stallings \cite[Cor.~6.3]{StallingsTopologyFiniteGraphs1983}.
Let $S = \langle c \rangle < \pi_1(\Sigma)$ and define the finite set of \emph{forbidden} words $F$ to be all reduced words 
obtained by concatenating distinct letters $\beta_i$ except for the full word $c$ itself, i.e.
\[
F = \bigl\{ \beta_{i_1} \dots \beta_{i_l} \ \big|\ 
(i_1, \dots, i_l) \neq (1, \dots, k),\ 
i_j \in \{1, \dots, k\},\ i_j \neq i_{j'} \text{ if } j \neq j' \bigr\}.
\]
This set $F$ is finite.  
By Hall’s theorem, there exists a finite–index subgroup $S' < \pi_1(\Sigma)$ that contains $S$ and excludes all forbidden words $F$. Up to replacing $S'$ by its normal core, we may further assume that $S'$ is normal in $\pi_1(\Sigma)$.
Let $\hat\Sigma \to \Sigma$ be the finite covering corresponding to $S'$, of degree $n$.
Since $\langle c \rangle \subset S'$, each boundary component $c$ of $\Sigma$ has exactly $n$ lifts 
in $\hat\Sigma$ that project homeomorphically to $c$.

\smallskip
\noindent
\textbf{Step 3.}
We claim that in a neighborhood of each vertex $v$ of $\cG$, any lift $\hat c$ of $c$ intersects at most one local sector of a lift of $v$.

Assume by contradiction that some lift $\hat c$ intersects two distinct sectors $\hat Q_0$ and $\hat Q_1$ 
around the same lift $\hat v$ of $v$ in the same sheet of $\hat\Sigma$.  
Let $U$ be a small neighborhood of $v$, and let $H$ be a homotopy between $c$ and its retraction $\gamma$ inside $\cG$.

Let $t_0$ and $t_1$ be time parameters such that $\hat c(t_0) \in \hat Q_0$, $\hat c(t_1) \in \hat Q_1$ and they retract to the same point, i.e. $H(1,t_0) = H(1, t_1)$.
Then the segment of $c$ between $t_0,t_1$ maps, under the retraction, 
to a closed (not necessarily simple) curve $\gamma_1$ in $\cG$ representing a reduced word $\beta \in F$.
Lifting the homotopy $H$ to $\hat H$ with $\hat H(0,\cdot) = \hat c$ gives a lifted loop $\hat\gamma_1 = \hat H(1,\cdot)$ 
contained in $\hat\cG$ and projecting homeomorphically to $\gamma_1$.
This would imply that the forbidden word $\beta \in F$ belongs to $S'$, a contradiction.

Hence, each lift $\hat c_i$ meets distinct quadrants of the lifted vertices in distinct sheets, 
so each $\hat c_i$ retracts onto a \emph{simple} closed curve contained in $\hat\cG$.

\smallskip
\noindent
\textbf{Step 4.}
Repeating this process for all the (finitely many) boundary components whose retractions are not simple, 
and replacing $\Sigma$ by the successive finite covers $\hat\Sigma$, we eventually obtain a finite cover 
where every boundary component retracts onto a simple closed curve in the lift of any fatgraph $\cG$.
Since there are only finitely many fatgraphs by Lemma \ref{lem:finite_number_fat_graph}, we can also iterate this process for each fatgraph of $\Sigma$ with no vertex of valence 2 (by collapsing such a vertex to an edge.
Indeed, taking successive finite covers does not introduce new self–intersections:
Otherwise, a boundary component in the cover would meet two distinct sectors above a single vertex 
without such an intersection occurring in the base surface, which is impossible.
This concludes the proof.
\end{proof}

\begin{proof}[Proof of Proposition \ref{prop: cover embedded jsj tori}]
First, up to taking a double cover, we assume that $M$ is orientable.
Since $M$ is closed and irreducible, it admits a JSJ decomposition into finitely many Seifert fibered pieces and atoroidal pieces separated by incompressible embedded tori. 

Next, we claim that we can take a finite cover of $M$ such that each Seifert piece in the JSJ decomposition of $M$ is topologically an orientable surface times $S^1$: If $P$ is any Seifert piece in the JSJ decomposition, we can take a cover $\hat P$ of $P$ such that $\hat P = \Sigma \times S^1$ with $\Sigma$ an orientable surface. Once we choose such covers $\hat P_i$ for each of the Seifert pieces $P_i$, we can apply Lemma 9.2 of  \cite{FriedlDistinctFiniteCovers2021}, to deduce that there exists a finite cover $\hat M$ of $M$ such that all its Seifert pieces are finite covers of $\hat P_i$.

So from now on, we assume that each Seifert piece in $M$ is a product of an orientable surface with $S^1$. Our goal will now be to show that we may take a further finite cover $\hat M$ of $M$ such that any pseudo-Anosov flow $\phi$ on $M$ lifts to a flow $\hat \phi$ on $\hat M$ such that all the JSJ tori are isotopic to \emph{embedded} quasi-transverse tori.

Let $P$ be a Seifert piece of $M$, and let $p \colon P \longrightarrow \Sigma$
be the corresponding circle fibration over an orientable compact surface $\Sigma$ with boundary.

Let $T$ be the union of JSJ tori corresponding to the boundary of $P$.
We now explain how Lemma~\ref{lem:finite_cover_for_all_graph} can be used to produce a finite cover of $P$ in which all modified JSJ tori bounding periodic Seifert pieces are embedded, uniformly for all pseudo-Anosov flows.

\smallskip
\noindent
\textbf{Step 1.}
Let $\phi$ be a pseudo-Anosov flow on $M$ such that $P$ is a periodic Seifert piece for $\phi$.
Denote by $T_\phi$ the corresponding collection of modified JSJ tori, weakly embedded, quasi-transverse to $\phi$, and homotopic to~$T$.
Let $P_\phi$ be the corresponding modified Seifert piece bounded by $T_\phi$.

Let $Z_\phi \subset P$ be the spine for $\phi$ and let $\cG_\phi \subset \Sigma$ be its associated fatgraph (see Section \ref{subsec: spine}).

The images of the JSJ tori $T$ under the projection $p\colon P \to \Sigma$ are boundary components of $\Sigma$; let $c \subset \partial\Sigma$ be one such boundary component corresponding to a JSJ torus $T_0 \subset T$.

The modified surface $\Sigma_\phi := p(P_\phi)$ is obtained from $\Sigma$ by pushing the boundary along the projection of the weakly embedded tori $T_\phi$; it is homotopy equivalent to $\Sigma$, but may have singularities when the tori in $T_\phi$ are only weakly embedded.
The image of $\cG_\phi$ in $\Sigma_\phi$ is again a finite graph, and each boundary component $c_\phi \subset \partial\Sigma_\phi$ corresponds to the projection of a modified JSJ torus in $T_\phi$ quasi-transverse to $\phi$.
A modified JSJ torus corresponding to $c_\phi$ fails to be embedded if and only if the boundary component $c$ retracts, in $\Sigma$, onto a non-simple closed curve in the fatgraph $\cG_\phi$.
Equivalently: self-intersections of the modified torus correspond precisely to corners where $c$ meets two distinct local sectors at the same vertex of $\cG_\phi$.

\smallskip
\noindent
\textbf{Step 2.}
Apply Lemma~\ref{lem:finite_cover_for_all_graph} to the surface~$\Sigma$.
It provides a finite covering $\hat\Sigma \to \Sigma$
with the following property: for every fatgraph $\cG \subset \Sigma$, and in particular for every graph $\cG_\phi$ arising from a spine as above, each boundary component $\hat c$ of $\hat\Sigma$ lying over a boundary component $c\subset\partial\Sigma$ retracts onto a simple closed curve contained in the lift $\hat\cG$ of $\cG$.

Consider now the pullback Seifert fibration $\hat P \to \hat\Sigma$
obtained as the associated finite cover of $P$.
The preimage $\hat T$ of $T$ in $\partial\hat P$ is a finite union of tori covering those of~$T$.
Given any pseudo-Anosov flow $\phi$ on $M$, if $\hat M$ is a finite cover of $M$ that contains (any cover of) $\hat P$ as a JSJ piece, then the lift $\hat \phi$ of $\phi$ to $\hat M$ is such that its spine $\hat Z_\phi \subset \hat P$ projects to a fatgraph $\hat \cG_\phi \subset \hat\Sigma$ that is the lift of $\cG_\phi$.
By the property of $\hat\Sigma$, each boundary component $\hat c\subset\partial\hat\Sigma$ lying over $c$ retracts onto a simple closed curve in $\hat \cG_\phi$.
Therefore, the corresponding lifted modified JSJ tori in $\hat P$ are embedded.

In other words, after passing to the finite cover $\hat P$ of each Seifert piece $P$, all modified JSJ tori bounding periodic Seifert pieces for any pseudo-Anosov flow become embedded in that piece.

\smallskip
\noindent
\textbf{Step 3.}
There are only finitely many Seifert pieces in the JSJ decomposition of~$M$.  
For each such piece $P_j$ we have constructed a finite cover $\hat P_j \to P_j$ with the property that, for any pseudo-Anosov flow on $M$, all modified JSJ tori bounding periodic Seifert pieces are embedded in the finite cover $\hat P_j$. Using again \cite[Lemma 9.2]{FriedlDistinctFiniteCovers2021}, we can build a finite cover $\hat M$ of $M$ such that all of its Seifert pieces are finite covers of $\hat P_j$. By the argument of Step 2., we deduce that all the quasi-transverse JSJ tori bounding the periodic Seifert pieces of $\hat M$ are embedded.

By \cite[Theorem 6.10]{barbotPseudoAnosovFlowsToroidal2013a}, the only modified JSJ tori that may fail to be truly embedded are those bounding a periodic Seifert piece, so we conclude that all the lifts $\hat \phi$ on the finite cover $\hat M$ have embedded quasi-transverse JSJ tori, as claimed.\qedhere

\end{proof}

\section{Foliations on quasi-transverse tori}\label{sec: foliation QT tori}

\subsection{Generalities on quasi-Morse–Smale laminations}

We start by introducing {quasi–Morse–Smale laminations}.  
This will describe the foliations induced by the stable and unstable foliations of a pseudo–Anosov flow on a quasi–transverse torus.  
We will present a way to encode the topological information needed to reconstruct this pair of foliations.
These quasi–Morse–Smale laminations were first introduced in \cite[Definition~1.9]{pauletAnosovFlowsDimension2025}, but since we deal with pseudo–Anosov flows in the present article, we need to slightly adapt the definition.

\begin{convention}\label{conv: orientable}
Throughout this section, all manifolds and surfaces are assumed to be orientable.  
By passing to a finite cover if necessary, we may also assume that there are no incompressible Klein bottles.
\end{convention}

\begin{defi}[Quasi–Morse–Smale lamination]\label{def: qms lam}
A quasi–Morse–Smale lamination is the data
\[
(\cK,\Gamma_*,S^\iin,S^\out,\chi)
\]
where:
\begin{enumerate}[label=(\arabic*)]
    \item $\cK$ is a $1$–dimensional lamination on a closed orientable surface $S$, such that:
    \begin{itemize}
        \item[--] \label{def: qms lam; it: finite compact}
        there is a finite set of compact leaves 
        \(\Gamma = \{\gamma_1,\dots,\gamma_N\}\subset\cK\);
        \item[--] \label{def: qms lam; it: half non-compact}
        each non–compact half–leaf accumulates on a single compact leaf (with contracting or expanding holonomy).
    \end{itemize}
    \item \label{def: qms lam; it: marked}
    $\Gamma_*\subset \Gamma$ is a subset containing all compact leaves whose holonomy\footnote{The holonomy of a compact leaf $c$ is the first return map of the lamination on a sufficiently small transversal $I$ to $\cK$ intersecting $c$. We say that it is \emph{contracting} (resp.\ \emph{expanding}) if the fixed point $I\cap c$ is topologically attracting (resp.\ repelling). This is not defined for isolated compact leaves.} is contracting on one side and expanding on the other, and disjoint from compact leaves having contracting (or expanding) holonomy on both sides (for a chosen orientation).  
    We call $\Gamma_*$ the set of \emph{marked compact leaves}.
    \item \label{def: qms lam; it: in out}
    $S^\iin \sqcup S^\out$ is a partition of $S\ssm\Gamma_*$ into two disjoint open sets such that each leaf $\gamma_*\in \Gamma_*$ is adjacent to $S^\iin$ on one side and to $S^\out$ on the other.  
    We call $(S^\iin,S^\out)$ an \emph{(in,out)–splitting} of $S$ for $\cK$.
    \item \label{def: qms lam; it: dynamical orientation}
    $\chi$ is an orientation of each compact leaf in $\Gamma$ such that the holonomy (when defined) is expanding on the $S^\iin$–side and contracting on the $S^\out$–side.  
    We call $\chi$ a \emph{dynamical orientation} of $\Gamma$.
\end{enumerate}
\end{defi}

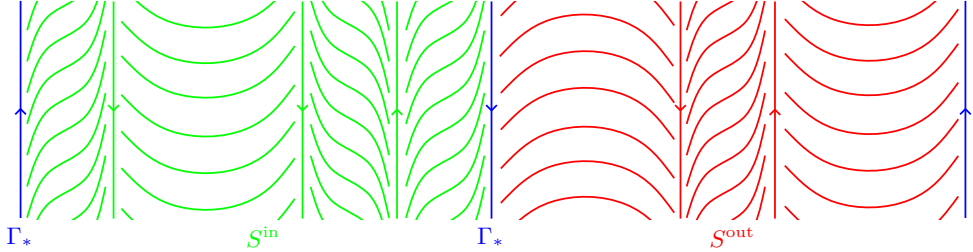
\begin{figure}[h]

    \labellist
        \small
        \pinlabel \textcolor{blue}{$\Gamma_*$} [t] at 3 0
        \pinlabel \textcolor{blue}{$\Gamma_*$} [t] at 268 0
        \pinlabel \textcolor{green}{$S^\iin$} [t] at 140 0
        \pinlabel \textcolor{red}{$S^\out$} [t] at 404 0
    \endlabellist

    \centering
    \includegraphics[width=1\linewidth]{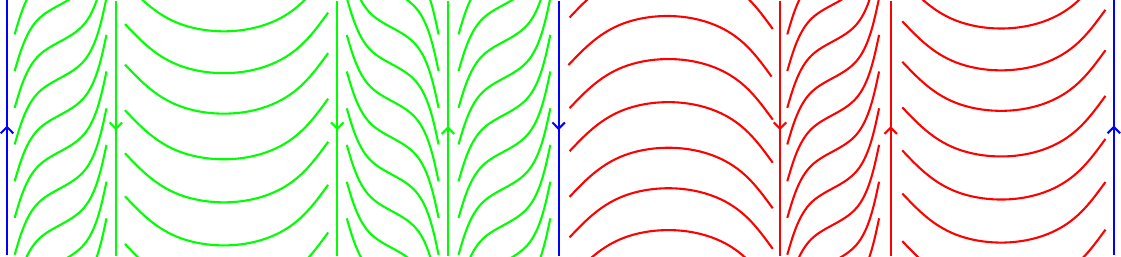}
    \vspace{2mm}
    \caption{A quasi-Morse-Smale foliation on the torus with dynamical orientation of compact leaves.}
    \label{fig: qms foliation torus}
\end{figure}

\begin{rmk}
A \qms{} lamination models the lamination obtained as the trace of the stable or unstable lamination of a piece of pseudo–Anosov flow cut along embedded quasi–transverse tori, where $\cK$ is a sublamination of the stable or unstable foliation of the pseudo–Anosov flow.

More precisely, in the case of an Anosov flow, this is made explicit in \cite[Proposition~1.11]{pauletAnosovFlowsDimension2025}, together with the fact that any piece of Anosov flow is a building block in the sense of \cite[Definition~1.2]{pauletAnosovFlowsDimension2025}.  
The lamination $\cK$ corresponds to the trace of the stable or unstable manifolds of the orbits entirely contained in the piece;  
the marked compact leaves correspond to periodic orbits tangent to the cutting surface, along which the flow switches the transverse orientation;  
the sets $S^\iin$ and $S^\out$ correspond to the regions where the flow points inside or outside the piece.

The only difference between the present definition of \qms{} lamination and \cite[Definition~1.9]{pauletAnosovFlowsDimension2025} is the definition of $\Gamma_*$.  
If $T$ is an alternating quasi-transverse torus (Definition \ref{def: qt normal forms}) for a pseudo–Anosov flow $\phi^t$, then $\Gamma$ is the set of all closed leaves of $\cK$, and $\Gamma_*$ is the subset of periodic orbits on $T$ for which the orientation of the flow changes from one side to the other.  
If $T$ does not contain any singular orbit, then $\Gamma_*=\Gamma$; in general, however, $\Gamma\smallsetminus \Gamma_*$ may still contain periodic orbits on $T$ (necessarily singular) for which the flow crosses $T$ in the same direction on both sides.  
See Remark~\ref{rem: gamma_**} for more details.
\end{rmk}

\begin{rmk}
Given a lamination $\cK$, the additional data $(\Gamma_*,S^\iin,S^\out,\chi)$ are not determined uniquely by $\cK$ in general.

For instance, compact leaves can be isolated (on one side or on both) in the lamination, in which case the holonomy is not defined.  
Such leaves may or may not be included in the marked set $\Gamma_*$ (see Item~\ref{def: qms lam; it: marked}).  
In other words, the lamination $\cK$ alone does not determine $\Gamma_*$.

Similarly, once an (in,out)–splitting $(S^\iin,S^\out)$ is fixed, the dynamical orientation is only partially determined.  
If a compact leaf $\gamma$ is non–isolated on at least one side and lies in the boundary of $S^\iin$, then its orientation must be chosen so that the holonomy is expanding on the side contained in $S^\iin$.  
If it lies in the boundary of $S^\out$, we choose the orientation so that the holonomy is contracting on that side (Item~\ref{def: qms lam; it: dynamical orientation}).  
However, if $\gamma$ is isolated on both sides, its orientation remains free.  
Thus, in general, $\chi$ is not completely determined by the (in,out)–splitting.

This is another difference with the original definition in \cite[Definition~1.9]{pauletAnosovFlowsDimension2025}, where only \emph{filling} laminations are considered, and for which $\Gamma_*$ is determined by $\cK$ and the orientation $\chi$ is determined by $\cK$ together with the (in,out)–splitting.
\end{rmk}

\begin{defi}[Morse–Smale lamination]\label{def: ms lam}
A \qms{} lamination for which the set $\Gamma_*$ is empty is called a \emph{Morse–Smale lamination}.
\end{defi}

We will often abuse notation and write simply $\cK$ for a \qms{} lamination, keeping the marked set $\Gamma_*$, the splitting $(S^\iin,S^\out)$ and the dynamical orientation $\chi$ implicit.

\medskip

Let $S$ be a closed orientable surface and $(\cK_1,\cK_2)$ a pair of \qms{} laminations on $S$.  
Denote by $\Gamma_{\cK_i}$ the compact leaves of $\cK_i$, by $\Gamma_{*,\cK_i}$ the marked compact leaves, and by
\[
S\ssm \Gamma_{*,\cK_i} = S^\iin_i \sqcup S^\out_i
\]
the corresponding (in,out)–splitting, for $i=1,2$.

\begin{defi}[Quasi–transverse laminations]\label{def: sqt laminations}
We say that the pair $(\cK_1,\cK_2)$ is \emph{quasi–transverse} if:
\begin{enumerate}
\item Denoting $\Gamma_\mm := \Gamma_{\cK_1}\cap\Gamma_{\cK_2}$, we have that 
\[
\Gamma_{*,\cK_1}= \Gamma_{*,\cK_2} \subset \Gamma_\mm.
\]
In other words, the two laminations must coincide on their marked leaves, and may coincide on more compact leaves.

    \item they are transverse to each other on the complement $S\ssm\Gamma_\mm$;
    \item the splittings are opposite in the sense that
      $S^\iin_1 = S^\out_2.$
    
\end{enumerate}
\end{defi}

\begin{rmk}\label{rem: gamma_**}
As mentioned above, a pair of quasi–transverse \qms{} laminations is intended to model the trace of the stable and unstable foliations of a pseudo–Anosov flow on a quasi–transverse embedded torus.  
The set $\Gamma_\mm$ is the set of periodic orbits that are tangent to the torus.  
In the pseudo–Anosov setting, a periodic orbit tangent to $S$ may locally split $S$ into two disjoint transverse components lying in non–adjacent quadrants, where the flow induces the same transverse orientation on both sides; this explains why $\Gamma_\mm$ may be strictly larger than the set of marked compact leaves $\Gamma_* = \Gamma_{*, \cK_i}$.  

Note that we work with laminations rather than full foliations because we will cut pieces of pseudo–Anosov flows along quasi–transverse surfaces and then restrict to suitable sublaminations of the stable and unstable foliations on each piece.
\end{rmk}

\begin{defi}[pre-foliation]
We say that a lamination $\cK$ on a closed orientable (possibly not connected) surface $S$ is a \emph{pre-foliation} if there exists a foliation $\cF$ on $S$ such that the leaves of $\cK$ are leaves of $\cF$.  
In particular, $S$ has Euler characteristic $0$, hence $S$ is a union of tori.
\end{defi}

While a pre-foliation can be a sublamination of many distinct foliations, our next goal is to show that a \qms{} pre-foliation can be completed to a canonical \qms{} foliation, in a minimal sense and unique up to homeomorphism.

Let $\cK$ be a \qms{} pre-foliation on a torus $S$.  
Cut $S$ along a compact leaf $c$, which is always non–contractible by the Poincaré–Bendixson Theorem; we obtain an annulus.  
We know from \cite[Theorem~4.2.15]{hectorIntroductionGeometryFoliations1981} that every foliation of the annulus tangent to the boundary is obtained by gluing together a finite number of Reeb components and suspension foliations along compact leaves, possibly glued together along further compact leaves.

\begin{rmk}\label{rmk: compact leaves cut in annuli}
For any two compact leaves $\gamma$ and $\gamma'$ in $\Gamma$, they are freely homotopic as unoriented closed curves in $S$.  

The set of compact leaves $\Gamma$ of a \qms{} pre-foliation cuts the torus $S$ into a finite union of annuli with non–compact leaves in their interiors, where the pre-foliation is either a suspension without fixed point in the interior, a Reeb component, or empty.
\end{rmk}

Let $\cK$ be a \qms{} pre-foliation on a torus $S$.  
Let $\Gamma_\cK$ be its compact leaves, $\Gamma_{*,\cK}$ the marked compact leaves, $S^\iin_\cK \sqcup S^\out_\cK$ the (in,out)–splitting of $S$ for $\cK$, and $\chi_\cK$ its dynamical orientation.

\begin{lem}\label{lem: extend qms pre-foliation by qms foliation}
There exists a \qms{} foliation $\cF$ on $S$ which contains $\cK$ as a sublamination and satisfies:
\begin{enumerate}[label=(\roman*)]
    \item \label{ìtem_extend_qms_same_closed}
    $\Gamma_\cF = \Gamma_\cK$ and $\Gamma_{*,\cF} = \Gamma_{*,\cK}$;
    \item \label{item_extend_qms_same_split}
    the splittings and dynamical orientations coincide:
    \[
      S^\iin_\cF = S^\iin_\cK,\quad S^\out_\cF = S^\out_\cK,\quad \chi_\cF = \chi_\cK.
    \]
\end{enumerate}
Moreover, such $\cF$ is unique up to topological equivalence\footnote{An orientation–preserving homeomorphism of $S$ mapping one foliation to the foliation}.
\end{lem}

\begin{proof}
Cut the torus $S$ along the compact leaves $\Gamma_\cK$.  
We obtain a finite collection of annuli $A_i$, with oriented boundary components, and no compact leaves in their interiors.

Suppose first that $A_i \subset S^\iin_\cK$.  
On each boundary component of $A_i$, the induced orientation and the dynamical orientation determine whether the holonomy along this boundary is expanding or contracting on the $A_i$–side.  
By definition of $S^\iin_\cK$, we require expanding holonomy on the oriented boundary leaves.  

If the two boundary leaves of $A_i$ are coherently oriented, then the unique way (up to homeomorphism fixing the boundary) to extend $\cK$ to a foliation on $A_i$ with no compact leaves in the interior and expanding holonomy on the boundary is by a Reeb foliation.  
If the two boundary leaves are incoherently oriented, the unique such extension is a suspension foliation with no fixed point in the interior.

Similarly, if $A_i \subset S^\out_\cK$, we build the foliation on $A_i$ exactly as above, but now requiring \emph{contracting} holonomy on the $A_i$–side of the oriented boundary leaves.  

Gluing all the annuli $A_i$ back along their boundary circles, we obtain a foliation $\cF$ on $S$ that contains $\cK$, has the same compact leaves and marked set, and realizes the prescribed splitting and dynamical orientation; hence $\cF$ is a \qms{} foliation satisfying (i) and (ii).

For uniqueness, let $\cF'$ be another \qms{} foliation satisfying (i) and (ii).  
Cutting along $\Gamma_{\cF'}=\Gamma_\cK$ again yields the same collection of annuli $A_i$, and condition~(ii) forces $\cF'$ to be of the same type (Reeb or suspension, with the same behaviour of holonomy) as $\cF$ on each $A_i$.  
By \cite[Theorem~4.2.15]{hectorIntroductionGeometryFoliations1981}, two such foliations on an annulus are topologically equivalent by a homeomorphism preserving the boundary and the orientation.  
Gluing these homeomorphisms along the boundary circles gives a global topological equivalence between $\cF$ and $\cF'$.
\end{proof}

\begin{defi}[Parallel vs.\ scalloped bi–(pre)foliations]\label{def: parallel scalloped}
Let $(\cF_+,\cF_-)$ be a pair of quasi–transverse pre-foliations on an oriented torus $S$, and let $\Gamma := \Gamma_+ \cup \Gamma_-$ be the union of their compact leaves.  
We say that the pair $(\cF_+,\cF_-)$ is \emph{parallel} if every compact leaf of $\cF_+$ is freely homotopic to a compact leaf of $\cF_-$ (as unoriented curves); and \emph{scalloped} otherwise.
\end{defi}

Notice that if the set of marked leaves $\Gamma_*\neq\emptyset$, then the pair must be parallel.

\medskip

\begin{defi}[Geometric enumeration]\label{def: geometric enumeration}
Let $\cF$ be a \qms{} pre-foliation on a torus $T$. Fix a first leaf $\gamma_0$ in $\Gamma$ and a cyclic order on the compact leaves $\Gamma$.
We say that $\{\gamma_0,\dots,\gamma_{n-1}\}$ is a \emph{geometric enumeration} of $\Gamma$ if they are cyclically ordered.

Let $(\cF_+,\cF_-)$ be a pair of parallel quasi–transverse pre-foliations on $T$.
Fix a first leaf $\gamma_0$ in $\Gamma$ and a cyclic order on $\Gamma:=\Gamma_+\cup\Gamma_-$.
We say that $\{\gamma_0,\dots,\gamma_{n-1}\}$ is a \emph{geometric enumeration} of $\Gamma$ if they are cyclically ordered.
\end{defi}

\begin{rmk}
    When $T$ is oriented, a cyclic order and a geometric enumeration are determined by just a choice of a first leaf $\gamma_0$: The cyclic order we choose by convention is the one determined by the dynamical orientation of $\gamma_0$ and the orientation of $T$, so that $\gamma_1$ is the first compact leaf on the right of $\gamma_0$.
\end{rmk}

Note that, for a fixed orientation on $T$, there are as many geometric enumerations as there are choices of a first compact leaf.  
A choice of geometric enumeration will allow us to encode the data of a pair of parallel quasi–transverse \qms{} pre-foliations into an abstract combinatorial object that we call the \emph{bar code}.

\begin{defi}[Abstract bar code]\label{def: abstract bar code}
An \emph{abstract bar code} is the data of a map
\[
\sigma\colon \{0,\dots,n-1\} \longrightarrow \{+,-,*,\mm\}
\]
together with an integer $k\in\{1,\dots,n-1\}$.
\end{defi}

\begin{defi}[Bar code]\label{def: bar code}
Let $T$ be an oriented torus endowed with a parallel quasi–transverse pair of pre-foliations $(\cF^+,\cF^-)$, and let $\Gamma := \Gamma_+\cup\Gamma_-$ be the set of compact leaves of the pair.  
Choose compact leaves $\gamma_0^+\in\Gamma_+$ and $\gamma_0^-\in\Gamma_-$ as first compact leaves for $\cF^+$ and $\cF^-$.  
Let $\Gamma =\{\gamma_0,\dots,\gamma_{n-1}\}$ be the geometric enumeration obtained by taking $\gamma_0 = \gamma_0^+$.\footnote{Note that this geometric enumeration respects the cyclic order on $\Gamma_+$ given by $\gamma_0^+$, but not necessarily the cyclic order on $\Gamma_-$ given by $\gamma_0^-$.}

The \emph{bar code} of the pair $(\cF^+,\cF^-)$ is the map
\[
\sigma\colon\{0,\dots,n-1\}\longrightarrow \{+,-,*,\mm\}
\]
defined by
\[
\sigma(i)=
\begin{cases}
*  &\text{if } \gamma_i\in \Gamma_*,\\[0.2em]
\mm &\text{if } \gamma_i\in \Gamma_\mm\smallsetminus \Gamma_*,\\[0.2em]
+  &\text{if } \gamma_i\in \Gamma_+\smallsetminus \Gamma_\mm,\\[0.2em]
-  &\text{if } \gamma_i\in \Gamma_-\smallsetminus \Gamma_\mm,
\end{cases}
\]
together with the integer $k$ such that $\gamma_k = \gamma_0^-$.
\end{defi}

\begin{rmk}
The \emph{combinatorial type} of a single quasi–Morse–Smale foliation and of a quasi–transverse \qms{} bifoliation was introduced in \cite[Definitions~7.14 and~7.36]{pauletAnosovFlowsDimension2025}, but the present notion of bar code is slightly different:  
First, it is adapted to the current definition of quasi–transverse bifoliation, which includes the larger set $\Gamma_\mm$ designed for pseudo–Anosov flows (see Remark~\ref{rem: gamma_**}).  
Second, this coding is better suited to describe the bifoliation itself (see Proposition~\ref{prop: bar code determine}).  
The combinatorial type introduced in \cite{pauletAnosovFlowsDimension2025} sufficed to recover the topological equivalence class of a \qms{} bifoliation, but carried more information than actually needed for our purposes here.
\end{rmk}

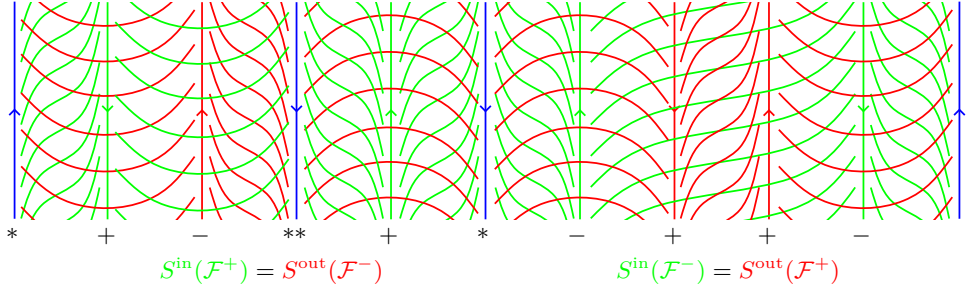
\begin{figure}[h]

    \labellist
        \small
        \pinlabel $*$ [t] at 2 0
        \pinlabel $+$ [t] at 55 0
        \pinlabel $-$ [t] at 108 0
        \pinlabel $**$ [t] at 161 0
        \pinlabel $+$ [t] at 214 0
        \pinlabel $*$ [t] at 267 0
        \pinlabel $-$ [t] at 320 0
        \pinlabel $+$ [t] at 374 0
        \pinlabel $+$ [t] at 427 0
        \pinlabel $-$ [t] at 480 0

        \pinlabel ${\color{green}S^\iin(\cF^+)}={\color{red}S^\out(\cF^-)}$ [t] at 148 -17
        \pinlabel ${\color{green}S^\iin(\cF^-)}={\color{red}S^\out(\cF^+)}$ [t] at 405 -17
        
    \endlabellist

    \centering
    \includegraphics[width=1\linewidth]{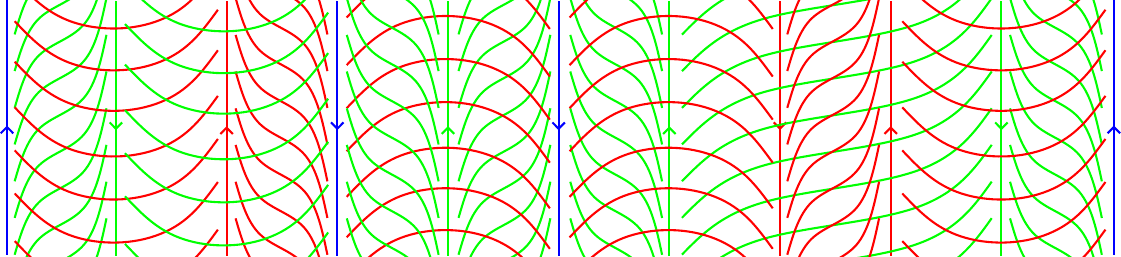}
    \vspace{0.5cm}
    \caption{Bar code of a parallel quasi-transverse pair pair of \qms{} foliations $(\cF_+, \cF_-)$}
    \label{fig: bar code}
\end{figure}

\begin{prop}[Realization of abstract bar code]\label{prop: real bar code}
An abstract bar code 
\[ 
(\sigma\colon\{0,\dots,n-1\}\to\{+,-,*,\mm\},\,k)
\]
is the bar code of a parallel quasi–transverse bifoliation $(\cF^+,\cF^-)$ on a torus $T$ if and only if $n$ is even and $\vert\{\sigma=*\}\vert$ is even.
\end{prop}

\begin{proof}
If $\sigma$ is the bar code of a pair of quasi–transverse foliations $(\cF^+,\cF^-)$ on a torus $T$, then, by Definition~\ref{def: qms lam} and orientability, the number of marked compact leaves of the union is even, hence $\vert\{\sigma=*\}\vert$ is even.

Next, we show that for a geometric enumeration 
\(\Gamma=\{\gamma_0,\dots,\gamma_{n-1}\}\) of the compact leaves, two consecutive leaves $\gamma_i$ and $\gamma_{i+1}$ (with indices taken modulo $n$) have incoherent dynamical orientations, that is, they are freely homotopic with opposite orientations.  
This immediately implies that $n$ is even.

Indeed, the two leaves $\gamma_i$ and $\gamma_{i+1}$ bound an open annulus $C\subset T$ whose interior contains no compact leaves, so the restrictions of $\cF^+$ and $\cF^-$ to $C$ are elementary foliations and are transverse to each other.  
Without loss of generality we may assume that $C\subset T^\iin_+ = T^\out_-$.  

Suppose, by contradiction, that $\gamma_i$ and $\gamma_{i+1}$ have coherent dynamical orientations.  
Up to swapping $+$ and $-$ and exchanging $i$ and $i+1$, we are in one of the following situations:
\begin{enumerate}
    \item both $\gamma_i$ and $\gamma_{i+1}$ belong to $\cF^+\smallsetminus \Gamma_\mm$;
    \item $\gamma_i\in \cF^+\smallsetminus\Gamma_*$ and $\gamma_{i+1}\in \cF^-\smallsetminus\Gamma_\mm$;
    \item both $\gamma_i$ and $\gamma_{i+1}$ belong to $\Gamma_\mm$;
    \item $\gamma_i\in \cF^+\smallsetminus\Gamma_\mm$ and $\gamma_{i+1}\in\Gamma_\mm$.
\end{enumerate}
In cases (1), (3) and (4), the leaves $\gamma_i$ and $\gamma_{i+1}$ lie in the same foliation, say $\cF^+$, and bound the annulus $C$; hence the restriction of $\cF^+$ to $C$ is a Reeb component.  
This forces $\cF^-$ to admit a compact leaf inside $C$, contradicting the assumption that $\Gamma$ contains all compact leaves.

In case (2), the leaves of $\cF^+$ in $C$ accumulate on $\gamma_i$ with expanding holonomy and are transverse to $\gamma_{i+1}$, while the leaves of $\cF^-$ are transverse to $\gamma_i$ and accumulate on $\gamma_{i+1}$ with contracting holonomy.  
Again, this configuration is impossible under the transversality assumption.

Thus the union $\Gamma$ of compact leaves must have even cardinality, and therefore $n$ is even.

Conversely, given an abstract bar code satisfying the two conditions, one constructs a parallel quasi–transverse bifoliation on a torus by gluing model bi–foliated cylinders according to the pattern prescribed by $\sigma$ and $k$, using elementary Reeb and suspension pieces as in the proof of Lemma~\ref{lem: extend qms pre-foliation by qms foliation}.  
The details follow the same lines as in \cite{pauletAnosovFlowsDimension2025} and are omitted.
\end{proof}

\begin{prop} \label{prop: bar code determine}
    Suppose $(\cF_1^+, \cF_1^-)$ and $(\cF_2^+, \cF_2^-)$ are two quasi-transverse parallel bi-foliations on tori $T_1$ and $T_2$ with the same bar code. Then there exists a homeomorphism $h\colon T_1 \ssm \Gamma_{\mm,1} \to T_2 \ssm \Gamma_{\mm,2}$ preserving the orientation and mapping the restriction of the pair $(\cF_1^+, \cF_1^-)$ to the restriction of the pair $(\cF_2^+, \cF_2^-)$, preserving the geometric enumeration.
\end{prop}

\begin{proof}
All bi-foliated model cylinders with the same bar code on the boundary circles are homeomorphic.
The homeomorphism extends to the boundary except maybe for the case where the boundary is a double marked circle leaf ($\sigma = \mm$).
\end{proof}

\begin{rmk}
In general, the homeomorphism given by Proposition~\ref{prop: bar code determine} does not extend across the compact leaves in $\Gamma_\mm$.
Indeed, let $\gamma_0\in\Gamma_\mm$, and choose local coordinates
$(x,\theta)\in [0,1] \times S^1$
in a tubular neighborhood of $\gamma_0=\{x=0\}$.
Consider a quasi-transverse pair $(\cF_1^+,\cF_1^-)$ such that, in these coordinates, the other leaves of $\cF_1^+$ are the curves $\theta = -\ln(x) +c,$ and those of $\cF^-_1$ are $\theta = \ln(x) +c.$
Now consider another pair $(\cF_2^{+},\cF_2^{-})$ where $\cF_2^+ = \cF^+_1$ but the leaves of $\cF_2^{-}$ are instead given by $\theta = 2\ln(x) +c.$

The two pairs are topologically equivalent on the annulus
$]0,1]\times S^1$.
However, any homeomorphism conjugating the two pairs cannot extend continuously across the compact leaf $x=0$.
To see this, one can consider the sequence of intersection points $\{p_n\}$ of a pair of leaves $(l^+,l^-)$ of $(\cF_1^+,\cF_1^-)$ and the sequence of intersection points $\{p'_n\}$ of a pair $(l'{}^+,l'{}^-)$ of $(\cF_2^+,\cF_2^-)$  : the first converges to a single point in $\{x=0\}$ and the second has distinct points in its limit set.
\end{rmk}

\subsection{Foliation induced on quasi-transverse tori}
\label{subsec: foliation induced}

Let $T$ be a weakly embedded maximally transverse quasi-transverse torus for a pseudo-Anosov flow $(M,\phi)$, and let $\cF^s,\cF^u$ be the stable and unstable (singular) foliations.
Set
\[
\cF_T^s := \cF^s\cap T,\qquad \cF_T^u := \cF^u\cap T.
\]

\begin{prop}\label{prop: trace foliation on qt torus in pa flow}
 $\cF_T^s$ and $\cF_T^u$ are Morse-Smale foliations on $T$ (in the sense of Definition~\ref{def: ms lam}), they are transverse on $T\setminus \cO$, and they are tangent along the periodic orbits $\cO$ of $\phi$ contained in $T$.
\end{prop}

\begin{proof}
By Definition~\ref{def: qt surface} and Definition~\ref{def: weakly embedded quasi transverse surface}, the complement
$T\setminus\cO$ is embedded and transverse to $\phi$.
Write $A_i$ for the open Birkhoff annuli in $T$ bounded by
$O_{i-1}$ and $O_i$.
By \cite[Proposition 5.2.2]{barthelmePseudoAnosovFlowsPlane2025},
the stable and unstable foliations induce a pair of one-dimensional (non-singular) transverse foliations
$(f_i^s,f_i^u)$ on the interior of $A_i$, extending to a pair of one-dimensional (possibly singular) foliations on $A_i$
whose closed leaves include the $O_i$ (and along which the trace is tangent).
Doing this on each $A_i$ gives us a pair of one-dimensional foliations $(\cF_T^s, \cF_T^u)$ on $T$, transverse on $T \setminus \cO$ and tangent along $\cO$.

Consider $\cF_T^s$ (the argument for $\cF_T^u$ is identical). Compact leaves $\gamma$ of $\cF_T^s$ occur as intersections with $T$ of stable cylinder leaves that contain periodic orbits.
Take a small transverse segment $\tau\subset T$ crossing $\gamma$ once. The holonomy return map of $\cF_T^s$ along $\gamma$ is conjugate to the first return map of the flow on a transversal inside the corresponding stable leaf; hence is hyperbolic (contracting for one orientation, expanding for the opposite). In particular, the holonomy is either contracting on both sides or expanding on both sides for a chosen orientation of $\gamma$.
Hyperbolic holonomy implies such compact leaves are isolated, hence there can be only finitely many on the compact surface $T$.

Remove the compact leaves of $\cF_T^s$ from $T$. We obtain a finite union of open annuli, and on each such annulus the foliation has no compact leaves in its interior.
By the standard structure theorem for foliations on an annulus tangent to the boundary (see \cite{hectorIntroductionGeometryFoliations1981}), the restriction of $\cF_T^s$ to such an annulus is (up to topological equivalence fixing the boundary) either a Reeb component, or a suspension foliation without fixed point in the interior.
In all cases, every non-compact half-leaf accumulates on a compact boundary leaf with contracting / expanding holonomy.

This is precisely Definition \ref{def: ms lam} of Morse-Smale foliation.
\end{proof}

Let $T$ be a weakly embedded maximally transverse quasi-transverse torus in a pseudo-Anosov flow $(M,\phi)$.
We actually want to consider \qms{} foliations on $T$, so we will consider the trace $\cF_T^s$ on some transverse annuli of $T$ and the trace $\cF_T^u$ on some other.

Fix a transverse orientation of $T$, and let $\cO_a$ be the set of periodic orbits contained in $T$
along which the transverse orientation of $\phi$ is alternating.
Let $T^+$ be the union of components of $T\setminus \cO_a$ where the transverse orientation of $\phi$
coincides with that of $T$, and let $T^-$ be the union of the remaining components.
Define
\begin{align}\label{eq: quasi transverse bifoliation in pa flow}
\begin{array}{cc}
            \cF_T^+ =& (\cF^s \cap T^+) \cup (\cF^u \cap T^-) \cup \cO_a,\\
            \cF_T^- =& (\cF^u \cap T^+) \cup (\cF^s \cap T^-) \cup \cO_a.
\end{array}
\end{align}

We have

\begin{prop}\label{prop: qt bifoliation on qt surface in pA flow}
The pair $(\cF_T^+,\cF_T^-)$ is a pair of quasi-transverse quasi-Morse-Smale foliations on $T$.
\end{prop}

\begin{proof}
Write $\cF_T^s:=\cF^s\cap T$ and $\cF_T^u:=\cF^u\cap T$ for the trace foliations on $T$.
By Proposition~\ref{prop: trace foliation on qt torus in pa flow}, each of $\cF_T^s$ and $\cF_T^u$ is a Morse-Smale
foliation on $T$, transverse to each other away from the periodic orbits of the flow contained in $T$.

\smallskip
\noindent\textbf{Claim 1.} Each of $\cF_T^\pm$ is a one-dimensional quasi-Morse-Smale foliation on $T$.

On each connected component $A$ of $T\setminus \cO_a$, the foliation $\cF_T^+$ is defined to be either
$\cF_T^s|_A$ (if $A\subset T^+$) or $\cF_T^u|_A$ (if $A\subset T^-$); similarly for $\cF_T^-$.

The compact leaves of $\cF_T^\pm$ are contained in compact leaves of the union $\cF_T^s \cup \cF_T^u$, hence there are finitely many of them. Each such compact leaf has hyperbolic holonomy
on each $A$.

We now specify the quasi-Morse-Smale structure (marked leaves, splitting, and dynamical orientations).
Declare the marked set to be
$\Gamma_*(\cF_T^\pm):=\cO_a$.
Cutting along $\cO_a$ gives the decomposition
$T\setminus \cO_a = T^+ \sqcup T^-$,
which we take as the $(\text{in},\text{out})$-splitting for $\cF_T^+$ by setting
\[
T^{\mathrm{in}}({\cF_T^+}):=T^+,\qquad T^{\mathrm{out}}({\cF_T^+}):=T^-,
\]
and for $\cF_T^-$ by swapping:
\[
T^{\mathrm{in}}({\cF_T^-}):=T^-,\qquad T^{\mathrm{out}}(\cF_T^-):=T^+.
\]
Finally, the dynamical orientation on each compact leaf is the one induced by the flow direction
on the unique periodic orbit inside the same stable / unstable leaf, i.e. such that the oriented compact leaf and the oriented periodic orbit are isotopic inside the leaf.

We know from standard pseudo-Anosov theory that the holonomy of a cylinder leaf of $\cF^s$ is expanding along the oriented periodic orbit, and contracting for $\cF^u$.
By taking the trace on $T$ we get conjugated holonomy of the oriented compact leaves of $\cF_T^s$ and $\cF_T^u$.
With our choice of splitting, we get that the holonomy is expanding on the $\iin$-side and contracting on the $\out$-side.

With these choices, $\cF_T^\pm$ satisfies the definition of a quasi-Morse-Smale foliation on the torus.
\qed

\smallskip
\noindent\textbf{Claim 2.} The pair $(\cF_T^+,\cF_T^-)$ is quasi-transverse.

Let $\Gamma_{**} := \cO$ be the set of periodic orbits of $\phi$ contained in $T$ (equivalently, the compact leaves common to $\cF_T^s$ and $\cF_T^u$).
Every orbit in $\cO_a$ is contained in $\cO$
so $\Gamma_{*,\cF_T^+}=\Gamma_{*,\cF_T^-}=\cO_a\subset \Gamma_{**}$.
Moreover, on $T\setminus \Gamma_{**}$, stable and unstable traces are transverse, and on each component of
$T\setminus \cO_a$ we have arranged that $\cF_T^+$ coincides with one of $\{\cF_T^s,\cF_T^u\}$ while $\cF_T^-$ coincides with the other.
Hence $\cF_T^+$ and $\cF_T^-$ are transverse on $T\setminus \Gamma_{**}$.

Finally, by the definition of the splittings above we have
\[
T^{\mathrm{in}}({\cF_T^+})=T^+=T^{\mathrm{out}}({\cF_T^-}),
\]
so the splittings are opposite. Therefore $(\cF_T^+,\cF_T^-)$ is a quasi-transverse pair.

This concludes the proof.
\end{proof}

\section{Pseudo-Anosov pieces} \label{sec: pA pieces}

\subsection{Pieces of pseudo-Anosov flow}
\label{subsec: pap}

\begin{defi}[Weakly embedded tori, good collection] \label{def: weakly embedded}
Let $\cT=\{T_1,\dots,T_n\}$ be a collection of quasi-transverse tori in a closed $3$-manifold $M$ endowed with a pseudo-Anosov flow $\varphi$.
We say that $\cT$ is \emph{weakly embedded} if the union $\bigcup_i T_i$ is embedded except possibly along periodic orbits of $\varphi$ contained in the $T_i$ (in particular, distinct tori may intersect only along periodic orbits).

We say that $\cT$ is a \emph{good collection} if, moreover, each $T_i$ is homotopic to an embedded torus $T_i'$ such that the collection
$\cT'=\{T_1',\dots,T_n'\}$ is embedded and pairwise disjoint.
\end{defi}

\begin{defi}[Boundary and corner periodic orbits] \label{defi: boundary corner periodic orbits}
Let $\cT$ be a weakly embedded collection of quasi-transverse tori for a pseudo-Anosov flow $\varphi$.
The set of \emph{boundary periodic orbits} is the set
$\cO$ of periodic orbits of $\varphi$ contained in $\bigcup_i T_i$.
The subset of \emph{corner periodic orbits} is the set
$\cO_c$ of periodic orbits of $\varphi$ contained in $T_i\cap T_j$ for some $i\neq j$.
\end{defi}

\begin{defi}[Pseudo-Anosov piece]\label{def: pseudo anosov piece}
Let $\phi$ be a pseudo-Anosov flow on a closed $3$-manifold $M$ and let $\cT$ be a finite good collection of weakly embedded quasi-transverse tori.
Let $\cT'=\{T_1',\dots,T_n'\}$ be the associated embedded disjoint collection as in Definition~\ref{def: weakly embedded}.
Let $P'$ be a union of closures of connected components of $M\setminus \cT'$ and let $\varphi':=\res{\phi}{P'}$.

For each connected component $P_i'$ of $P'$, push each boundary torus of $P_i'$ corresponding to some $T_j'$ back into the quasi-transverse position corresponding to $T_j$, by a homotopy supported in a collar of $\partial P_i'$.
Denote by $P_i$ the resulting set and by $\varphi:=\res{\phi}{P}$ the restricted flow, where $P:=\bigcup_i P_i$.
We call $(P,\varphi)$ a \emph{pseudo-Anosov piece}, and refer to $(P',\varphi')$ as the \emph{associated embedded piece}.
Any finite disjoint union of pseudo-Anosov pieces is declared to be a pseudo-Anosov piece.
\end{defi}

The purpose of this definition is to be compatible with the modified JSJ decomposition:
every modified JSJ piece is a pseudo-Anosov piece in the above sense (see Section~\ref{sec: preli}).

\begin{convention}
\label{conv: phi varphi}
    As we both consider pseudo-Anosov flows on closed manifolds and pseudo-Anosov pieces, which are flows on manifolds with boundary, we use the following mnemonic convention: We will generally use $\phi$ (possibly with subscripts) to denote flows on a \emph{closed} manifold and $\varphi$ (with matching subscripts) to denote the flows on manifolds with boundary.
\end{convention}

\begin{rmk}\label{rmk: no klein bottle}
In this paper, we cut pieces only along tori.
In the non-orientable setting, incompressible Klein bottles may also appear and one can do analogous definitions and statements.
All arguments can be adapted to that more general framework, but we restrict to the torus case in order to keep the exposition lighter.
\end{rmk}

\begin{rmk}\label{rmk: pap not manifold}
We allow $\cT$ to be only weakly embedded, so a \pap{} $P$ need not be a manifold.
Topologically, $P$ is a $3$-manifold with finitely many \emph{corner circles} coming from the identifications
along $\cO_c$.
\end{rmk}

The embedded piece $P'$ is a submanifold of $M$.
There is a homotopy $H\colon P' \to P$.

\begin{defi}[Interior and boundary of a piece]\label{def: interior boundary pap}
The \emph{interior} $\intr P$ is the image $H(\intr P')$.
The \emph{formal boundary} $\partial P$ is the collection $\{ H(S_i) \}$ where each $S_i$ is a boundary component of $\partial P'$. 
The \emph{real boundary} $\partial_r P$ is the image $H(\partial P')$.
\end{defi}

Note that $\partial P$ is a {formal} subset of $P$.
If $P$ is a manifold without singularity (isotopic to $P'$), then
$\partial_r P=\partial P$ and coincides with the usual boundary of $P$ as a submanifold.

\begin{rmk}\label{rmk: pap components}
A \pap{} $P$ need not be connected, and we do not require it to contain all connected components of $M\setminus \cT$:
we allow $P'$ to be any union of closures of components.
If $P'$ contains all components of $M\setminus \cT'$, then $\intr P' \to M\setminus \cT'$ is a diffeomorphism and
the map $\partial P' \to \cT$ has degree $2$.
\end{rmk}

\begin{defi}
    We define $\Pin \subset \partial P$ the union of boundary surfaces where the flow points inside of $P$,
and $\Pout \subset \partial P$ those where it points outside of $P$.
These two subsets are adjacent along the alternating periodic orbits $\cO_a\subset \cO$.
\end{defi}

We illustrate the potential differences between a piece $P'$, the modified piece $P$ and the different definitions of boundaries in the following example.
\begin{example}

Consider a manifold $M$ and a pseudo-Anosov flow $\phi$ such that the collection $\cT$ in $M$ is made of two separating quasi-transverse tori $T_1$ and $T_2$, each embedded, each containing two periodic orbits, and such that the collection is weakly embedded but not embedded: $T_1$ and $T_2$ intersect along one common periodic orbit $Q$. 
The tori $T_1$ and $T_2$ are homotopic to disjoint embedded tori $T'_1$ and $T'_2$.
Let $P'$ be the union of all closures of connected components of $M \ssm \cT'$.
The set $P'$ admits 3 connected components with boundary tori which are copies of $T'_i$. The set $P$ is then obtained by pushing each copy of $T'_i$ to a quasi-transverse position by homotopy. See Figure \ref{fig: cut pap}.

The formal boundary of $P$ is the collection of 4 tori, which are copies of $T_i$.
The formal boundary and the (real) boundary are different: The tori bounding the central piece intersect in the real boundary along the corner periodic orbits projecting to $Q$, hence corresponding to a unique connected component of the real boundary, but not of the formal boundary. 
The tori bounding the side pieces also contain a corner periodic orbit corresponding to $Q$, but they do not bound the same connected component, hence they do not intersect in the boundary of $P$.
\end{example}

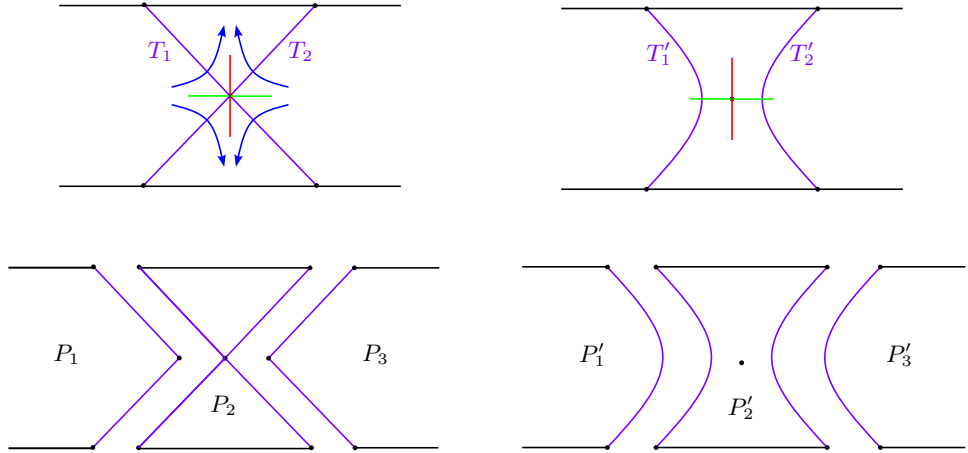
\begin{figure}[h]
\labellist
\small
\pinlabel $\textcolor{monviolet}{T_1}$ [tr] at 107 260
\pinlabel $\textcolor{monviolet}{T_2}$ [tl] at 172 260

\pinlabel $\textcolor{monviolet}{T_1'}$ [tr] at 420 260
\pinlabel $\textcolor{monviolet}{T_2'}$ [tl] at 487 260

\pinlabel $P_1$ [r] at 50 60
\pinlabel $P_3$ [l] at 217 60
\pinlabel $P_2$ [t] at 135 39

\pinlabel $P'_1$ [r] at 380 60
\pinlabel $P'_3$ [l] at 546 60
\pinlabel $P'_2$ [t] at 460 39

\endlabellist

    \centering
    \includegraphics[width=1\linewidth]{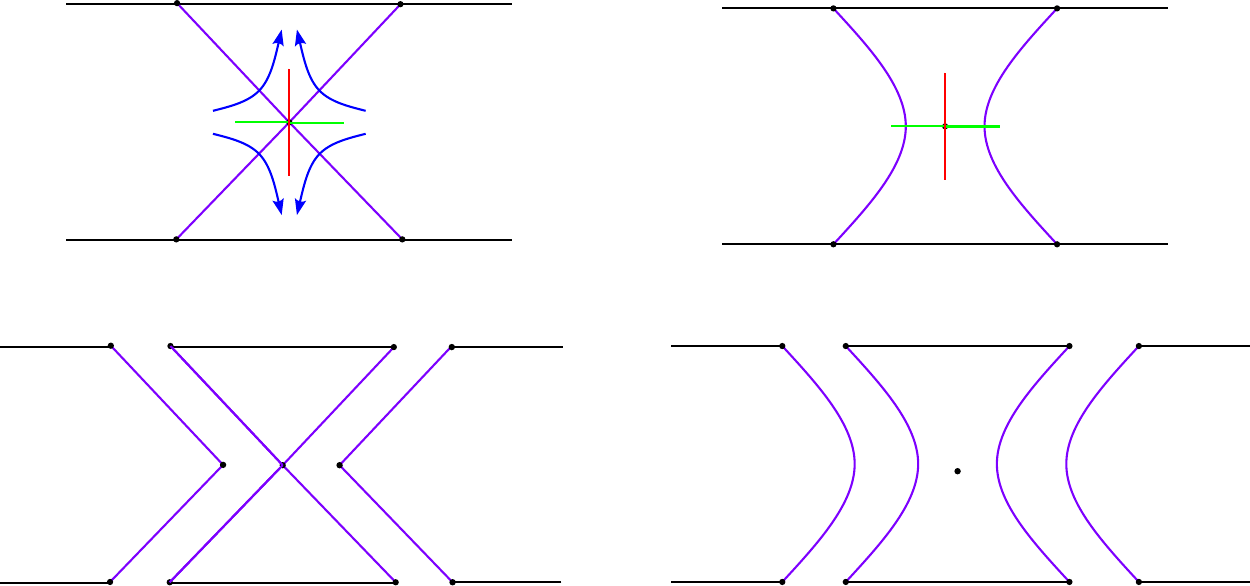}
    \caption{A pseudo-Anosov piece (left) and associated embedded piece (right).}
    \label{fig: cut pap}
\end{figure}

\begin{defi}[Pseudo-Anosov gluing map]
\label{def: gluing map}
A \emph{\pa{} gluing map of a \pap{} $(P,\varphi)$} is a $\cC^1$-involution 
$f \colon \partial P \to \partial P$ pairing the boundary surfaces of the formal boundary $\partial P$, so that the quotient $P/f$ is a closed manifold and $\varphi$ induces a flow $\varphi_f$ on $P/f$ which is pseudo-Anosov.
We say that $(P,\varphi,f)$ is a \textit{\pa{} triple}.
\end{defi}

\begin{rmk}
    The formal boundary is a formal subset of $P$, hence the quotient $P/f$ is well defined.
\end{rmk}

\begin{nota}[Partition of boundary] \label{nota: pairing gluing}
    We denote by $\partial_+P$ and $\partial_-P$ a partition of the boundary surfaces of $\partial P$ corresponding to a pairing of a \pagp{} $f$ for a \pap{} $(P,\varphi)$.
\end{nota}

In particular, a \pa{} gluing map maps periodic orbits contained in $\partial P$ to other periodic orbits in $\partial P$, preserving the flow orientation, and maps $\Pout$ to $\Pin$.

\begin{rmk}
Note that the flow $(P/f, \varphi_f)$ does not, in general, coincide with $(M,\phi)$, the original pseudo-Anosov flow from which the piece $(P,\varphi)$ was cut, even when $P/f$ and $M$ are diffeomorphic.
\end{rmk}

Let $\cF^s$ and $\cF^u$ be the stable and unstable foliations on $P_f := P/f$ for the pseudo-Anosov flow $\varphi_f$ induced by $\varphi$.
Consider $\pi\colon P \to P_f$ the projection, and denote $\cT = \pi(\partial P)$, and $\hat \cT =  \pi(\partial P \ssm \cO)$.
Denote $(\cF_\cT^s, \cF_\cT^u)$ the trace of $(\cF^s, \cF^u)$ on $\cT$ (see Proposition \ref{prop: trace foliation on qt torus in pa flow}).
Denote $\pi^\out$ the restriction of the projection $\pi$ to $\Pout$, and $\pi^\iin$ the restriction of the projection $\pi$ to $P^\iin$. 
Both are homeomorphisms onto their image $\hat \cT $.
Denote $\cF_\partial$ the foliation on $\pP$ formed by taking the union of the lifts of $\cF_\cT^s$ in $\Pin$ together with the lifts of $\cF_\cT^u$ in $\Pout$, as well as the tangent periodic orbits. More precisely
\begin{equation} \label{eq: boundary qms foliation on pa piece}
    \cF_\partial = (\pi^\out)^\inv (\cF_\cT^u \cap \hat \cT)  \cup (\pi^\iin)^\inv (\cF_\cT^s \cap \hat \cT) \cup \cO
\end{equation}
It is a \qms{} foliation, as it projects to the foliation $\cF_\cT^\pm$ on $\cT$ defined in \eqref{eq: quasi transverse bifoliation in pa flow}.

Note that the foliation $\cF_\partial$ contains much more than just the information about the flow $\varphi$ on the piece $P$, as it depends on the foliations of the complete flow $\varphi_f$ on the closed manifold $P/f$, hence strongly depends on the choice of the gluing map $f$. 
The flow $\varphi$ on $P$ only gives rise to a sublamination of $\cF_\partial$:
\begin{defi}[Boundary foliation and boundary prefoliation] \label{def: boundary lam and fol}
    We say that $\cF_\partial$ defined by \eqref{eq: boundary qms foliation on pa piece} is the \textit{boundary foliation} of the \pa{} triple $(P,\varphi,f)$.
    The \emph{boundary prefoliation}, denoted $\cK_\partial$, of the piece $P$ is the sublamination of $\cF_\partial$ formed by the trace of stable and unstable manifolds of orbits entirely contained in $P$.
\end{defi}

\begin{rmk} \label{rmk: dynamic of the boundary prefoliation}
The lamination $\cK_\partial$ corresponds to the trace of stable and unstable invariant laminations of the hyperbolic maximal invariant set for $\varphi$ in $P$ (possibly with pronged singular orbits) on the boundary.
    In other words, leaves of $\cK_\partial$ are leaves of $\cF^s \cap \Pin$ such that the forward orbit of the point never crosses $\partial P$ again, and leaves of $\cF^u \cap \Pout$ such that the backward orbit of the point never crosses $\partial P$ again.
    The projection on $\cT$ in $M$ corresponds to leaves of the trace of $\cF^s$ whose orbits never cross $\cT$ again in the future, and orbits on the trace of $\cF^u$ never crossing $\cT$ again in the past.
    Notice that the projection is not a lamination, but the union of a sublamination of $\cF^s \cap \cT$ and a sublamination of $\cF^u \cap \cT$ (intersecting each other).
    Note also that $\cF_\partial$ will depend on the \pa{} gluing map $f$, whereas $\cK_\partial$ only depends on the piece $(P,\varphi)$.
\end{rmk}

The boundary prefoliation $\cK_\partial$ is a closed sublamination of a \qms{} foliation. From \cite[Proposition 1.16]{pauletAnosovFlowsDimension2025}, we know that a connected component $C$ of the complement $\partial P \ssm \cK_\partial$ is either an annulus bounded by compact leaves\footnote{We could also have a Moebius strip in the case of a Klein bottle.}, or homeomorphic to $\R^2$ and the accessible boundary\footnote{Here, the \emph{accessible boundary} is defined as the set of points on the boundary of a connected component $C$ which can be obtained as the endpoints of a segment entirely contained in $C$.} is the union of two non-compact leaves $l$ and $l'$ which are asymptotic to each other at both ends.
In this latter case, we say that $C$ is a \textit{strip}.

\begin{rmk} \label{rmk: projection of boundary foliation}
One can easily see that the pair $(\cF_\partial, f_*\cF_\partial)$ on $\pP$ projects to $(\cF_\cT^+, \cF_\cT^-)$ in $\cT$ defined in \eqref{eq: quasi transverse bifoliation in pa flow}.
\end{rmk}

Having connected components of $\partial P \setminus \cK_\partial$ that are annuli bounded by compact leaves is often more delicate to handle.
In \cite{beguinBuildingAnosovFlows2017} and \cite{pauletAnosovFlowsDimension2025} this situation is excluded by imposing a \emph{filling} condition on the boundary prefoliation: it is not a necessary hypothesis in general to construct Anosov flows by gluing building blocks with (quasi-)transverse boundary, but a technical assumption to make the proof work in the general setting of \cite{beguinBuildingAnosovFlows2017} and \cite{pauletAnosovFlowsDimension2025}.
In the present paper, this condition will also be useful to encode the information of the boundary foliation (depending on the triple $(P,\varphi,f)$) from the one coming from the boundary prefoliation (depending only on the piece $(P,\varphi)$).

We recall the following terminology coming from \cite{beguinBuildingAnosovFlows2017} and \cite{pauletAnosovFlowsDimension2025}.

\begin{defi}[Filling boundary prefoliation, filled \pap{}]\label{def: filling lamination, filled piece}
Let $(P,\varphi)$ be a \pap{} and let $\cK_\partial$ be its boundary prefoliation.
We say that $\cK_\partial$ is \emph{filling} if every connected component of $\partial P \setminus \cK_\partial$ is a strip.
A \pap{} with filling boundary prefoliation is called a \emph{filled \pap{}}.
\end{defi}

\begin{lem} \label{lem: triple quasi transiverse bi foliation}
    Let $(P,\varphi,f)$ be a \pa{} triple.  
    Then the pair
    $(\cK_\partial, f_* \cK_\partial)$ is a pair of quasi-transverse pre-foliations on $\pP$.
\end{lem}

\begin{proof}
By construction, the pair $(\cF_\partial,f_*\cF_\partial)$ on $\partial P$ projects under $\pi$ to the pair
$(\cF_\cT^+,\cF_\cT^-)$ on $\pi(\partial P)$ induced by the traces of $\cF^s$ and $\cF^u$
(see Proposition~\ref{prop: qt bifoliation on qt surface in pA flow}).
Hence $(\cF_\partial,f_*\cF_\partial)$ is a quasi-transverse bifoliation on $\partial P$.
Since $\cK_\partial$ is a sublamination of $\cF_\partial$, the lemma follows.
\end{proof}

Recall (Definition~\ref{def: bar code}) that the topological type of a pair of quasi-transverse \qms{} prefoliations
on a torus is encoded by its bar code.  Associated to a \pa{} triple, we will use two related combinatorial invariants:
one defined from the boundary prefoliation and one defined from the boundary foliation.

\begin{defi}[Bar code and complete bar code of a \pa{} triple]\label{def: pA bar code}
Let $(P,\varphi,f)$ be a \pa{} triple.  Denote by $\cK_\partial$ its boundary prefoliation and by $\cF_\partial$ its boundary foliation.
\begin{enumerate}
    \item The \emph{bar code} of $(P,\varphi,f)$ is the bar code of the pair $(\cK_\partial,f_*\cK_\partial)$ on each parallel boundary component,
    after choosing a first compact leaf $\gamma_0$ for $\cK_\partial$ and taking $f(\gamma_0)$ as the first compact leaf for $f_*\cK_\partial$.
    \item The \emph{complete bar code} of $(P,\varphi,f)$ is the bar code of the pair $(\cF_\partial,f_*\cF_\partial)$ on each parallel boundary component,
    after choosing a first compact leaf $\gamma_0$ of $\cF_\partial$ that belongs to $\cK_\partial$ and taking $f(\gamma_0)$ as the first compact leaf for $f_*\cF_\partial$.
\end{enumerate}
\end{defi}

\begin{rmk}
Note that the (complete) bar code depends only on the choice of a first compact leaf $\gamma_0$ of $\cK_\partial$:
by convention we take $f(\gamma_0)$ as the first compact leaf for $f_*\cK_\partial$ (together with a choice of orientation of the boundary torus).
\end{rmk}

As pointed out earlier, the filling case is more rigid:

\begin{prop}\label{prop: bar code and complete bar code for filling piece}
Let $(P,\varphi)$ be a \pap{} whose boundary prefoliation $\cK_\partial$ is filling.
Then for any \pa{} gluing map $f$, the bar code and the complete bar code associated to the triple $(P,\varphi,f)$ coincide.
\end{prop}

\begin{proof}
It suffices to remark that if $\cK_\partial$ is a filling \qms{} sublamination of a \qms{} foliation $\cF_\partial$, then the compact leaves of $\cF_\partial$ are all compact leaves of $\cK_\partial$, and the dynamical orientations coincide. 
In other words, all foliations extending the prefoliation $\cK_\partial$ will satisfy Lemma \ref{lem: extend qms pre-foliation by qms foliation}. We refer to the proof of that lemma for more details.
\end{proof}

The conclusion of Proposition~\ref{prop: bar code and complete bar code for filling piece} is false in general when the boundary prefoliation is not filling, as shown in the following example.
    \begin{example}
        \label{ex: non-filling bar code}
We will build a pseudo-Anosov triple  $(P,\varphi,f)$, using the gluing constructions of \cite{beguinBuildingAnosovFlows2017}, for which the bar code and the complete bar code are distinct. We refer to Figure \ref{fig: gluing new compact leaves}.

\begin{figure}[h]
\labellist
\small

\pinlabel $(P_1,\varphi_1)$ at 114 778
\pinlabel $(P_2,\varphi_2)$ at 676 778
\pinlabel $(P_\mathrm{BL},\varphi_\mathrm{BL})$ at 399 778

\pinlabel $f_1$ [b] at 278 718
\pinlabel $f_2$ [b] at 528 722

\pinlabel $f_2$ [b] at 420 397

\pinlabel ${\color{red}\partial^\out P_1}$ [b] at 273 869
\pinlabel ${\color{green}\partial^\iin P_{\mathrm{BL}}}$ [b] at 375 869
\pinlabel ${\color{red}\partial^\out P_{\mathrm{BL}}}$ [b] at 514 869
\pinlabel ${\color{green}\partial^\out P_2}$ [b] at 621 869

\endlabellist

    \centering
    \includegraphics[width=\linewidth]{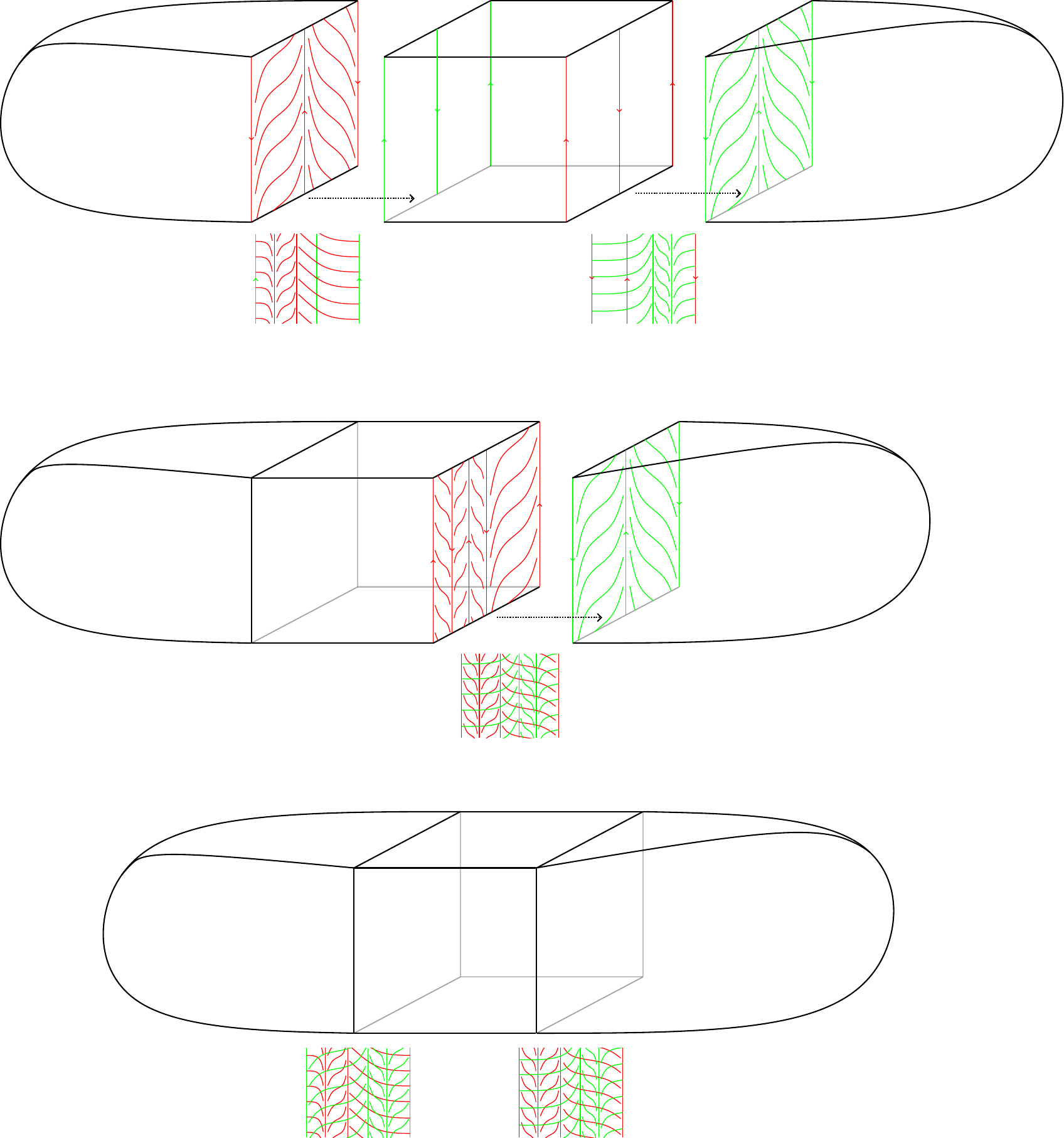}
    \caption{Gluing construction of Example \ref{ex: non-filling bar code}.}
    \label{fig: gluing new compact leaves}
\end{figure}
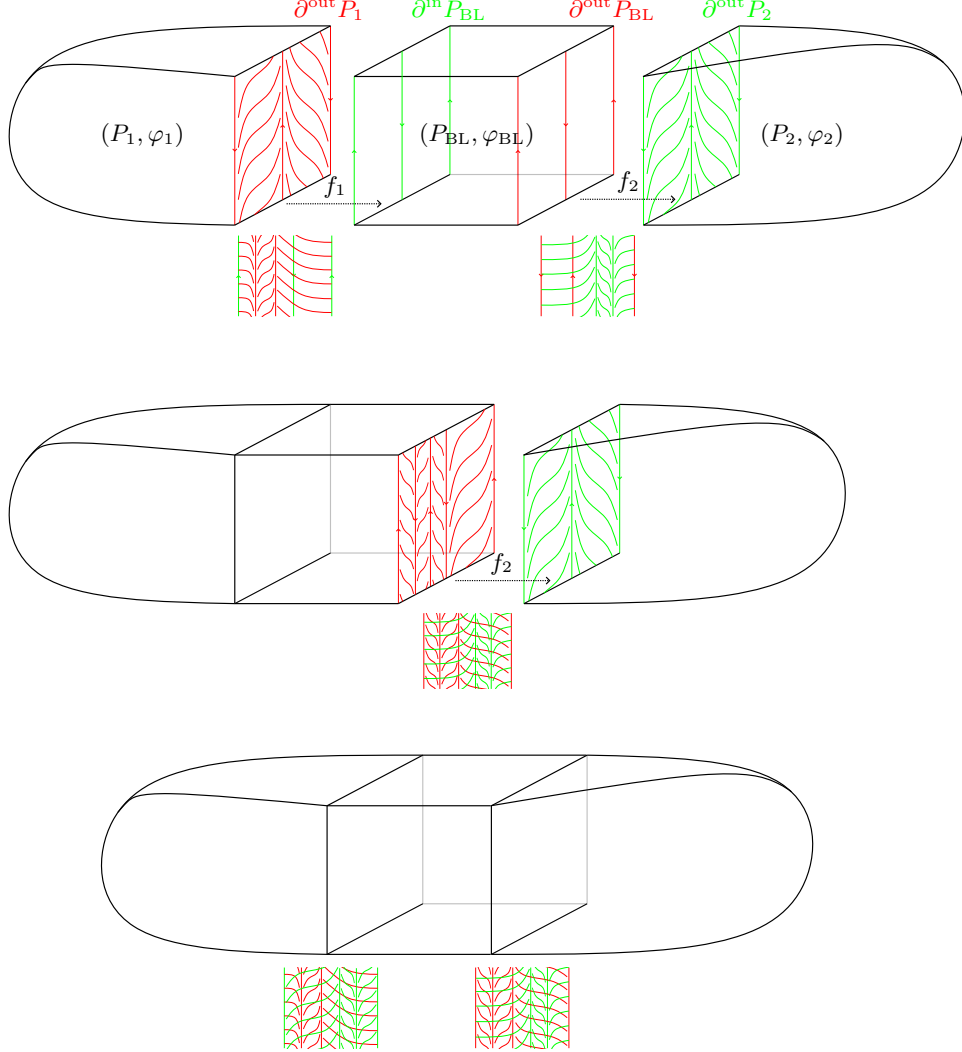

Let $(P_{\mathrm{BL}},\varphi_{\mathrm{BL}})$ be the Bonatti--Langevin building
block (see \cite{bonattiExempleFlotAnosov1994}). It has one entrance boundary torus
$T^{\mathrm{in}}_{\mathrm{BL}}$ and one exit boundary torus
$T^{\mathrm{out}}_{\mathrm{BL}}$. On each of these tori, the boundary
prefoliation has exactly two compact leaves with opposite dynamical orientations.
These leaves divide each boundary torus into two open annuli, which we denote by
\[
A^{\mathrm{in}}_1,\ A^{\mathrm{in}}_2
    \subset T^{\mathrm{in}}_{\mathrm{BL}},
\qquad
A^{\mathrm{out}}_1,\ A^{\mathrm{out}}_2
    \subset T^{\mathrm{out}}_{\mathrm{BL}},
\]
indexed so that every orbit entering through $A^{\mathrm{in}}_i$ exits through
$A^{\mathrm{out}}_i$.

Choose a repelling building block $(P_-,\varphi_-)$ with one exit boundary torus,
whose boundary lamination is a foliation with two compact leaves having opposite
dynamical orientations. Let $(P_+,\varphi_+)$ be the corresponding attracting block,
obtained by reversing the flow. Such blocks can be obtained from
\cite[Theorem~1.10]{beguinBuildingAnosovFlows2017}.

We first glue the exit boundary of $P_-$ to
$T^{\mathrm{in}}_{\mathrm{BL}}$ by a transverse gluing map which places
the two compact leaves coming from $P_-$ inside the annulus
$A^{\mathrm{in}}_1$. By
\cite[Proposition 1.10]{beguinBuildingAnosovFlows2017}, the resulting flow is again a building block (hyperbolic plug) with remaining exit boundary $T^{\mathrm{out}}_{\mathrm{BL}}$.
The two compact leaves inserted in $A^{\mathrm{in}}_1$ are transported through
the Bonatti-Langevin block and give rise to two additional compact leaves in
$A^{\mathrm{out}}_1$. Thus the induced boundary foliation on
$T^{\mathrm{out}}_{\mathrm{BL}}$ has four compact leaves, whereas the boundary
prefoliation intrinsic to $P_{\mathrm{BL}}$ has only the original two.

We then glue $T^{\mathrm{out}}_{\mathrm{BL}}$ to the entrance boundary of $P_+$
by a strongly transverse gluing map which places the two compact leaves of the
attracting block inside the other complementary annulus
$A^{\mathrm{out}}_2$. The gluing theorem of
\cite{beguinBuildingAnosovFlows2017} produces an Anosov flow on the resulting
closed manifold.

Equivalently, this construction may be viewed as a single gluing of the
disconnected pseudo-Anosov piece
\[
(P,\varphi)
=
(P_-,\varphi_-)
\sqcup
(P_{\mathrm{BL}},\varphi_{\mathrm{BL}})
\sqcup
(P_+,\varphi_+).
\]
Its boundary prefoliation $\cK_\partial$ records only the invariant laminations
intrinsic to these three pieces. The boundary foliation $\cF_\partial$ associated
with the complete gluing, however, contains the two additional compact leaves
created by the passage through the Bonatti-Langevin block. Consequently, the
pairs
\[
(\cK_\partial,f_*\cK_\partial)
\qquad\text{and}\qquad
(\cF_\partial,f_*\cF_\partial)
\]
have different bar codes. Hence the bar code and the complete bar code of the
triple do not coincide; see Figure~\ref{fig: Bar code vs complete bar code}.

\begin{figure}[h]
\labellist
\small

\pinlabel $+$ [t] at 0 110
\pinlabel $-$ [t] at 15 110
\pinlabel $-$ [t] at 31 110
\pinlabel $+$ [t] at 47 110

\pinlabel $-$ [t] at 129 110
\pinlabel $-$ [t] at 156 110
\pinlabel $+$ [t] at 175 110
\pinlabel $+$ [t] at 190 110

\pinlabel $+$ [t] at 0 0
\pinlabel $-$ [t] at 15 0
\pinlabel $-$ [t] at 31 0
\pinlabel $+$ [t] at 47 0
\pinlabel $+$ [t] at 62 0

\pinlabel $-$ [t] at 129 0
\pinlabel $-$ [t] at 144 0
\pinlabel $-$ [t] at 160 0
\pinlabel $+$ [t] at 175 0
\pinlabel $+$ [t] at 190 0

\endlabellist

    \centering
    \includegraphics[width=0.5\linewidth]{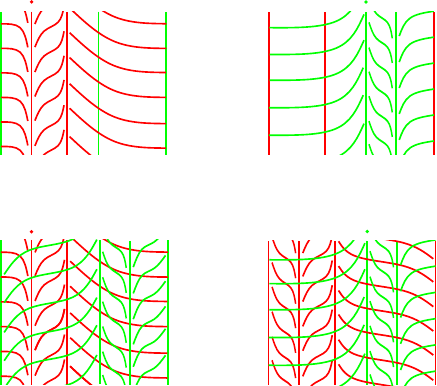}
    \vspace{5mm}
    \caption{Bar code (top) and complete bar code (bottom). The dot marks the first compact leaf of the image lamination.}
    \label{fig: Bar code vs complete bar code}
\end{figure}
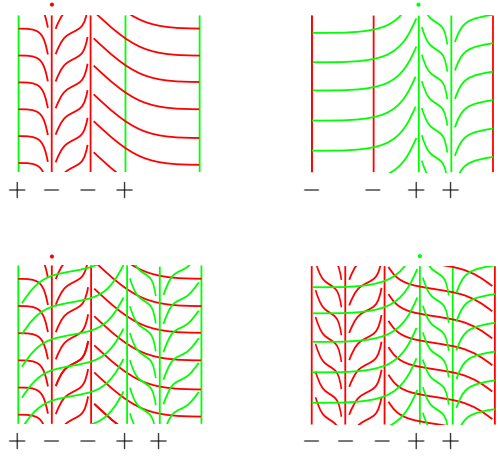

\end{example}

In the above example, while the complete bar code differed from the bar code, the complete bar code was completely determined by the bar code together with the isotopy class of the gluing map, and one may wonder whether this is a general fact. It turns out that it is not, as one can even build examples with a fixed bar code and a fixed isotopy class of gluing and still obtain distinct complete bar codes, as illustrated in the following example:

\begin{example}
    \label{ex: same bar code different complete bar codes}

In this example we construct a pseudo-Anosov piece $(P,\varphi)$ and two isotopic pseudo-Anosov
gluing maps $f$ and $f'$ such that the triples $(P,\varphi,f)$ and $(P,\varphi,f')$
have the same bar code but distinct complete bar codes.

For $i=1,2,3$, let $(P_i,\varphi_i)$ be three copies of the Bonatti--Langevin
building block introduced in Example \ref{ex: non-filling bar code}. Denote by
$T_i^{\mathrm{in}}$ and $T_i^{\mathrm{out}}$ their entrance and exit boundary
tori. The two compact leaves of the boundary prefoliation divide each of these
tori into two open annuli
\[
A_i^{\mathrm{in}}, B_i^{\mathrm{in}}
    \subset T_i^{\mathrm{in}},
\qquad
A_i^{\mathrm{out}}, B_i^{\mathrm{out}}
    \subset T_i^{\mathrm{out}},
\]
indexed so that orbits entering through $A_i^{\mathrm{in}}$ exit through
$A_i^{\mathrm{out}}$, and similarly for the annuli $B_i^{\mathrm{in}}$ and
$B_i^{\mathrm{out}}$. See Figure \ref{fig: different complete bar codes}.

\begin{figure}[h]
\labellist
\small

\pinlabel $P_+$ [b] at 221 1190
\pinlabel $P_1$ [b] at 417 1190
\pinlabel $P_2$ [b] at 659 1190
\pinlabel $P_3$ [b] at 880 1190
\pinlabel $P_-$ [b] at 1089 1190

\pinlabel $f_{12}$ [b] at 504 1050

\pinlabel $f_{23}$ [r] at 562 703
\pinlabel $f'_{23}$ [r] at 562 613

\pinlabel $f_+$ [b] at 269 396

\pinlabel $f_-$ [b] at 696 118

\endlabellist

    \centering
    \includegraphics[width=1\linewidth]{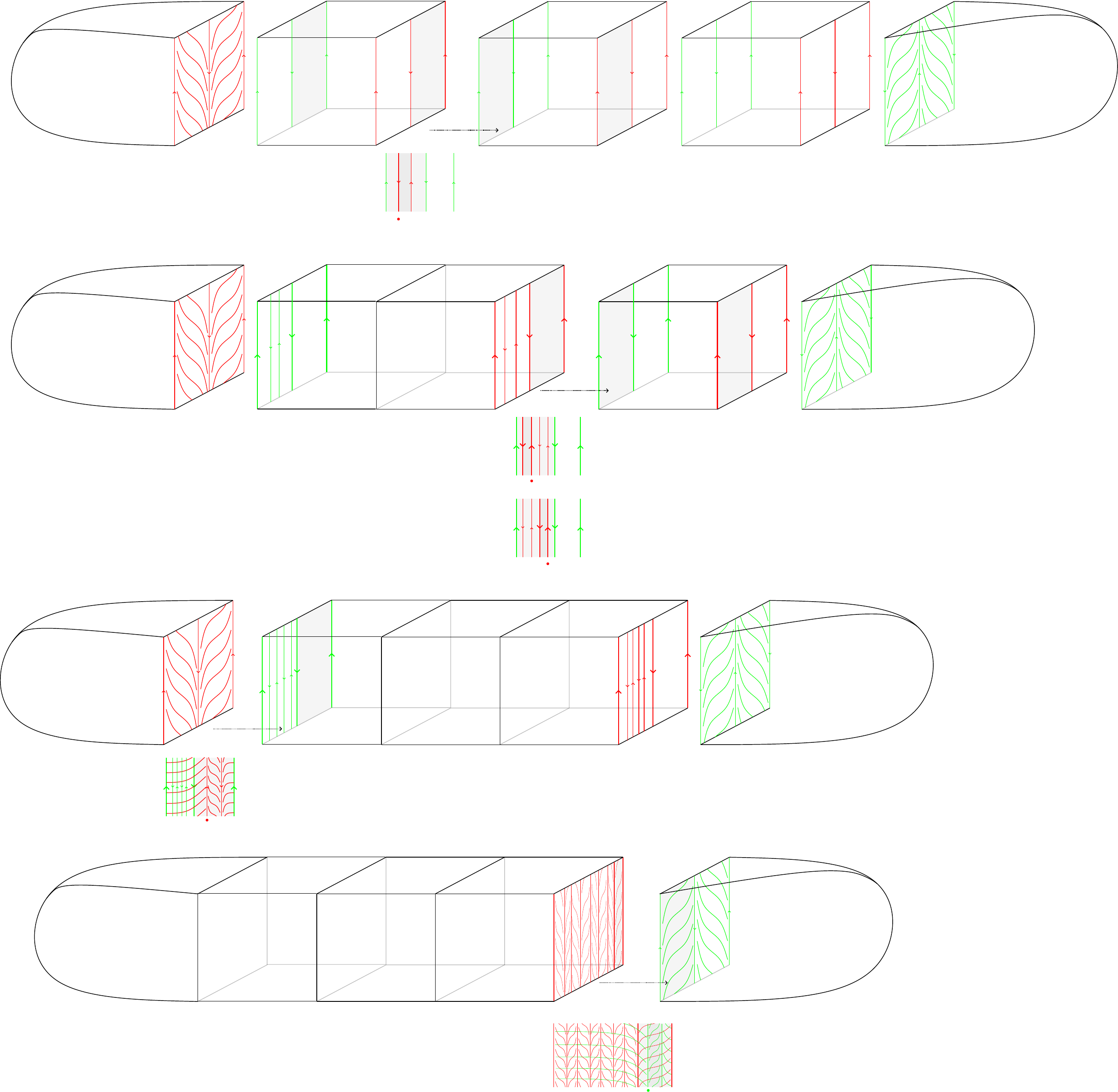}
    \caption{Gluing construction of Example \ref{ex: same bar code different complete bar codes}}
    \label{fig: different complete bar codes}
\end{figure}
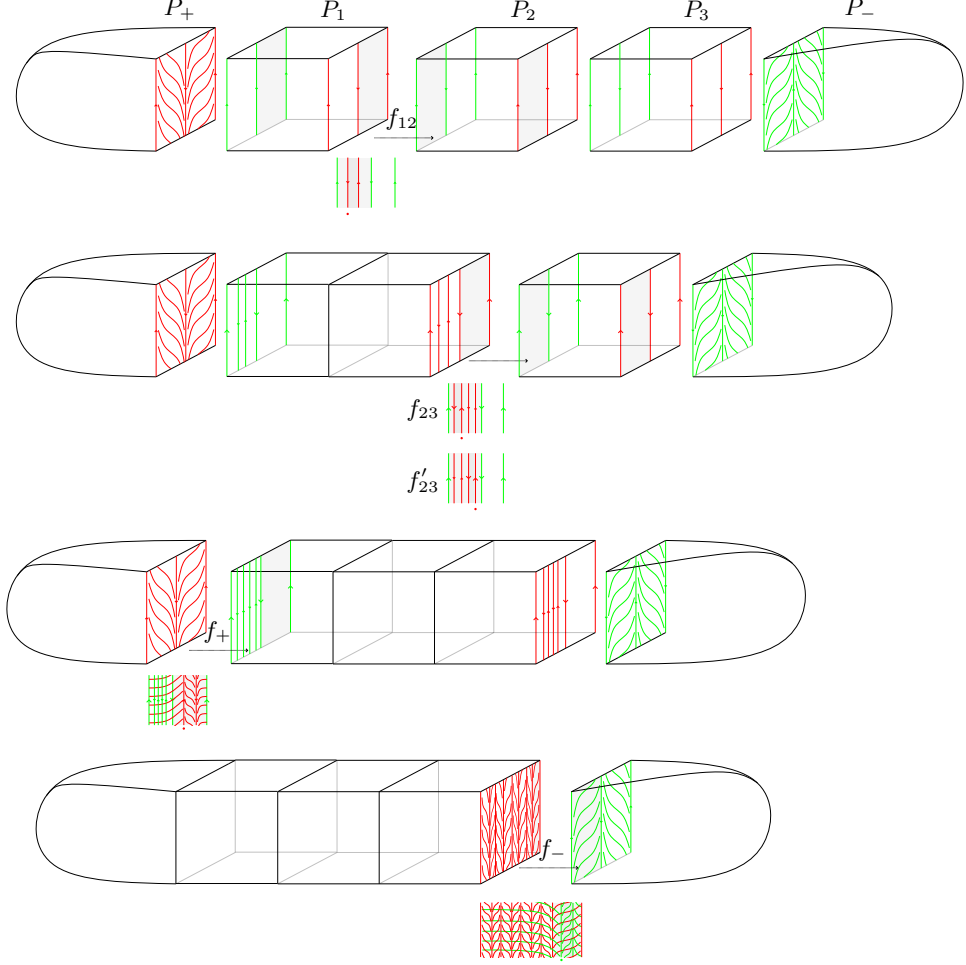

We first glue $T_1^{\mathrm{out}}$ to $T_2^{\mathrm{in}}$ by a
transverse map
$f_{12}\colon T_1^{\mathrm{out}}\to T_2^{\mathrm{in}}$
which glues the two compact leaves coming from $P_1$ inside
$A_2^{\mathrm{in}}$. As in Example~\ref{ex: non-filling bar code}, the gluing
construction of \cite{beguinBuildingAnosovFlows2017} produces a new building
block with entrance boundary $T_1^{\mathrm{in}}$ and exit boundary
$T_2^{\mathrm{out}}$.

The two compact leaves inserted in $A_2^{\mathrm{in}}$ are transported through
$P_2$ and produce two additional compact leaves in $A_2^{\mathrm{out}}$.
The induced boundary foliation on $T_2^{\mathrm{out}}$ has now four
compact leaves
$c_1,c_2,c_3,c_4$
in cyclic order, with alternating dynamical
orientations. We choose the enumeration so that $c_1$ and $c_4$ are the two
compact leaves intrinsic to the boundary prefoliation of $P_2$, while $c_2$
and $c_3$ are the additional leaves transported from $P_1$. Notice that
$c_2$ and $c_3$ lie in the annulus $A_2^{\mathrm{out}}$ bounded by $c_1$ and
$c_4$.

We now choose two transverse gluing maps
\[
f_{23},f_{23}'\colon
T_2^{\mathrm{out}}\longrightarrow T_3^{\mathrm{in}}
\]
which are isotopic. We choose them so that
$f_{23}^{-1}(\Gamma_{\mathcal K_{\partial,3}})$
consists of two compact leaves contained in the annulus bounded by
$c_1$ and $c_2$, whereas
$(f_{23}')^{-1}(\Gamma_{\mathcal K_{\partial,3}})$
consists of two compact leaves contained in the annulus bounded by
$c_3$ and $c_4$. Such maps are clearly isotopic.

The two gluings produce building blocks $(Q,\psi)$ and $(Q',\psi')$, each with
remaining boundary components $T_1^{\mathrm{in}}$ and
$T_3^{\mathrm{out}}$. We close both blocks by attaching a repelling block to
$T_1^{\mathrm{in}}$ and an attracting block to $T_3^{\mathrm{out}}$ similarly to the construction of Example \ref{ex: non-filling bar code}, using
strongly transverse gluing maps chosen in the same isotopy classes for the two
constructions. See Figure \ref{fig: comparing complete bar codes}. The gluing theorem of
\cite{beguinBuildingAnosovFlows2017} then produces two Anosov flows on the same closed
manifold.

Equivalently, we can see those two constructions as one gluing of one disconnected
pseudo-Anosov piece
\[
(P,\varphi)
=
(P_-,\varphi_-)
\sqcup
(P_1,\varphi_1)
\sqcup
(P_2,\varphi_2)
\sqcup
(P_3,\varphi_3)
\sqcup
(P_+,\varphi_+).
\]
Let $f$ and $f'$ denote the resulting pseudo-Anosov gluing maps from the two constructions. The
maps $f$ and $f'$ agree on every boundary component except for the pairing
between $T_2^{\mathrm{out}}$ and $T_3^{\mathrm{in}}$, where they are given by
the isotopic maps $f_{23}$ and $f_{23}'$, so $f$ and $f'$ are isotopic.

The two triples have the same bar code. Indeed, on
$T_2^{\mathrm{out}}$, the boundary prefoliation intrinsic to $P_2$ contains
only the leaves $c_1$ and $c_4$. In both constructions, the two compact leaves
coming from $P_3$ are placed inside the same complementary annulus
$A_2^{\mathrm{out}}$, with the same dynamical orientations and the same cyclic
order relative to $c_1$ and $c_4$.

The complete bar codes are however distinct. The boundary foliation
$\mathcal F_\partial$ also contains the additional compact leaves $c_2$ and
$c_3$. For $f$, the stable compact leaves coming from $P_3$ occur between $c_1$ and
$c_2$, whereas for $f'$ they occur between $c_3$ and $c_4$. Taking $c_1$ as
the first compact leaf, the corresponding geometric enumerations are therefore
different.

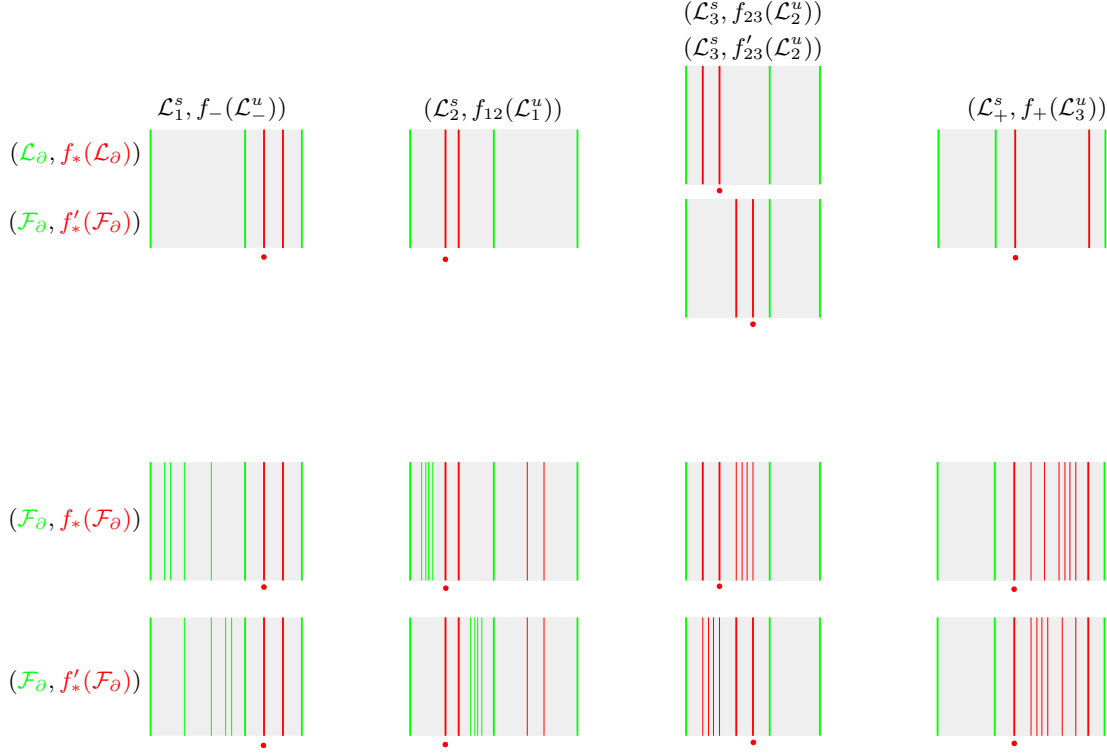
\begin{figure}[h]
\labellist
\small

\pinlabel $\cL^s_1,f_-(\cL^u_-))$ [b] at 39 333
\pinlabel $(\cL^s_2,f_{12}(\cL^u_1))$ [b] at 185 333
\pinlabel $(\cL^s_3,f'_{23}(\cL^u_2))$ [b] at 325 366
\pinlabel $(\cL^s_3,f_{23}(\cL^u_2))$ [b] at 325 385
\pinlabel $(\cL^s_+,f_+(\cL^u_3))$ [b] at 478 333

\pinlabel $(\textcolor{green}{\cF_\partial},\textcolor{red}{f_*(\cF_\partial)})$ [r] at 0 123
\pinlabel $(\textcolor{green}{\cF_\partial},\textcolor{red}{f'_*(\cF_\partial)})$ [r] at 0 36

\pinlabel $(\textcolor{green}{\cL_\partial},\textcolor{red}{f_*(\cL_\partial)})$ [r] at 0 320
\pinlabel $(\textcolor{green}{\cF_\partial},\textcolor{red}{f'_*(\cF_\partial)})$ [r] at 0 283

\endlabellist

    \centering
    \vspace{1cm}
    \includegraphics[width=1\linewidth]{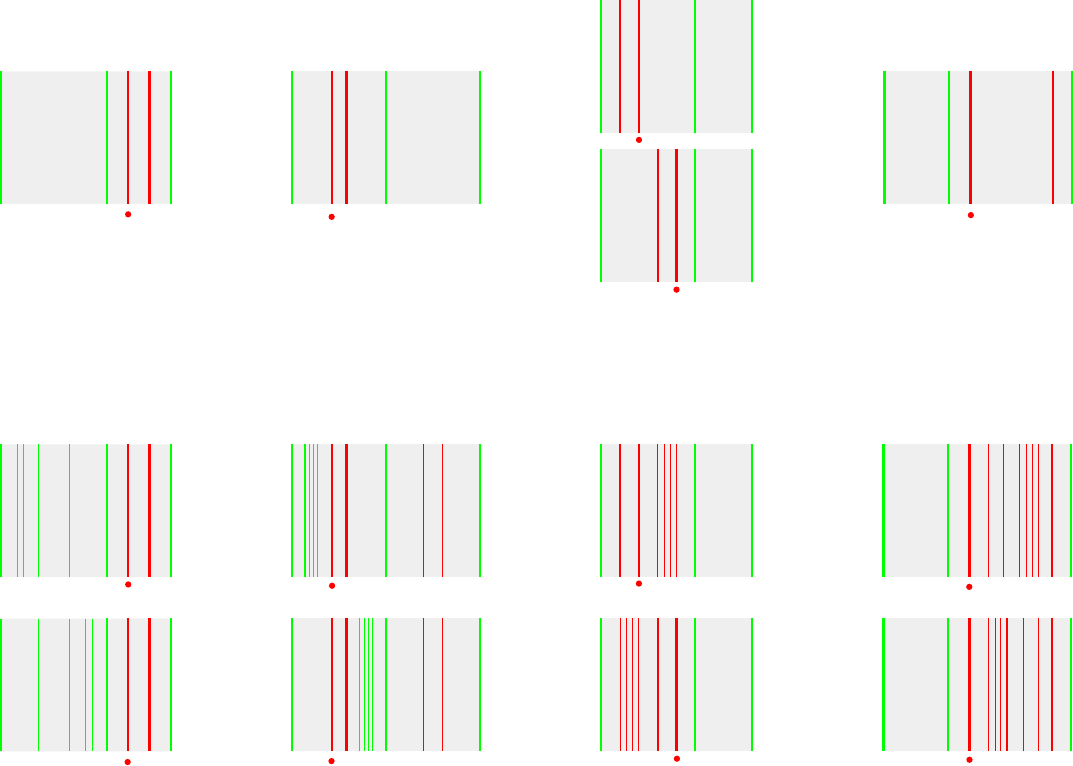}
    \caption{Comparing the two complete bar codes. The bar codes of the prefoliation are the same. The dot marks the first compact leaf of the image. }
    \label{fig: comparing complete bar codes}
\end{figure}

\end{example}

We define a notion of equivalent triple.

\begin{defi}[Equivalent triples]\label{def: equivalent triple}
Two \pa{} triples $(P_0,\varphi_0,f_0)$ and $(P_1,\varphi_1,f_1)$ are \emph{equivalent} if
\begin{enumerate}[label=(\roman*)]
    \item $(P_0, \varphi_0)$ and $(P_1, \varphi_1)$ are orbit equivalent;
    \label{def: eq triple; it: orbit eq}
    \item there exists an orbit equivalence $h:P_0 \to P_1$ such that the maps $h_*f_0$ and $f_1$ are homotopic on $\partial P_1$.
     \label{def: eq triple; it: isotopic gluing}
\end{enumerate}
\end{defi}

As a direct consequence, if $(P_0,\varphi_0,f_0)$ and $(P_1,\varphi_1,f_1)$ are equivalent triples,
then the quotient manifolds $P_0/f_0$ and $P_1/f_1$ are diffeomorphic.

\begin{rmk} \label{rmk: equivalent triple same block}
    If $(P_0, \varphi_0, f_0)$ and $(P_1, \varphi_1, f_1)$ are two equivalent triples, we can always assume that $P_0 = P_1 =P$, that $\varphi_0$ and $\varphi_1$ are isotopically equivalent on $P$, and that $f_0$ and $f_1$ are isotopic.

    Indeed, consider any diffeomorphism $g\colon P_0 \to P_1$ in the isotopy class of the orbit equivalence $h$ of Item \ref{def: eq triple; it: isotopic gluing} of Definition \ref{def: equivalent triple}.
    We replace $(P_0, \varphi_0, f_0)$ by the push-forward $(g_*P_0, g_*\varphi_0, g_*f_0) = (P_1, \varphi_0', f_0')$.
    This is a pseudo-Anosov triple inducing a pseudo-Anosov flow $\varphi_0'$ on $P_1 / f_0'$ which is orbit equivalent to the induced flow $\varphi_0$ of the pseudo-Anosov triple, where the orbit equivalence is just the quotient of the diffeomorphism $g$.
    This triple is clearly equivalent to $(P_1, \varphi_1, f_1)$ in the sense of Definition \ref{def: equivalent triple}, and satisfies that the orbit equivalence from Item \ref{def: eq triple; it: isotopic gluing} is isotopic to the identity in $P$.
\end{rmk}

\begin{rmk}\label{rmk: orbit equivalence boundary lam}
In general, an orbit equivalence between two \paps{} need not preserve the boundary foliation
associated to specific gluing maps $f_0$ and $f_1$, since the boundary foliation $\cF_\partial$ depends on the choice of gluing map.
However, such an orbit equivalence does preserve the boundary prefoliation $\cK_\partial$ (Definition~\ref{def: boundary lam and fol}),
because the boundary prefoliation is intrinsic to the piece.
\end{rmk}

If $(P,\varphi)$ is a \pap{}, let $\cO$ be the collection of periodic orbits of $\varphi$ contained in $\partial P$.

\begin{defi}[Same (complete) bar code]\label{def: same bar code}
Let $(P,\varphi_0,f_0)$ and $(P,\varphi_1,f_1)$ be two equivalent triples.
For $i=0,1$, let $\cK_{\partial,i}$ (resp.\ $\cF_{\partial,i}$) denote the boundary prefoliation (resp.\ boundary foliation) of $(P,\varphi_i,f_i)$.
\begin{itemize}
    \item We say that the triples have the \emph{same bar code} if the following holds on each boundary component:
    choose a first compact leaf $\gamma_0^0$ of $\cK_{\partial,0}$, and let $\gamma_0^1$ be the corresponding compact leaf of $\cK_{\partial,1}$
    given by the orbit equivalence between $\varphi_0$ and $\varphi_1$ (Remark \ref{rmk: orbit equivalence boundary lam}).
    Choose an orientation of the boundary component of $\partial P$. The bar codes of the pairs
    $(\cK_{\partial,0},(f_0)_*\cK_{\partial,0})$ and 
    $(\cK_{\partial,1},(f_1)_*\cK_{\partial,1})$
    are equal.
    \item We say that the triples have the \emph{same complete bar code} if the same condition holds with $\cK_{\partial,i}$ replaced by $\cF_{\partial,i}$.
\end{itemize}
\end{defi}

\begin{lem}\label{lem: carac same bar code}
The triples $(P,\varphi_0,f_0)$ and $(P,\varphi_1,f_1)$ have the same bar code if and only if the following holds.
Let $h\colon P \to P$ be an orbit equivalence between $\varphi_0$ and $\varphi_1$.
For any compact leaf $\gamma^-$ of $\cK_{\partial,0}$ on $\partial_-P$ and any compact leaf $\gamma^+$ of $\cK_{\partial,0}$ on $\partial_+ P$:
\begin{enumerate}[label=(\roman*)]
    \item $\gamma^+$ is the $k$-th compact leaf of the pair $(\cK_{\partial,0},(f_0)_*\cK_{\partial,0})$ if and only if
    $h(\gamma^+)$ is the $k$-th compact leaf of the pair $(\cK_{\partial,1},(f_1)_*\cK_{\partial,1})$ on $\partial_+P$ ;
    \item $f_0(\gamma^-)$ is the $k$-th compact leaf of the pair $(\cK_{\partial,0},(f_0)_*\cK_{\partial,0})$ if and only if
    $f_1\bigl(h(\gamma^-)\bigr)$ is the $k$-th compact leaf of the pair $(\cK_{\partial,1},(f_1)_*\cK_{\partial,1})$ on $\partial_+P$.
\end{enumerate}
\end{lem}

\begin{proof}
    This is obvious from the definition of bar code.
\end{proof}

\subsection{Return map}

Let $(M,\phi)$ be a pseudo-Anosov flow and let $\cT$ be a good collection of quasi-transverse
tori in~$M$.
Recall that $\phi$ comes with stable and unstable (possibly singular) foliations
$(\cF^s,\cF^u)$, and we denote by
\[
\cF_\cT^s:=\cF^s\cap \cT,
\qquad
\cF_\cT^u:=\cF^u\cap \cT
\]
their traces on~$\cT$.
In this subsection, we will describe the behavior of the return map of the flow on $\cT$, and more precisely near the boundary of its domain of definition.

This will later allow us to control how orbit segments travel between specific regions of the
lifts of~$\cT$ in the universal cover.

We denote by $\Theta$ the first-return map on $\cT$. It is defined on the set of points $x\in \cT$ whose positive $\phi$-orbit meets $\cT$ again.

\begin{lem}\label{lem: domain of theta}
There exist sublaminations $\cK_\cT^s\subset \cF_\cT^s$ and $\cK_\cT^u\subset \cF_\cT^u$ such that the first return map
$\Theta$ is a homeomorphism
\[
\Theta\colon \cT\setminus \cK_\cT^s \longrightarrow \cT\setminus \cK_\cT^u .
\]
\end{lem}

\begin{rmk}\label{rmk:KsKu_vs_boundarylam}
The pair $(\cK_\cT^s,\cK_\cT^u)$ lives on $\cT\subset M$ and should not be confused with the boundary prefoliation
$\cK_\partial$ of a pseudo-Anosov piece (Definition~\ref{def: boundary lam and fol}), which lives on the boundary of a piece.
If $(M,\phi)$ is obtained by gluing a pseudo-Anosov triple $(P,\varphi,f)$ and $\cT=\pi(\partial P)$,
then $\cK_\partial$ and $f_*\cK_\partial$ project to laminations on $\cT$ which, after exchanging stable and unstable traces on the
appropriate transverse components of $\cT$, recover precisely $\cK_\cT^s$ and $\cK_\cT^u$.
\end{rmk}

\begin{proof}
Cut $(M,\phi)$ open along $\cT$ and let $(P,\varphi)$ be the resulting pseudo-Anosov piece.
Let $\cK_\partial$ be the boundary prefoliation of $(P,\varphi)$ (Definition~\ref{def: boundary lam and fol}) on $\partial P$,
and let $\pi\colon \partial P\to \cT$ be the natural projection induced by the gluing.

Set $\cK_\cT = \pi(\cK_\partial)$ for the induced lamination on $\cT$, and define
\[
\cK_\cT^s := \cK_\cT\cap (\cF^s\cap \cT),
\qquad
\cK_\cT^u := \cK_\cT\cap (\cF^u\cap \cT).
\]
By definition of the boundary prefoliation (see also Remark~\ref{rmk: dynamic of the boundary prefoliation}),
a point $x\in \cT$ lies in $\cK_\cT^s$ if and only if its positive $\phi$-orbit never meets $\cT$ again, and
$x\in \cK_\cT^u$ if and only if its negative $\phi$-orbit never meets $\cT$ again.

Therefore $\Theta$ is defined exactly on $\cT\setminus \cK_\cT^s$ and its image is exactly $\cT\setminus \cK_\cT^u$.
Moreover, away from finitely many periodic orbits where $\cT$ is only quasi-transverse, $\cT$ is genuinely transverse
to $\phi$; on each transverse component the return time is continuous, hence $\Theta$ restricts to a homeomorphism
$\cT\setminus \cK_\cT^s \to \cT\setminus \cK_\cT^u$.
\end{proof}

\medskip

We now pass to the universal cover $\widetilde M\simeq \R^3$.
Let $\widetilde \phi$ be the lift of $\phi$, and let $\widetilde \cT$ be the complete lift of the good
collection~$\cT$.

The next claim ensures that $\widetilde \cT$ behaves like a collection of ``walls'' in $\R^3$:
even though distinct lifts may meet along lifts of periodic orbits (coming from weak embeddedness in $M$),
they cannot cross transversely.

\begin{claim}\label{claim:liftedT_planes_noncrossing}
The collection $\widetilde \cT$ is a collection of properly embedded planes in $\widetilde M$.
Moreover, the planes are \emph{non-crossing}: two distinct planes of $\widetilde \cT$ cannot intersect transversely.
\end{claim}

\begin{proof}
Each $T\in\cT$ is incompressible, hence every lift $\widetilde T\subset \widetilde M$ is a plane.

We first show that $\widetilde T$ is embedded.
By weak embeddedness, the only self-intersections of $T$ occur along finitely many periodic orbits $O_1, \dots, O_n$ of $\phi$
contained in~$T$.
Suppose that some lift $\widetilde T$ has a self-intersection in $\widetilde M$ along a lift of a periodic orbit $\tilde O_i$.
Then one can find a simple closed curve $\eta\subset \widetilde T$ which is transverse to (say) the lifted stable
foliation $\widetilde \cF^s$.
Projecting to $M$, this gives a null-homotopic closed transversal to $\cF^s$, contradicting Novikov's theorem.
Hence $\widetilde T$ is embedded.

We now prove non-crossing.
Since $\cT$ is a {good collection} (Definition~\ref{def: weakly embedded}), for each $T_i\in\cT$ there exists an
embedded torus $T_i'$ homotopic to $T_i$ such that the collection $\cT'=\{T_1',\dots,T_n'\}$ is embedded and pairwise disjoint.
Lift the homotopies to $\widetilde M$.
Each plane $\widetilde T_i$ lies at bounded distance from a corresponding lift $\widetilde T_i'$.
Since the $\widetilde T_i'$ are disjoint planes in $\R^3$, no two of them cross.
If two planes $\widetilde T,\widetilde S$ in $\widetilde \cT$ crossed transversely, then any planes at bounded distance
and homotopic to them would still intersect, contradicting disjointness of the lifted collection $\widetilde{\cT'}$.
\end{proof}

\medskip

Denote by $\widetilde{\cO}$ the complete lift of the periodic orbits contained in $\cT$, and by
$\widetilde{\cO}_a\subset \widetilde{\cO}$ the complete lift of the alternating ones (Subsection~\ref{subsec: foliation induced}).
Since $\widetilde O\in \widetilde{\cO}$ projects to a periodic orbit in $M$, its stabilizer in $G:=\pi_1(M)$ is infinite cyclic;
we will call such orbits \emph{periodic} in this equivariant sense.

Let $\tilde \Theta\colon \widetilde \cT\to \widetilde \cT$ be the lifted first return map.
Denote by $(\widetilde \cF_\cT^s,\widetilde \cF_\cT^u)$ the lifts of $(\cF_\cT^s,\cF_\cT^u)$ to $\widetilde \cT$, and similarly
$(\widetilde \cK_\cT^s,\widetilde \cK_\cT^u)$ the lifts of $(\cK_\cT^s,\cK_\cT^u)$. By Lemma \ref{lem: domain of theta}, we get:

\begin{fact}\label{fact: domain of lift of theta}
The map $\widetilde\Theta$ is a homeomorphism from
$\widetilde \cT\setminus \widetilde \cK_\cT^s$ to $\widetilde \cT\setminus \widetilde \cK_\cT^u$.
\end{fact}

\begin{defi}[Periodic leaf, periodic bands, and nonperiodic bands]\label{def: periodic leaf}
A leaf of $\widetilde \cF_\cT^s$ (resp.\ $\widetilde \cF_\cT^u$) whose stabilizer in $G=\pi_1(M)$ is nontrivial
(hence infinite cyclic) is called \emph{periodic}.

By the description of $\cK_\cT^s$, each half of any \emph{nonperiodic} leaf of $\widetilde \cK_\cT^s$ is asymptotic
(for a fixed $G$-invariant metric on $\widetilde M$) to a periodic leaf of $\widetilde \cK_\cT^s$.
Consequently, every connected component of $\widetilde \cT\setminus \widetilde \cK_\cT^s$ is either
\begin{itemize}
\item a band bounded by periodic leaves, called a \emph{periodic band}, or
\item a band bounded by two nonperiodic leaves that are asymptotic to each other at both ends,
called a \emph{nonperiodic band}.
\end{itemize}
The analogous statement holds for $\widetilde \cK_\cT^u$.
\end{defi}
Note that a periodic band has an infinite cyclic stabilizer in $G$, whereas a nonperiodic band has trivial stabilizer.

From now on, fix a connected component $C$ of $\widetilde \cT\setminus \widetilde \cK_\cT^s$.

\begin{defi}[Periodic boundary]\label{def: periodic boundary}
The \emph{periodic boundary} of $C$ is the union of the periodic leaves of $\widetilde \cF_\cT^s$ which are asymptotic to $C$.
\end{defi}

We refer to Figure \ref{fig: periodic boundary}.

\begin{figure}[h]
\labellist
\small

\pinlabel $C$ at 112 48

\endlabellist

    \centering
    \includegraphics[width=0.5\linewidth]{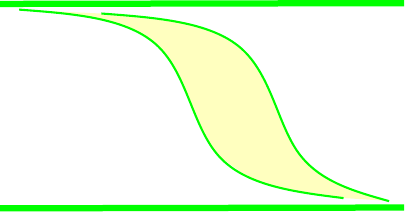}
    \caption{Periodic boundary of a non-periodic band $C$ in bold green.}
    \label{fig: periodic boundary}
\end{figure}

\begin{rmk}\label{rmk: periodic boundary}
Since $\widetilde \cK_\cT^s$ is a closed sublamination of $\widetilde \cF_\cT^s$, any periodic leaf of $\widetilde \cF_\cT^s$
contained in $\overline C$ must lie in $\widetilde \cK_\cT^s$.
If $C$ is a periodic band, its periodic boundary lies in the accessible boundary of $C$.
If $C$ is a nonperiodic band, the periodic leaves given by the asymptotics lie in the (non-accessible) ideal boundary of~$C$.
\end{rmk}

\begin{defi} \label{def: maximal band}
    We say that $B$ is a maximal band if it is a periodic band such that its interior only contains compact leaves of either $\widetilde \cF_\cT^s$ (in which case we call it a \emph{$s$-maximal band}) or $\widetilde \cF_\cT^u$ (called a \emph{$u$-maximal band}), and there is no periodic band strictly containing it.
\end{defi}

Notice that it follows that the boundary of an $s$-maximal band $B$ consists of periodic leaves of $\widetilde \cF_\cT^u$ or orbits of the flow, and reciprocally for a $u$-maximal band. 

\begin{lem}[Action of the lifted return map]\label{lem: action lift return map transverse segment}
Let $C$ be a connected component of $\widetilde \cT\setminus \widetilde \cK_\cT^s$ and set
$C'=\widetilde\Theta(C)$, a connected component of $\widetilde \cT\setminus \widetilde \cK_\cT^u$.
Let $\sigma^u\subset C$ be a half-open segment contained in a single leaf of $\widetilde \cF_\cT^u$,
with its non-compact end on $\partial C$.
Then $\widetilde\Theta(\sigma^u)\subset C'$ is a half-leaf of $\widetilde \cF_\cT^u$, and its non-compact end
is asymptotic to a periodic boundary leaf of~$C'$.
\end{lem}

\begin{proof}
Write $\widetilde\Theta(x)=\widetilde \phi^{\tau(x)}(x)$ on its domain, where $\tau(x)>0$ is the first return time of $x$
to $\widetilde \cT$.
Since $\widetilde\Theta$ is obtained by flowing along $\widetilde \phi$, it preserves the unstable foliation,
so $\widetilde\Theta(\sigma^u)$ is contained in a single leaf of $\widetilde \cF_\cT^u$ and is a connected half-open segment in~$C'$.

Let $p\in \partial C\cap \overline{\sigma^u}$ be the endpoint on $\partial C$.
By definition, $\partial C\subset \widetilde \cK_\cT^s$, and points of $\widetilde \cK_\cT^s$ are precisely those whose
positive orbit never meets $\widetilde \cT$ again.
By continuity of the return time on~$C$, we have $\tau(x)\to +\infty$ as $x\to p$ in~$\sigma^u$.

In particular, $\widetilde\Theta(\sigma^u)$ has infinite length inside its unstable leaf, hence is a half-leaf.
Finally, $\partial C'\subset \widetilde \cK_\cT^u$ by definition of $C'$.
By the structure of $\widetilde \cK_\cT^u$, any such nonperiodic unstable half ray in $C'$ accumulates on (and is asymptotic to)
a periodic leaf of $\widetilde \cK_\cT^u$, i.e.\ on a periodic boundary leaf of $C'$.
\end{proof}

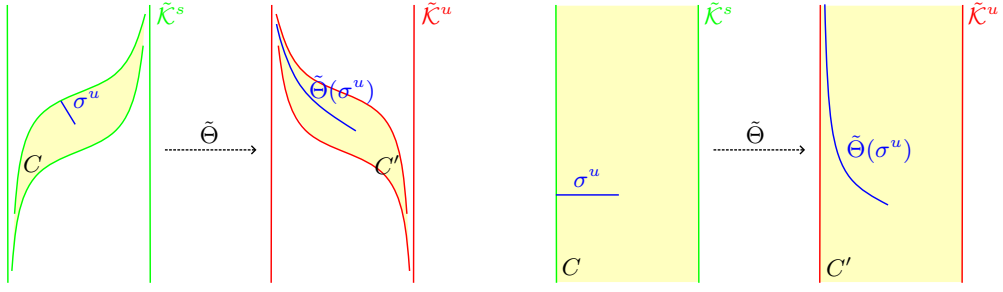
\begin{figure}[h]
\labellist
\small

\pinlabel $\textcolor{blue}{\sigma^u}$ [bl] at 42 117
\pinlabel $\textcolor{blue}{\sigma^u}$ [b] at 401 61
\pinlabel $\textcolor{blue}{\tilde\Theta(\sigma^u)}$ [bl] at 205 120
\pinlabel $\textcolor{blue}{\tilde\Theta(\sigma^u)}$ [bl] at 576 79

\pinlabel $C$ at 18 82
\pinlabel $C'$ at 263 82

\pinlabel $C$ at 390 11
\pinlabel $C'$ at 575 11 

\pinlabel $\tilde\Theta$ [b] at 140 91
\pinlabel $\tilde\Theta$ [b] at 517 91

\pinlabel ${\color{green}\tilde\cK^s}$ [l] at 98 186
\pinlabel ${\color{red}\tilde\cK^u}$ [l] at 281 186
\pinlabel ${\color{green}\tilde\cK^s}$ [l] at 476 186
\pinlabel ${\color{red}\tilde\cK^u}$ [l] at 658 186

\endlabellist

    \centering
    \includegraphics[width=\linewidth]{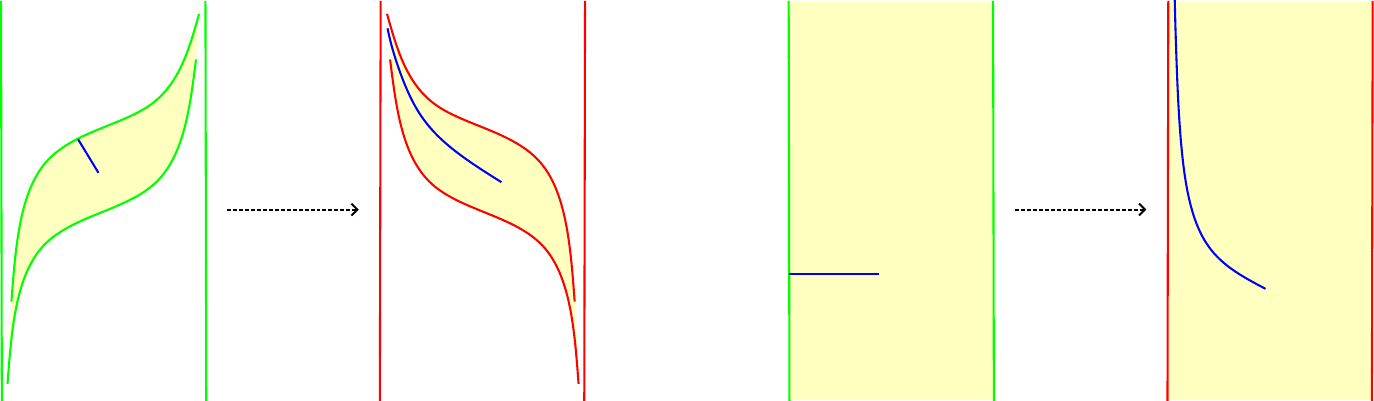}
    \caption{Action of the return map $\widetilde \Theta$ on a segment $\sigma^u$ of $\widetilde \cF_\cT^u$ in a non-periodic band (left) and in a periodic band (right).}
    \label{fig: return map transverse segment}
\end{figure}

The same argument yields a correspondence between accessible boundary leaves of $C$ and periodic boundary leaves of $C'$:

\begin{claim}\label{claim: corresponding nbh to side of periodic boundary}
For any accessible boundary leaf $\gamma$ of $C$, there exists a periodic boundary leaf $\gamma_p$ of $C'$ such that
if $(x_n)$ is any sequence in $C$ converging to a point in $\gamma$, then $\widetilde \Theta(x_n)$ accumulates on $\gamma_p$ in $C'$.
\end{claim}

\section{Free homotopy data of blocks}

\label{sec: free data block}

The goal of this section is to prove Theorem~\ref{thmintro: tool thm for free homotopy data}.
Roughly speaking, it provides a criterion ensuring that two \pafs{} on a given manifold $M$
have the same free homotopy data globally, provided that their free homotopy data agree \emph{piecewise}
on a block decomposition and that the dynamics along the boundary tori match in the appropriate sense.
The main tool is a partial order on certain subchains (lines) inside the chains of lozenges associated to
the boundary tori.
This order is designed to encode which orbit segments can travel from one transverse lift to another in the universal cover,
and therefore to detect fixed points of deck transformations in the orbit space, which is exactly what records free homotopy classes of periodic orbits.

\subsection{Trace of quasi-transverse tori}\label{subsec: trace qt tori}

In the previous section, we described the geometry of the lifts $\widetilde \cT$ inside $\widetilde M$ in terms of
bands (periodic/nonperiodic bands, maximal bands, sides of periodic boundary leaves).
The goal of the present subsection is to translate this ``band picture'' into the orbit space.
This translation is crucial for the proof of Theorem~\ref{thmintro: tool thm for free homotopy data}:
the partial order introduced later is defined purely in the orbit space, and it is specifically designed to
detect when an orbit segment can travel from one transverse lift to another.

Let $\cQ_\phi$ be the orbit space of the lifted flow $\widetilde\phi$ and let $(\cQ^s,\cQ^u)$ be the induced stable and unstable bifoliation on~$\cQ_\phi$.
We denote by $\pi\colon \widetilde M\to M$ and $p\colon \widetilde M\to \cQ_\phi$ the projections.
See Section~\ref{sec: preli} for the setup.
We use the \emph{adjacency type} of consecutive lozenges in a chain to indicate whether they share exactly one corner, a stable side, or an unstable side.
We refer to Definitions~\ref{def: lozenge} and~\ref{def: chain lozenges} for lozenges and chains of lozenges.

\begin{defi}[Minimal chain]\label{def:minimal_chain}
Let $\cC=\{L_i\}_{i\in\Z}$ be a bi-infinite chain of lozenges in $\cQ_\phi$.
We say that $\cC$ is \emph{minimal} if it contains no proper bi-infinite subchain of lozenges.
\end{defi}

Equivalently, for every $i$, the three consecutive lozenges $L_{i-1},L_i,L_{i+1}$ do not all share a common corner
(see \cite[Lemma~3.2.3]{barthelmePseudoAnosovFlowsPlane2025}).

\begin{defi}[Line of lozenges]\mbox{}\label{def: line lozenge}
A \emph{line of lozenges} is a (finite or infinite) chain of lozenges $\cL=\{L_j\}$ such that any two consecutive lozenges share a side of the same foliation. More precisely
\begin{itemize}
    \item A \emph{$u$-line of lozenges} is a line $\cL$ such that any two consecutive lozenges share an \emph{unstable} side.
    \item A \emph{$s$-line of lozenges} is a line $\cL$ such that any two consecutive lozenges share a \emph{stable} side.
\end{itemize}
\end{defi}

Notice that $\cL$ is a $u$-line if and only if every stable leaf crossing one lozenge of $\cL$
crosses every lozenge of $\cL$; equivalently, $\cL$ is contained in the stable saturation of any lozenge $L\in\cL$\footnote{Note that this is the opposite convention to the one chosen in \cite{barbotPseudoAnosovFlowsToroidal2013a}, i.e., their $s$-scalloped bi-infinite chain of lozenges would correspond to a bi-infinite $u$-line for us.}.
Similarly for $s$-lines (with stable and unstable swapped).
Also, with this convention a single lozenge is both a $u$-line and an $s$-line.

\begin{defi}[Maximal line inside a chain]\label{def: maximal line}
Let $\cC$ be a chain of lozenges in $\cQ_\phi$.
\begin{enumerate}
    \item If $\cL\subset \cC$ is a line of lozenges, we say that $\cL$ is \emph{maximal in $\cC$} if there is no line of lozenges $\cL'\subset\cC$ with $\cL\subsetneq \cL'$.
    \item A lozenge $L\in \cL$ is \emph{extremal} if it is the first or last lozenge of the (indexed) line $\cL$.
    \item A corner $p$ in a maximal line $\cL$ is \emph{extremal} if it is the first or last corner of the line
    (for the natural indexing where $p_i$ is the unique corner common to $L_i$ and $L_{i+1}$).
\end{enumerate}
\end{defi}

Since a lozenge in a minimal chain is contained in at most two maximal lines of lozenges, we immediately obtain:

\begin{claim}[Decomposition of a minimal chain into maximal lines]\label{claim: chain decompo lines}
Let $\cC=\{L_i\}_i$ be a minimal chain of lozenges.
Either $\cC$ is itself a bi-infinite line, or there is a unique decomposition
$\cC=\bigcup_j \cL_j$
into maximal lines of lozenges $\cL_j$, where each consecutive pair $\cL_j,\cL_{j+1}$
meets either at an extremal corner or along an extremal lozenge.
\end{claim}

\begin{figure}[h]
\labellist
\small

\pinlabel {\colorbox{jauneclair}{$\cL_1$}} at 130 95
\pinlabel {\colorbox{orangeclair}{$\cL_2$}} at 190 156
\pinlabel {\colorbox{roseclair}{$\cL_3$}} at 331 250

\endlabellist

    \centering
    \includegraphics[width=0.5\linewidth]{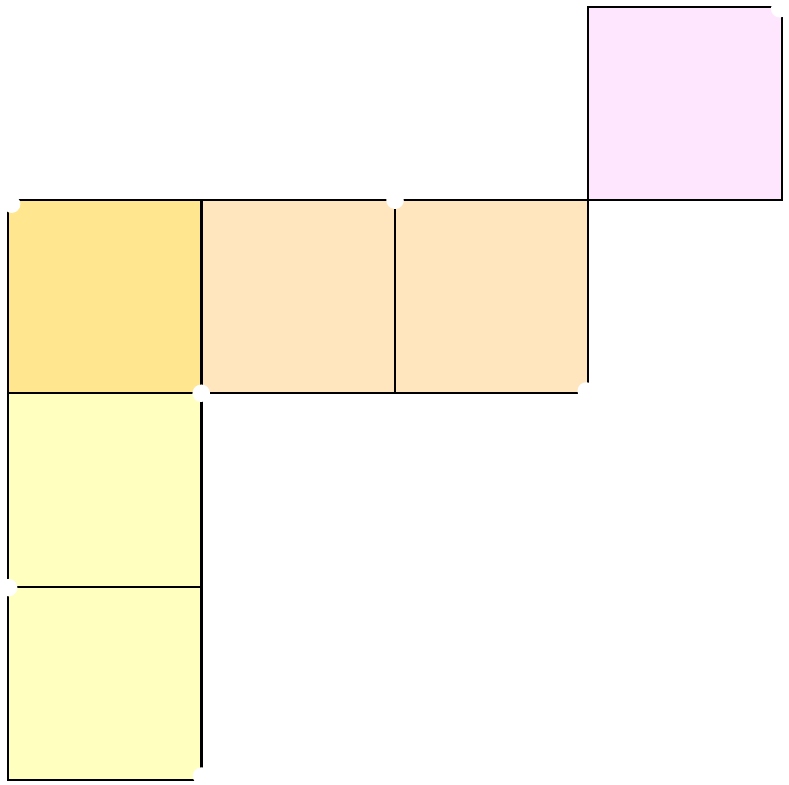}
    \caption{Decomposition of a minimal chain into maximal lines of lozenges.}
    \label{fig: line in chain}
\end{figure}

\medskip

We now return to the good collection $\cT$ of quasi-transverse tori in $M$, its complete lift $\widetilde \cT$ in $\widetilde M$,
and its \emph{trace} in the orbit space which is the collection of chains
\[
\cC := p(\widetilde \cT) \subset \cQ_\phi.
\]
By \cite[Proposition~5.3.5]{barthelmePseudoAnosovFlowsPlane2025}, $\cC$ is a collection of minimal chains of lozenges,
finite up to the $\pi_1(M)$-action.
We now describe how the periodic band decomposition of $\widetilde \cT$ translates in $\cQ_\phi$.

Let $B\subset \widetilde \cT$ be a periodic band.
If $B$ has no periodic leaves in its interior, then $B$ projects to a single lozenge in $\cQ_\phi$;
we call such a band \emph{elementary}.
If $B$ is a union of $n$ elementary bands $B_1,\dots,B_n$ glued along a periodic boundary leaves
(for instance along a periodic unstable leaf),
then its projection is a $u$-line of $n$ lozenges sharing unstable sides.
In particular we obtain:

\begin{claim}\label{claim: cylinder and lines}
The projection of any maximal periodic band of $\widetilde \cT$ to $\cQ_\phi$ is a maximal line of lozenges in the trace $\cC$.
\end{claim}

\subsection{Lines and order} \label{subsec_lines_order}

We now introduce a partial order on lines of lozenges.
This order generalizes the partial orders used in \cite{BM24,BFM25},
and it serves the same purpose: controlling the free homotopy data by detecting when deck transformations have fixed points in appropriate regions of the orbit space.

\begin{defi}[Simple collection of chains]\label{def:simple_collection}
A collection $\cC$ of chains of lozenges is \emph{simple} if for every corner $c$ of a lozenge in $\cC$,
the $\pi_1(M)$-orbit of $c$ does not meet the interior of any lozenge of $\cC$.
\end{defi}

\begin{defi}[Order on simple lines]\label{def:order_lines}
Let $\cL$ and $\cL'$ be two simple lines of lozenges.
We say that $\cL \prec \cL'$ if
$\cL \subset \cQ^u(\cL')$
and
$\cL' \subset \cQ^s(\cL).$
\end{defi}

The relation $\prec$ is clearly transitive and $\pi_1(M)$-invariant.
Moreover, two lines are comparable for $\prec$ if and only if they have a Markovian intersection.

Using the non-corner criterion (see \cite[Lemma~2.4.7]{barthelmePseudoAnosovFlowsPlane2025}),
one checks the following case-by-case description (see Figure~\ref{fig: markov intersection}):

\begin{lem}\label{lem: line intersection type}
Suppose $\cL \prec \cL'$. Then:
\begin{enumerate}
    \item\label{lem: line intersection; it: u_u}
    If both $\cL$ and $\cL'$ are $u$-lines, then there exists a lozenge $L'\in \cL'$ such that $\cL \prec L'$.
    \item\label{lem: line intersection; it: s_s}
    If both $\cL$ and $\cL'$ are $s$-lines, then there exists a lozenge $L\in \cL$ such that $L \prec \cL'$.
    \item\label{lem: line intersection; it: s_u}
    If $\cL$ is an $s$-line and $\cL'$ is a $u$-line, then there exist unique lozenges $L\in \cL$ and $L'\in \cL'$ such that $L \prec L'$.
    \item\label{lem: line intersection; it: u_s}
    If $\cL$ is a $u$-line and $\cL'$ is an $s$-line, then for every $L\in \cL$ and $L'\in \cL'$ we have $L \prec L'$.
\end{enumerate}
\end{lem}

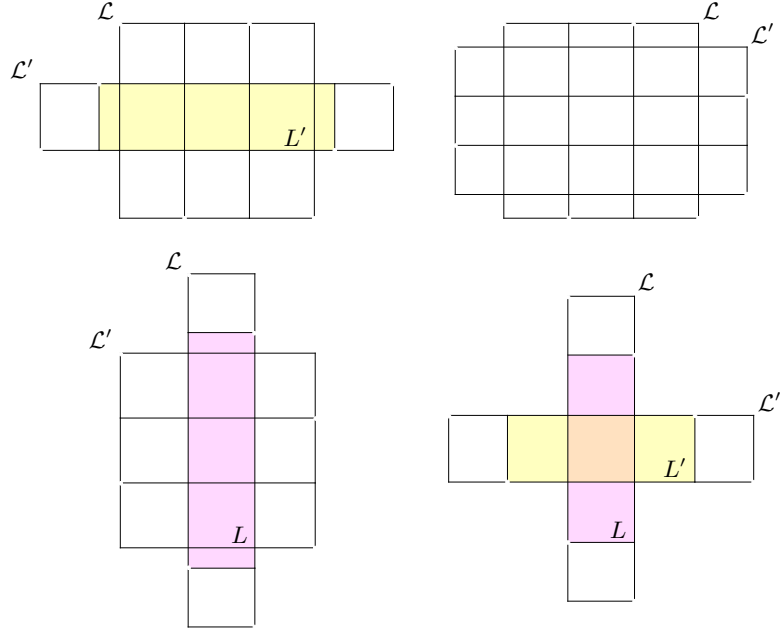
\begin{figure}[h]
\labellist
\small

\pinlabel $\cL$ [br] at 93 691
\pinlabel $\cL'$ [br] at 2 623

\pinlabel $\cL$ [bl] at 753 691
\pinlabel $\cL'$ [bl] at 808 664

\pinlabel $\cL$ [br] at 171 405
\pinlabel $\cL'$ [br] at 92 315

\pinlabel $\cL'$ [bl] at 816 243
\pinlabel $\cL$ [bl] at 679 381

\pinlabel $L'$ [br] at 314 546
\pinlabel $L$ [br] at 246 92

\pinlabel $L$ [br] at 679 98
\pinlabel $L'$ [br] at 748 167

\endlabellist

    \centering
    \includegraphics[width=0.75\linewidth]{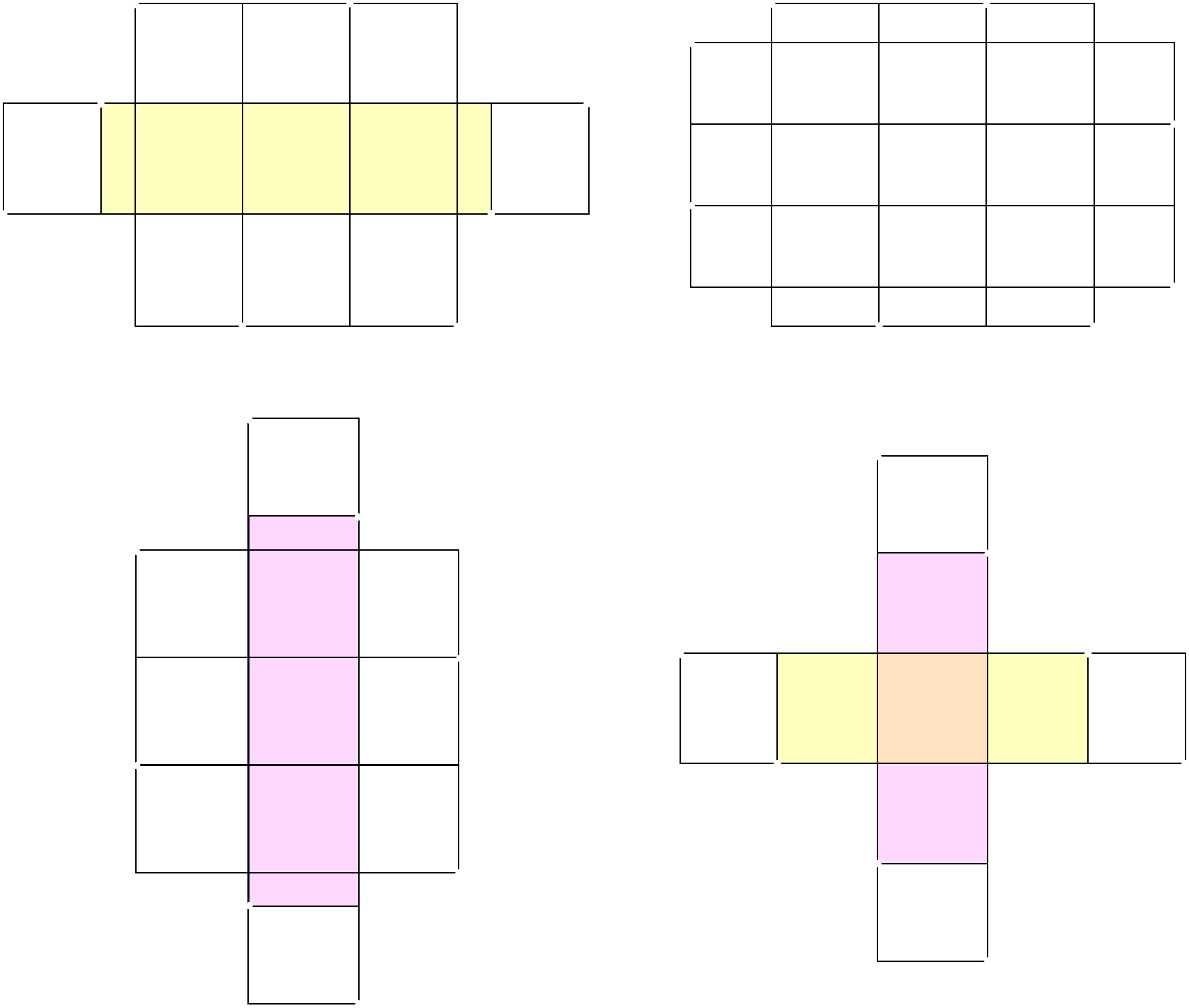}
    \caption{The four possible Markovian intersection types for lines $\cL \prec \cL'$. Here $\cQ^u$ is the vertical foliation and $\cQ^s$ is the horizontal foliation. Top left: Case \ref{lem: line intersection; it: u_u}. Bottom left: Case \ref{lem: line intersection; it: s_s}.Bottom right: Case \ref{lem: line intersection; it: s_u}. Top right: Case \ref{lem: line intersection; it: u_s}.}
    \label{fig: markov intersection}
\end{figure}

In particular, the order on lines is equivalent to the existence of a comparable pair of lozenges:

\begin{lem}\label{lem: order line equiv order lozenge}
Let $\cL,\cL'$ be simple lines of lozenges.
Then $\cL \prec \cL'$ if and only if there exist lozenges $L\in\cL$ and $L'\in\cL'$ such that $L \prec L'$.
\end{lem}

\begin{proof}
The forward implication follows from Lemma~\ref{lem: line intersection type}.
For the converse, assume $L \prec L'$.

If $\cL$ is an $s$-line, then $\cQ^u(\cL)=\cQ^u(L)$ and $\cQ^s(L)\subset \cQ^s(\cL)$.
If $\cL'$ is a $u$-line, then
\[
\cL \subset \cQ^u(L) \subset \cQ^u(L') \subset \cQ^u(\cL'),
\qquad
\cL' \subset \cQ^s(L') \subset \cQ^s(L) \subset \cQ^s(\cL),
\]
hence $\cL \prec \cL'$.
If $\cL'$ is an $s$-line, the non-corner criterion implies that $\cL'\cap \cL=\cL'\cap L$,
and the same conclusion follows.

If instead $\cL$ is a $u$-line, then the non-corner criterion forces every lozenge of $\cL$ to have Markovian intersection with $L'$,
which again yields $\cL \prec \cL'$.
\end{proof}

\medskip

The key point is that the relation $\prec$ admits a dynamical reformulation in $\widetilde M$.
Since $p\colon \widetilde M\to \cQ_\phi$ is a locally trivial fibration, to any line of lozenges $\cL$ one can associate a (generally non-proper) topological plane $B\subset \widetilde M$ transverse to $\widetilde\phi$ whose projection to $\cQ_\phi$ is the interior of $\cL$.
The boundary of $B$ consists of the orbits corresponding to the extremal corners of~$\cL$.
See \cite{barbotMisePositionOptimale1995} or \cite[Chapter~6]{barthelmePseudoAnosovFlowsPlane2025} for details.

We can now give the dynamical reformulation of the order $\prec$: 

\begin{prop}\label{prop: order line and orbit}
Let $\cL$ and $\cL'$ be two lines of lozenges such that the collection $\{\cL,\cL'\}$ is simple, and let $B,B'\subset \widetilde M$ be transverse lifts of $\cL,\cL'$.
Then $\cL \prec \cL'$ if and only if there exists a positive orbit segment of $\widetilde\phi$ from $B$ to $B'$
(i.e., an orbit segment meeting $B$ first and $B'$ later). 
\end{prop}

Note that this result generalizes to our context \cite[Lemma 6.2]{BFM25} (which treats the case of $\cL$ and $\cL'$ being scalloped regions) and \cite[Observation 3.14]{BM24} (which treats the case of $\cL$ and $\cL'$ being single lozenges). The proof is essentially the same in every case and we quickly recall the argument and refer to \cite[Lemma 6.2]{BFM25} for more details.

\begin{proof}
By Lemma \ref{lem: order line equiv order lozenge}, we have that $\cL \prec \cL'$ if and only if $L\prec L'$ for lozenges $L\in \cL$ and $L'\in \cL'$.

In particular, if $L\prec L'$, $L\cap L' \neq \emptyset$ and therefore, there exists an orbit intersecting both $B$ and $B'$. Since the flow (eventually) contracts the stable direction and expands the unstable, the condition $L \subset \cQ^u(L')$ implies that the orbit segment from $B$ to $B'$ must be in the positive direction.

Now, assume instead that there exists a positive segment of an orbit intersecting $B$ and $B'$ in this order. Then there are lozenges $L\in \cL$ and $L'\in \cL'$ such that $L\cap L' \neq \emptyset$. Now, as the collection $\{\cL,\cL'\}$ is assumed to be simple, it implies that the corners of $L$ are not contained in $L'$ and neither are the corners of $L'$ in $L$. Thus the intersection of $L$ and $L'$ is Markovian, and once again, hyperbolicity of the flow implies that $L \subset \cQ^u(L')$, thus $L\prec L'$ and $\cL\prec \cL'$.\qedhere

\end{proof}

As a corollary that will be essential later, we can now recognize the existence of some periodic orbits using this partial order:
\begin{coro} \label{cor: characterization fixed point with line}
Let $\cC$ be a simple chain of lozenges in an orbit space $\cQ$, and let $\cL\subset \cC$ be a maximal line of lozenges.
For $g\in \pi_1(M)$, the following are equivalent:
\begin{enumerate}
    \item $g$ fixes a point in $\cL$;
    \item $\cL \prec g\cL$ or $\cL \prec g^{-1}\cL$.
\end{enumerate}
\end{coro}

\begin{proof}
Assume first that $\cL \prec g\cL$. As in the proof of Proposition \ref{prop: order line and orbit}, by Lemma~\ref{lem: line intersection type}, there exists a lozenge $L\subset \cL$ such that $L\prec gL$, which implies that $L$ and $gL$ have a Markovian intersection, and hence a fixed point in $L$.

Conversely, if $g$ fixes a point $x$ in a lozenge $L\subset \cL$, then Proposition \ref{prop: order line and orbit} implies that $\cL \prec g\cL$ or $g\cL \prec \cL$.
\end{proof}

\subsection{Proof of Theorem \ref{thmintro: tool thm for free homotopy data}}

\label{subsec: free data block}

In this section we recall the notion of free homotopy data of a flow and how it translates in the orbit space: periodic orbits correspond to conjugacy classes in $\pi_1(M)$ which acts in the orbit space with fixed points.
The chains of lozenges coming from quasi-transverse tori provide a convenient place to detect these fixed points via the order relation~$\prec$ introduced above.

Recall from the introduction the following definitions:
\begin{defi}[Free homotopy data]\label{def: free homotopy data}
Let $\phi$ be a pseudo-Anosov flow on a closed $3$-manifold $M$.
\begin{itemize}
    \item The \emph{free homotopy data} of $\phi$ is the subset
\[
\Per(\phi)
=\Bigl\{\gamma\in \pi_1(M)\ \Bigm|\ \gamma \text{ or } \gamma^{-1}
\text{ is represented by a periodic orbit of }\phi\Bigr\}.
\]
\item If $\cT$ is a good collection of quasi-transverse tori, we denote by
$\Per \bigl(\res{\phi}{M\setminus \cT}\bigr)$
the subset of free homotopy classes represented by periodic orbits of $\phi$ whose orbits do \emph{not} cross $\cT$ transversely (i.e.\ they are contained in a component of $M\setminus \cT$).
\end{itemize}
\end{defi}

\medskip

For $i=1,2$, let $\phi_i$ be Anosov flows on $M$, with stable and unstable foliations
$(\cF^s_i,\cF^u_i)$, and let $\cT_i$ be good collections of quasi-transverse tori for~$\phi_i$.
Denote by $\cQ_i$ the orbit spaces of $\widetilde\phi_i$ and by $\cC_i\subset \cQ_i$
the traces of $\tilde\cT_i$, i.e., the corresponding $\pi_1(M)$-invariant collections of minimal chains of lozenges that correspond to the projections of $\tilde\cT_i$ to $\cQ_i$.

\begin{defi}[Isomorphism of minimal chains]\label{def: chain isomorphism}
Let $\cC_1=\{L_{1,k}\}_{k\in\Z}$ and $\cC_2=\{L_{2,k}\}_{k\in\Z}$ be two bi-infinite minimal chains of lozenges in $\cQ_1$ and $\cQ_2$ respectively.
An \emph{isomorphism of minimal chains} is a bijection $\bar \xi\colon \cC_1\longrightarrow \cC_2$ satisfying:
\begin{enumerate}[label=(\roman*)]
    \item there exists $k_0\in \Z$ such that $\bar \xi (L_{1,k})= L_{2,k+k_0}$, and
    \item  $\bar \xi$ preserves the adjacency type of consecutive lozenges, i.e., $\bar \xi (L_{1,k})$ and $\bar \xi (L_{1,k+1})$ share a stable (resp.~unstable) side if and only if $L_{1,k}$ and $L_{1,k+1}$ share a stable (resp.~unstable) side.
\end{enumerate}
\end{defi}

\begin{lem}\label{lem: isomorphism of chain}
Suppose that $\xi\colon M\to M$ is a homeomorphism homotopic to the identity such that:
\begin{enumerate}
    \item $\xi(\cT_1)=\cT_2$, and
    \item $\xi$ maps compact leaves of $\cF^{s}_1\cap \cT_1$ to compact leaves of $\cF^{s}_2\cap \cT_2$,
    and compact leaves of $\cF^{u}_1\cap \cT_1$ to compact leaves of $\cF^{u}_2\cap \cT_2$.
\end{enumerate}
Then $\xi$ induces, for each $T\in \cT_1$, an isomorphism $\bar \xi$ of minimal chains between the trace of $T$ in $\cQ_1$ and the trace of $\xi(T)$ in $\cQ_2$.
\end{lem}

\begin{rmk}
More explicitly, enumerate $\cT_1=\{T_{1,k}\}_k$ and $\cT_2=\{T_{2,k}\}_k$ so that $\xi(T_{1,k})=T_{2,k}$.
Let $\cC_{1,k}$ (resp.\ $\cC_{2,k}$) be the trace of $T_{1,k}$ (resp.\ $T_{2,k}$) in $\cQ_1$ (resp.\ $\cQ_2$).
Then the induced chain isomorphism
$\bar \xi_k\colon \cC_{1,k}\to \cC_{2,k}$
is characterized as follows:
$\bar \xi_k$ sends a lozenge $L\in \cC_{1,k}$ to the lozenge $L'\in \cC_{2,k}$ if and only if $\xi$ sends the two consecutive compact leaves of $\cF^{s/u}_1\cap T_{1,k}$ that bound $L$ to the two consecutive compact leaves of $\cF^{s/u}_2\cap T_{2,k}$ that bound $L'$.
This description applies to both parallel and scalloped components (Definition~\ref{def: parallel scalloped}).
In the parallel case, having an isomorphism of minimal chains is equivalent to saying that the corresponding bar codes are equal for a consistent enumeration of the compact leaves.
\end{rmk}

\begin{proof}
Since $\xi$ is homotopic to the identity, we can choose a lift $\widetilde \xi\colon \widetilde M\to \widetilde M$ that commutes with the $\pi_1(M)$-action.
Then $\widetilde \xi$ sends the complete lift $\widetilde \cT_1$ to $\widetilde \cT_2$.
By assumption, $\widetilde \xi$ sends lifts of periodic orbits in $\widetilde \cT_1$ to lifts of periodic orbits in $\widetilde \cT_2$, and it preserves whether a periodic leaf comes from the stable or unstable trace.
In particular, $\widetilde \xi$ sends elementary periodic bands in $\widetilde \cT_1$ to elementary periodic bands in $\widetilde \cT_2$.
Projecting to the orbit spaces, this yields a bijection between lozenges in the trace chains.

Because $\xi$ is homotopic to the identity, it preserves the cyclic order of compact leaves on each component of $\cT_i$.
Therefore the induced correspondence respects the $\Z$-indexing up to translation.
Finally, since $\xi$ sends stable compact leaves to stable compact leaves and unstable compact leaves to unstable compact leaves, it preserves the adjacency type of consecutive lozenges.
Hence we obtain a chain isomorphism in the sense of Definition~\ref{def: chain isomorphism}.
\end{proof}

\medskip

We now have all the terminology needed to state and prove Theorem \ref{thmintro: tool thm for free homotopy data} that we restate here.

\toolthm*

\begin{rmk}
The collections of chains $\cC_i$ are simple, because they come from a good collection of tori.
In particular, the relation $\prec$ is well-defined on maximal lines in the trace.
\end{rmk}

The argument reduces the comparison of free homotopy data to detecting fixed points of deck transformations in the orbit space.

\begin{proof}[Proof of Theorem~\ref{thmintro: tool thm for free homotopy data}]

Let $g\in\pi_1(M)$ represent a periodic orbit of $\phi_1$.
If the orbit does not cross $\cT_1$ transversely, then by Assumption~\ref{thm:tool_assumption_piecewise}
the same free homotopy class is represented by a periodic orbit of $\phi_2$.

It remains to consider the periodic orbits of $\phi_1$ that cross $\cT_1$ transversely.
Such an orbit corresponds to a deck transformation $g$ fixing a point in the trace $\cC_1$ in $\cQ_1$.
By Corollary~\ref{cor: characterization fixed point with line}, there exists a maximal line $\cL\subset \cC_1$ such that
$\cL\prec g^{\pm 1}\cL$; replacing $g$ by $g^{-1}$ we may assume $\cL\prec g\cL$.
Let $\cL'=\bar \xi(\cL)$ be the corresponding maximal line in the trace $\cC_2\subset \cQ_2$.
By Assumption~\ref{thm:tool_assumption_order}, $\bar \xi$ preserves $\prec$, hence $\cL'\prec g\cL'$.
Applying Corollary~\ref{cor: characterization fixed point with line} in $\cQ_2$, we deduce that $g$ fixes a point in $\cL'$,
so $g$ represents a periodic orbit of $\phi_2$ as well.
Thus every free homotopy class represented by a periodic orbit of $\phi_1$ is also represented by one of $\phi_2$.
By symmetry, the reverse inclusion holds, proving $\Per(\phi_1)=\Per(\phi_2)$.
\end{proof}

This theorem implies the following corollary giving a large class of self-orbit equivalences for pseudo-Anosov flows on toroidal manifolds (the reader can compare with \cite[Theorem 1.4]{BM24} or \cite[Section 11]{BFP23}).

\begin{coro} \label{coro: 3D dehn twist}
   Let $T$ be a quasi-transverse torus for a transitive pseudo-Anosov flow $\phi$ on $M$. Let $\alpha$ be the free homotopy class of a compact leaf in $T$.
    Then a (3-dimensional) Dehn twist along $T$ in the direction of $\alpha$ can be represented by a self-orbit equivalence of $\phi$.
\end{coro}

\begin{proof}
Let $h$ be a representative of the Dehn twist which preserves the set of compact leaves on $T$ and is the identity outside of a small neighborhood of $T$.
Note that $h$ leaves invariant the set of free homotopy classes that can be made disjoint from a neighborhood of $T$.

Let $\phi_1 = \phi$ and $\phi_2 = h \circ \phi \circ h^{-1}$, and let $\xi \colon M \to M$ be the identity. Then we can apply Theorem \ref{thmintro: tool thm for free homotopy data} to $\phi_1,\phi_2$ and $\xi$.
By transitivity, we deduce that $\phi_1$ and $\phi_2$ are orbit equivalent, which concludes the proof.
\end{proof}

\section{Finiteness of block gluing}

\label{sec: block gluing}

Given Theorem \ref{thmintro: tool thm for free homotopy data}, we will now deduce Theorem \ref{thmintro: uniqueness of gluing}. We start by proving Proposition \ref{propintro: finiteness of bar code}, that we restate here.

\finitenessbarcode*

\begin{rmk}\label{rmk: finiteness not uniqueness}
This proposition is sharp in the sense that the complete bar code is
\emph{not} uniquely determined by the equivalence class of a triple.
As an example, for every $n\geq 1$, B\'eguin, Bonatti and Yu construct in
\cite[Section~11]{beguinBuildingAnosovFlows2017} a fixed building block
$(P,\phi)$
and $n$ gluing maps
$f_1,\ldots,f_n,$
all belonging to the same isotopy class, so that the triples $(P, \varphi, f_k)$ are all equivalent, but they induce Anosov flows $\psi_k$ which are all pairwise non-orbit equivalent (see proof of \cite[Theorem 1.13]{beguinBuildingAnosovFlows2017} for more details).

The distinction between these flows is visible in the combinatorics of the compact leaves under the gluing maps: the pairs
$\bigl(\cK_\partial,(f_k)_*\cK_\partial\bigr)$
have pairwise distinct bar codes (see Figure \ref{fig: different bar codes}).

\begin{figure}[h]
\labellist
\small
\pinlabel $f_i$ [l] at 165 122

\pinlabel $\textcolor{red}{\cL_\partial^-}$ [l] at 147 202
\pinlabel $\textcolor{green}{\cL_\partial^+}$ [l] at 147 44

\pinlabel $(P,\varphi,f_1)$ [b] at 276 170
\pinlabel $(P,\varphi,f_2)$ [b] at 464 170
\pinlabel $(P,\varphi,f_3)$ [b] at 652 170

\tiny 

\pinlabel $(+---+-+-++)$ [t] at 280 69
\pinlabel $(+-+---+-++)$ [t] at 463 69
\pinlabel $(+-+-+---++)$ [t] at 648 69

\endlabellist

    \centering
    \includegraphics[width=1\linewidth]{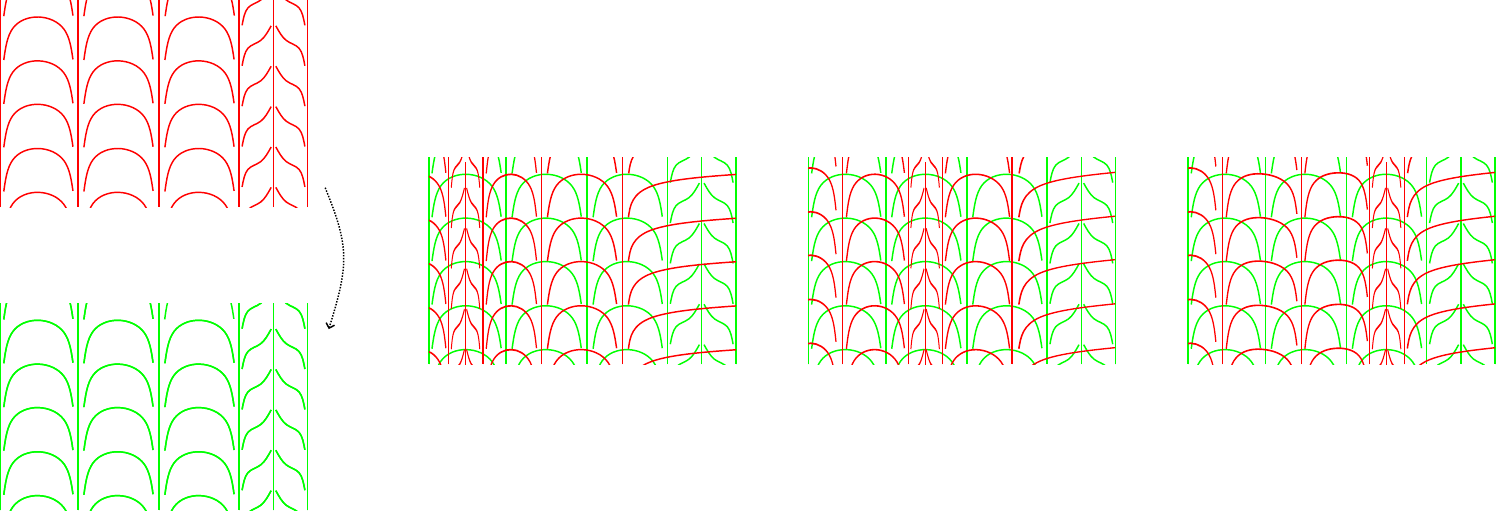}
    \caption{For $n=3$ we have 3 equivalent triples with 3 distinct bar codes.}
    \label{fig: different bar codes}
\end{figure}
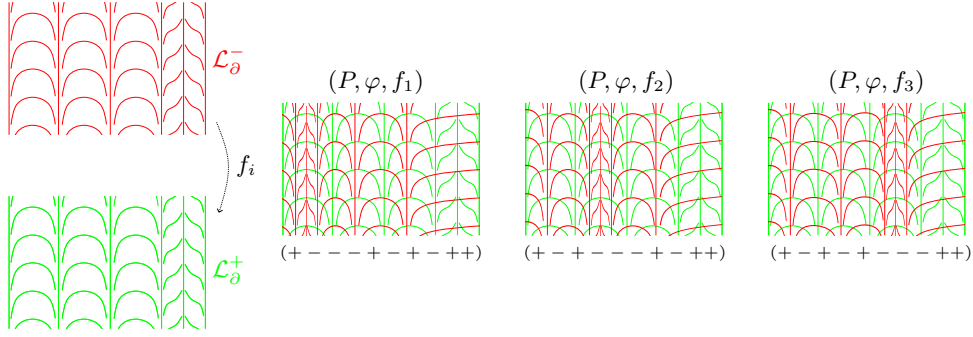

Since the boundary is transverse and the boundary laminations are filling,
the bar code agrees with the complete bar code by
Proposition \ref{prop: bar code and complete bar code for filling piece}.
Thus, for every $n$, there exists an equivalence class of triples containing
at least $n$ distinct complete bar codes and at least $n$ distinct orbit equivalence class of induced pseudo-Anosov flows.
\end{rmk}

\begin{proof}[Proof of Proposition~\ref{propintro: finiteness of bar code}]
Let $(P,\varphi,f)$ be a pseudo-Anosov triple, and let $\cF_\partial$ and $\cK_\partial$ denote respectively
the boundary foliation and the boundary prefoliation of the triple.

By Definition~\ref{def: pA bar code}, the complete bar code is determined by the combinatorics of the compact leaves of the pair
$(\cF_\partial,f_*\cF_\partial)$ on each parallel boundary component.
By definition of equivalent triple, the number of compact leaves of the boundary
prefoliation $\cK_\partial$ is the same by the correspondence induced by the orbit equivalence of the pieces.
Hence the only difference in bar codes comes from the compact leaves of the boundary
foliation $\cF_\partial$ which do {not} belong to $\cK_\partial$.
We will prove that the number of extra compact leaves is uniformly bounded.

Let $M:=P/f$ and $\phi$ the induced pseudo-Anosov flow on $M$. Up to taking a finite cover, which does not impact the result, we will assume, using Proposition \ref{prop: cover embedded jsj tori}, that $M$ is orientable and all the Seifert pieces in the JSJ decomposition of $M$ are products of an orientable surface with $S^1$.
Denote by $\pi\colon P\to M$ the quotient map and by $\cT:=\pi(\partial P)$ the image of the boundary.
The pair of foliations $(\cF_\partial,f_*\cF_\partial)$ projects to the pair $(\cF_\cT^+,\cF_\cT^-)$ on $\cT$, and the pair $(\cK_\partial,f_*\cK_\partial)$ projects to
a pair of sublaminations $(\cK_\cT^+,\cK_\cT^-)$ of $(\cF_\cT^+,\cF_\cT^-)$.
Consider a compact leaf $\gamma$ of $\cF_\cT^+\setminus \cK_\cT^+$ on, say, the stable foliation $\cF^s$. Recall that $\cK_T^+\cup \cK_T^-$ can be decomposed as the union $\cK_T^s \cup \cK_T^u$, where $\cK_T^s$ is the trace of the stable lamination and $\cK_T^u$ is the trace of the unstable lamination on $T$.
Since $\gamma\not\subset \cK_\cT^s$, the positive $\phi$-orbit of $\gamma$ crosses $\cT$ again.
Let $\Theta$ be the first return map from $\cT\setminus \cK_\cT^s$ to $\cT \setminus \cK_\cT^u$ given by Lemma~\ref{lem: domain of theta}.
Then the iterates
$\gamma,\ \Theta(\gamma),\ \Theta^2(\gamma),\dots$
are compact leaves of the trace of $\cF^s$, all contained in the same stable cylindrical leaf of $\cF^s$.

We first claim that this forward orbit is finite.
Indeed, otherwise as $\Theta$ is contracting along the stable trace, the lengths of these curves
tend to $0$. But they are all homotopically nontrivial simple closed curves, in particular their lengths are bounded
from below by a positive constant depending only on the topology of $M$, a contradiction.
So there exists $n\geq 0$ such that
$\gamma_n:=\Theta^n(\gamma)$
is a compact leaf of $\cK_\cT^s$.
Thus every compact leaf of $\cF_\cT^+\setminus \cK_\cT^+$ is associated with a compact leaf of $\cK_\cT^s$: the first reached by forward iteration under $\Theta$.

Fix now a compact leaf $c\subset \cK_\cT^s$ on a component $T\subset \cT$, and consider all compact leaves
of $\cF_\cT^+\setminus \cK_\cT^+$ which lie in the same stable cylindrical leaf of $\cF^s$ as $c$ and whose forward
$\Theta$-orbit ends at $c$.
We claim that the number of such leaves is uniformly bounded.
Suppose first that the stable cylindrical leaf containing $c$ meets the same torus $T$
in two distinct compact leaves $\gamma_i,\gamma_j$.
Let $A_{i,j}$ be the annulus in that stable cylindrical leaf bounded by $\gamma_i$ and $\gamma_j$,
and let $A^T_{i,j}\subset T$ be the annulus in $T$ bounded by the same two curves.
Then
\[
T_{i,j}:=A_{i,j}\cup A^T_{i,j}
\]
is an immersed incompressible torus in $M$.
Since it meets the JSJ torus $T$ essentially, it follows from the JSJ decomposition that
$T_{i,j}$ and $T$ lie in the same Seifert piece of $M$ and are not JSJ tori.
This Seifert piece is not free: by the description of quasi-transverse
tori in free Seifert pieces (see \cite{barbotFreeSeifertPieces2021}), the compact leaves on such a torus are exactly the periodic orbits of the flow,
hence there are no extra compact leaves coming from $\cF_\partial\setminus \cK_\cT$.
Thus $T$ must lie in a periodic Seifert piece $Q$ of $M$.

Let $Z\subset Q$ be the spine of $Q$ and let $\cG\subset \Sigma$ be the associated fatgraph
given by the projection of $Z$ to the base surface (see Section \ref{subsec: spine}).
The torus $T$ projects to a simple closed curve
\[
\alpha:=p(T)\subset \Sigma,
\]
whose isotopy class is determined by the isotopy class of the triple $(P, \varphi, f)$.
Now let $L$ be a stable cylindrical leaf of $\cF^s$ contributing extra compact leaves on $T$.
Its projection to $\Sigma$ is an embedded arc $\delta$ joining either one boundary component of $\Sigma$ to another boundary component, or one boundary component of $\Sigma$ to a periodic orbit of the spine.
Such an arc $\delta$ crosses at most once a vertex of the fatgraph.
Hence $\delta$ is determined, up to isotopy rel.\ endpoints, by the topology of the fatgraph.
Since the isotopy class of the boundary prefoliation is fixed inside the isotopy class of the triple,
the combinatorics of these projected arcs is fixed as well.
Hence there are only finitely many possibilities for the isotopy class of $\delta$.

For each such arc $\delta$, the number of intersections of $\alpha$ (which is fixed in the isotopy class of the triple) with $\delta$ is bounded by a constant
depending only on the topology of the fatgraph.
Therefore the number of compact leaves of the stable cylinder $L$ intersecting $T$ is uniformly bounded.
This proves that, for each compact leaf $c$ of $\cK_\cT^s$, only finitely many compact leaves of
$\cF_\cT^+\setminus \cK_\cT^+$ can be associated with $c$.
The same argument applies to compact leaves of $\cF_\cT^-\setminus \cK_\cT^-$ using unstable cylinders.
Since the total number of compact leaves is uniformly bounded inside the isotopy class of the triple, only finitely many complete bar codes
can occur.
This proves Proposition~\ref{propintro: finiteness of bar code}.
\end{proof}

We can now prove Theorem \ref{thmintro: uniqueness of gluing}, i.e., the fact that the complete bar code uniquely determines the free homotopy data.

\uniquenessgluing*

\begin{proof}[Proof of Theorem \ref{thmintro: uniqueness of gluing}]
Let $(P_0, \varphi_0, f_0)$ and $(P_1, \varphi_1, f_1)$ be two equivalent \pa{} triples.
By Remark \ref{rmk: equivalent triple same block} we can assume that $P_0 = P_1 = P$, the flows are isotopically equivalent and the gluing maps are isotopic.
Set $M_i=P/f_i$ for $i=0,1$.
Denote by $f_t \colon \pP \to \pP$ the isotopy between $f_0$ and $f_1$, and $\pi_i\colon P \to M_i$ the projection ($i=0,1$).

Let $\phi_i$ be the \pa{} flow on $M_i$ obtained from $\varphi_i$, and by $(\cF^s_i, \cF^u_i)$ the stable and unstable foliations of $\phi_i$ on $M_i$.
Consider $\cT_i = \pi_i (\partial P)$. It is a good collection of quasi-transverse tori for $\phi_i$.
Denote by $\pi_i^\pm$ the restriction of the projection $\pi_i$ to $\partial_\pm P$:\begin{equation} \label{eq: pi+ and pi-}
\begin{aligned}
         &\pi_i^+ \colon \partial_+P \to \cT_i \\
     &\pi_i^- \colon \partial_-P \to \cT_i 
\end{aligned}
\end{equation} Note that $\pi_i^\pm$ is a homeomorphism onto its image $\cT_i$.

\begin{lem} \label{lem: H diff of pieces}
There exists a diffeomorphism $H\colon P\to P$ isotopic to the identity such that
\begin{enumerate}
    \item $\res{H}{\partial_+P} = \text{id} $, \label{claim: H, it: H 1}
    \item $\res{H}{\partial_-P} = f_1^{-1} \circ f_0\colon \partial_-P \to \partial_-P $, \label{claim: H, it: H 2}
\end{enumerate}
In particular, $H\circ f_0=f_1\circ H$ on $\partial P$, hence $H$ descends to a diffeomorphism
$\overline H\colon M_0 \to M_1$ mapping $\cT_0$ to $\cT_1$.
\end{lem}

\begin{proof}
Assume first that $\partial P$ is an embedded surface (i.e., that $P$ is a smooth manifold, not one with corners).
Choose disjoint collar neighborhoods
\[
\cV^\pm \cong \partial_\pm P\times[0,1]\subset P,
\qquad \partial_\pm P \cong \partial_\pm P\times\{0\},
\]
so that $P\setminus(\cV^+\cup \cV^-)$ is a compact manifold with boundary.
Define $H\colon P\to P$ by
\begin{itemize}[label=-]
\item for $(x,t)\in \cV^+$, set $H(x,t)=(x,t)$;
\item for $(x,t)\in \cV^-$, set
$H(x,t)=\bigl( f_{1-t}^{-1}\circ f_0(x),\,t\bigr)\in \cV^-$;
\item on $P\setminus(\cV^+\cup \cV^-)$, set $H=\text{id}$.
\end{itemize}
We check that these definitions glue.
On $\cV^+$ there is no overlap with $\cV^-$ by construction.
On the interface $(x,1)\in \partial_\pm P\times\{1\}$ we have
$H(x,1)=\bigl(f_{0}^{-1}\circ f_0(x),\,1\bigr)
      = (x,1),$
 hence $H$ agrees with the identity on the boundary of the collar, and therefore
glues smoothly with the interior definition.
On $\partial_+P$ (i.e.\ $t=0$ in $\cV^+$) we get $H=\text{id}$.
On $\partial_-P$ (i.e.\ $t=0$ in $\cV^-$) we get
$H(x,0)=f_1^{-1}\circ f_0(x)$,
which proves Item \ref{claim: H, it: H 1} and Item \ref{claim: H, it: H 2}.
This map is clearly isotopic to the identity on $P$.

Assume now that $\partial P$ has corners, i.e. the formal boundary surfaces meet along
periodic orbits. Let $\cO_c \subset \partial P$ be the union of those corner periodic orbits.
Choose a small tubular neighborhood $\cN(\cO_c)$ in a continuation $\tilde P$ of $P$, a disjoint union of solid tori
$S^1\times D^2$, small enough that $\cN(\cO_c)$ meets $\partial P$ in annuli.
Set
\[
P^\circ := P\setminus \mathrm{int}\,\cN(\cO_c).
\]
Then $\partial P^\circ$ is an embedded surface (the corner intersections have been removed), and
the construction of Step~1 applies to produce a diffeomorphism
$H^\circ\colon P^\circ \longrightarrow P^\circ$
satisfying the boundary prescriptions on $\partial_\pm P\cap P^\circ$.

It remains to extend $H^\circ$ over each solid torus component of $\cN(\cO_c)$.
$H^\circ$ induces a diffeomorphism
between the boundary tori of corresponding components of $\cN(\cO_c)$ and $\cN(h(\cO_c))$ which preserves
the longitude, therefore extends to a diffeomorphism of the solid tori.
By construction, the conditions coming from the $\partial_+$ and $\partial_-$ sides coincide on the core orbit, meaning that if a corner orbit $O \in \cO_c$ is at the intersection of a $\partial_+ P$ and $\partial_- P$ component, the isotopy $f_t$ is the identity on $\cO_c$.
It follows that we can do this extension so that it satisfies the boundary condition inside $\cN(O_c)$.
Performing this on every component, we obtain a diffeomorphism
$H\colon P\to P$ isotopic to the identity, extending $H^\circ$, satisfying conditions \ref{claim: H, it: H 1} and \ref{claim: H, it: H 2}.

We have on $\partial_+P$ that $H\circ f_0 = f_1\circ H$,
and on $\partial_-P$ the same equality holds.
Hence $H\circ f_0=f_1\circ H$ on $\partial P$, so $H$ descends to a diffeomorphism
$\overline H\colon P/f_0 \to P/f_1$, i.e., $\overline H\colon M_0\to M_1$.
\end{proof}

Up to composing with the map from Lemma \ref{lem: H diff of pieces}, 
we can, and will, replace $\phi_0$ by its push-forward by the diffeomorphism $\bar H$ on $M_1$.
From now on, we still denote by $\phi_0$ the flow $\bar H_* \phi_0$ on $M=M_1$ and $\cT = \cT_1$.

\begin{rmk} \label{rmk: push on M and isotopic leaves}
Note that $\cT$ is a good collection of quasi-transverse tori for \emph{both} $\phi_0$ and $\phi_1$.
Moreover, by equivalence of the triples we know that the compact leaves of the foliation associated to $\phi_0$ and $\phi_1$ on $\cT$ are isotopic.
Without loss of generality, and to simplify notations, we can assume that the periodic orbits or $\phi_0$ and $\phi_1$ contained in $\cT$ coincide, and denote them by $\cO$.
\end{rmk}

Our goal will now be to construct the map $\xi$ from Theorem~\ref{thmintro: tool thm for free homotopy data}. 
We denote $(\cF^s_0, \cF^u_0)$ and $(\cF^s_1, \cF^u_1)$ the stable and unstable \pa{} foliations for $\phi_0$ and $\phi_1$ on $M$.
We denote $(\cF_{\cT,0}^s, \cF_{\cT,0}^u)$ and $(\cF_{\cT,1}^s, \cF_{\cT,1}^u)$ the trace on $\cT$. 
Recall by Remark \ref{rmk: push on M and isotopic leaves} that the compact leaves of each pair are assumed to be the same.

The gluing maps $f_0$ and $f_1$ realize the (complete) same bar codes, which means that there are geometric enumerations of the boundary foliations $\cF_{\partial,0}$ and $\cF_{\partial,1}$ on $\pP$ and $\pP$, compatible with the \pap{} orbit equivalence such that the pair $(\cF_{\partial,0}, (f_0)_*\cF_{\partial,0})$ on $\pP$ and $(\cF_{\partial,1}, (f_1)_*\cF_{\partial,1})$ have the same bar codes on the parallel components (see Definition \ref{def: same bar code}).

Denote by $h$ the orbit equivalence between $\varphi_0$ and $\varphi_1$ isotopic to the identity on $P$.
Define $h^\iin$ the restriction of $h$ to $P^\iin$ and $h^\out$ the restriction of $h$ to $P^\out$.

Let $\Gamma_i \colon P^\iin \to P^\out$ be the crossing map for the flow of $\varphi_i$ on $P$. 
Recall that $\cK_\partial$ denote the boundary prefoliation (Definition \ref{def: boundary lam and fol}).

\begin{claim} \label{claim: equivalence on boundary}
    The map $h^\iin$ maps $ \cK_{\partial,0} \cap  P^\iin$ to $ \cK_{\partial,1} \cap  P^\iin$ and $h^\out$ maps $ \cK_{\partial,0} \cap P^\out$ to $\cK_{\partial,1} \cap P^\out$.
    It commutes with the crossing maps:
    $$\forall\, x \in  P^\iin \ssm \cK_{\partial,0}, \qq   \Gamma_1(h^\iin(x)) =  h^\out ( \Gamma_0 (x))$$
\end{claim}

\begin{proof}
   The set $\cK_{\partial,0} \cap P^\iin$ corresponds to the orbits of $\varphi_0$ that enter $P$ and never cross the boundary again. This property is obviously preserved by the orbit equivalence.
   The commutation relation is obvious by definition of orbit equivalence. 
\end{proof}

Let $\Theta_i$ be the first return map for $\phi_i$ on $\hat \cT$ in $M$. It is a diffeomorphism from $\cT \ssm \cK_{\cT,i}^s$ to $\cT \ssm \cK_{\cT,i}^u$. 
Recall that $\cK_{\cT,i}^s$ is the projection of $\cK_{\partial,i} \cap \Pin$ and $\cK_{\cT,i}^u$ is the projection of $\cK_{\partial,i} \cap \Pout$.
Let $H \colon P \to P$ be the diffeomorphism given by Lemma \ref{lem: H diff of pieces}.
Define $h^s\colon \cT \to \cT$ and $h^u \colon \cT \to \cT$ by
\begin{equation} \label{eq: Hs and Hu}
    h^s := \pi_1^\iin \circ h^\iin \circ (\bar H \circ \pi_0^\iin)^\inv , \quad h^u := \pi_1^\out \circ h^\out \circ (\bar H \circ \pi_0^\out)^\inv
\end{equation}

\begin{claim} \label{claim: equivalence on projection torus}
 The map $h^s$ maps $\cK_{\cT,0}^s$ to $\cK_{\cT,1}^s$ and 
 $h^u$ maps $\cK_{\cT,0}^u$ to $\cK_{\cT,1}^u$.
 They satisfy the following commutation relation
$$\forall\, x \in  \cT \ssm \cK_{\cT,0}^s, \qq  \Theta_1(h^s(x)) = h^u (\Theta_0 (x))$$
\end{claim}

\begin{proof}
    Follows from Claim \ref{claim: equivalence on boundary}, Equation \eqref{eq: Hs and Hu}, and the identity
    \[
    \Theta_i = \pi_i^\out \circ \Gamma_i \circ (\pi_i^\iin)^{-1}\qedhere
    \]
\end{proof}

Consider the pair $(\cF_{\cT,i}^+, \cF_{\cT,i}^-)$ induced by the stable and unstable foliations of $\phi_i$ on $\cT_i$ from Proposition \ref{prop: qt bifoliation on qt surface in pA flow}.
The pair $(\cF_{\partial,i}, (f_i)_*(\cF_{\partial,i}))$ on $\pP_i$ (from Definition \ref{def: boundary lam and fol}) projects to $(\cF_{\cT,i}^+, \cF_{\cT,i}^-)$ on $\cT$ in $M$ (after possibly composing by the diffeomorphism of Lemma \ref{lem: H diff of pieces}).

Choose a first leaf of $\cK_{\partial,0}$; take its image under $h$ as the corresponding first leaf in $\cK_{\partial,1}$.
This gives a choice of a first leaf for $\cF_{\cT,0}^+$ and $\cF_{\cT,1}^+$ which defines a bar code for $(\cF_{\cT,0}^+, \cF_{\cT,0}^-)$ and $(\cF_{\cT,1}^+, \cF_{\cT,1}^-)$.
According to the assumption of Theorem \ref{thmintro: uniqueness of gluing}, the bar codes of $(\cF_{\cT,0}^+, \cF_{\cT,0}^-)$ and $(\cF_{\cT,1}^+, \cF_{\cT,1}^-)$ are the same.

\begin{claim} \label{claim: map f comaptible orbit equivalence}
    There is a map $\xi\colon \cT \to \cT$ homotopic to the identity on $\cT$ and mapping compact leaves of $(\cF_{\cT,0}^s,\cF_{\cT,0}^u)$ to compact leaves of $(\cF_{\cT,1}^s,\cF_{\cT,1}^u)$, and satisfying 
    \begin{enumerate}
        \item $\xi(c^s_0) = h^s(c^s_0)$ for $c^s_0$ compact leaf of $\cK_{\cT,0}^s$
        \item $\xi(c^u_0) = h^u(c^u_0)$ for $c^u_0$ compact leaf of $\cK_{\cT,0}^u$
    \end{enumerate}
\end{claim}
\begin{proof}
Consider a parallel component of $\partial_+ P$.
By equivalence of the \paps{} and equality of the bar codes, we can choose $$\xi^+ \colon \partial_+ P \to \partial_+ P$$ a homeomorphism mapping compact leaves of the lift of $(\cF_{\cT,0}^+, \cF_{\cT,0}^-)$ in $\partial_+P$ to compact leaves of the lift of $(\cF_{\cT,1}^+, \cF_{\cT,1}^-)$ in $\partial_+P$, preserving the geometric enumeration: $\xi^+$ maps the $k$-th compact leaf of the lift of $(\cF_{\cT,0}^+, \cF_{\cT,0}^-)$ to the $k$-th compact leaf of the lift of $(\cF_{\cT,1}^+, \cF_{\cT,1}^-)$. 
It suffices to map the first oriented compact leaf to the associated first oriented compact leaf of the boundary lamination by the isotopy equivalence $h$ between $\phi_0$ and $\phi_1$, and preserve the cyclic order given by the orientation of each component of $\partial_+P$.
As $\xi^+$ preserves the geometric enumeration of the pair of foliations, Lemma \ref{lem: carac same bar code} implies that it satisfies
\begin{enumerate}[label=(\arabic*)]
    \item \label{it: f^+ = varphi^s} $\xi^+ (c_0) = h (c_0)$ for $c_0$ a compact leaf of $\cK_{\partial,0}$, and
    \item \label{it: f^+ = varphi^u} $\xi^+ (c'_0) = f_1 \circ h \circ f_0^\inv (c'_0)$ for $c'_0$ a compact leaf of $(f_0)_* \cK_{\partial,0}$
\end{enumerate} 
As the flows $\phi_0$ and $\phi_1$ are isotopically equivalent, we can choose the map $\xi^+$ to be homotopic to the identity.
Similarly, on each scalloped component we can also build a homeomorphism $\xi^+$ homotopic to the identity mapping compact leaves to compact leaves and satisfying the same identities.
It suffices to remark that by isotopy equivalence of the \paps{} the compact leaves of $\cK_{\partial,0}$ and $\cK_{\partial,1}$ on a scalloped component are all isotopic in $\partial P$, 
and by isotopy of the gluing maps, the compact leaves of $(f_0)_*\cK_{\partial,0}$ and $(f_1)_*\cK_{\partial,1} $ are all isotopic in $\partial P$.

We now define the map $\xi$ by:
$$ \xi := \pi_1^+ \circ \xi^+ \circ (\bar H \circ \pi_0^+)^\inv \colon \cT \to \cT $$ 
It is clear that $\xi$ is homotopic to the identity on $\cT$.
Now recall that $\cF_{\cT,0}^+$ on $\cT$ coincides with the trace of $\cF^s_0$ on some \ccs{} of $\cT \ssm \cO_a$ and with the trace of $\cF^u_0$ on the other such that they alternate. 
Moreover, $\cF_{\cT,0}^+$ is equal to the trace of $\cF^s_0$ on a connected component $A$ of $\cT \ssm \cO_a$ if and only if $\cF_{\cT,1}^+$ is equal to the trace of $\cF^s_1$ on $A$.
The properties of the claim are consequences of properties \ref{it: f^+ = varphi^s} and \ref{it: f^+ = varphi^u} of $\xi^+$.
\end{proof}

We can extend $\xi$ to a homeomorphism of $M$ homotopic to the identity.
Let $\bar \xi$ be the map associated to $\xi$ given by Lemma \ref{lem: isomorphism of chain}. The map $\bar \xi$ is an isomorphism of chains $\cC_0$ and $\cC_1$, where $\cC_0$ and $\cC_1$ are the projections of the complete lift $\widetilde \cT$ in the orbit space $\cQ_0$ and $\cQ_1$ respectively.
It remains to show that $\bar \xi$ respects the order $\prec$ on maximal lines of lozenges of $\cC_i$, so that we can apply Theorem \ref{thmintro: tool thm for free homotopy data} to deduce Theorem \ref{thmintro: uniqueness of gluing}.\\

We lift everything in $\widetilde M$.
Suppose $\cL_0 \prec \cL_0'$ in $\cC_0$ and set $\cL_1 = \bar \xi (\cL_0)$ and $\cL_1' = \bar \xi(\cL_0')$ the corresponding maximal lines on $\cC_1$.
Take $B_0$ and $B_0'$ maximal bands (Definition \ref{def: maximal band}) in $\widetilde M$ contained in the lift $\widetilde \cT$ of $\cT$, transverse to $\phi_0$ and projecting to $\cL_0$ and $\cL_0'$ on $\cQ_0$. Since the order is transitive, we may assume that $\cL_0$ and $\cL_0'$ are chosen so that there is an orbit of $\tilde \phi_0$ crossing $B_0$ at point $x$ then $B_0'$ at point $x'$, and never crossing the collection $\widetilde \cT$ between those two points.
We therefore have $x' = \tilde \Theta_0 (x)$.
Let $C_0$ be the \cc{} of $\widetilde \cT \ssm \widetilde \cK_{\cT,0}^s$ containing $x$, and $C'_0 = \tilde \Theta_0 (C_0)$ the \cc{} of $\widetilde \cT \ssm \widetilde \cK_{\cT,0}^u$ containing $x'$.
The lamination $\widetilde \cK_{\cT,0}^s$ is closed in $\widetilde \cF_{\cT,0}^\pm$, hence the periodic boundary leaves of a \cc{} $C_0$ of $\cT \ssm \cK_{\cT,0}^s$ belong to the lamination $\widetilde \cK_{\cT,0}^s$ (Remark \ref{rmk: periodic boundary}).

\begin{nota} \label{nota: side periodic boundary}
Let $C$ be a connected componentof $\tilde \cT \setminus \tilde \cK^{s/u}_\cT$ intersecting $B$, and $c$ a periodic boundary of $C$. We denote $s(c, C)$ the side of $c$ in $\tilde \cT$ intersecting $C$.
\end{nota}

From Claim \ref{claim: corresponding nbh to side of periodic boundary}, there is a periodic boundary $c_0 \in \widetilde \cK_{\cT,0}^s$ of $C_0$ and a periodic boundary $c'_0 \in \widetilde \cK_{\cT,0}^u$ of $C'_0$ such that the side $s(c_0, C_0)$ is contained in $B_0$ and similarly the side $s(c'_0, C'_0)$ is contained in $B'_0$. See Figure \ref{fig: side periodic}.\\

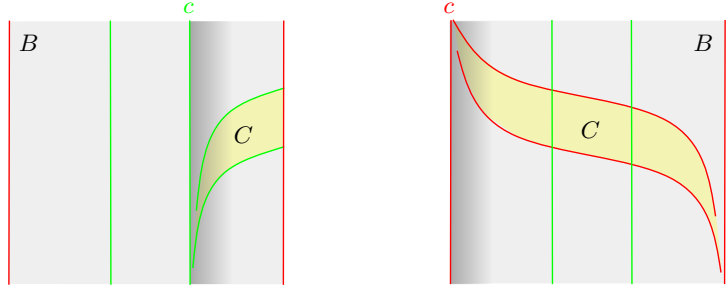
\begin{figure}[h]
\labellist
\small

\pinlabel $C$ at 171 109
\pinlabel $C$ at 423 113

\pinlabel $\textcolor{green}{c}$ [b] at 131 192
\pinlabel $\textcolor{red}{c}$ [b] at 320 192

\pinlabel $B$ at 15 176
\pinlabel $B$ at 505 176

\endlabellist

    \centering
    \includegraphics[width=0.75\linewidth]{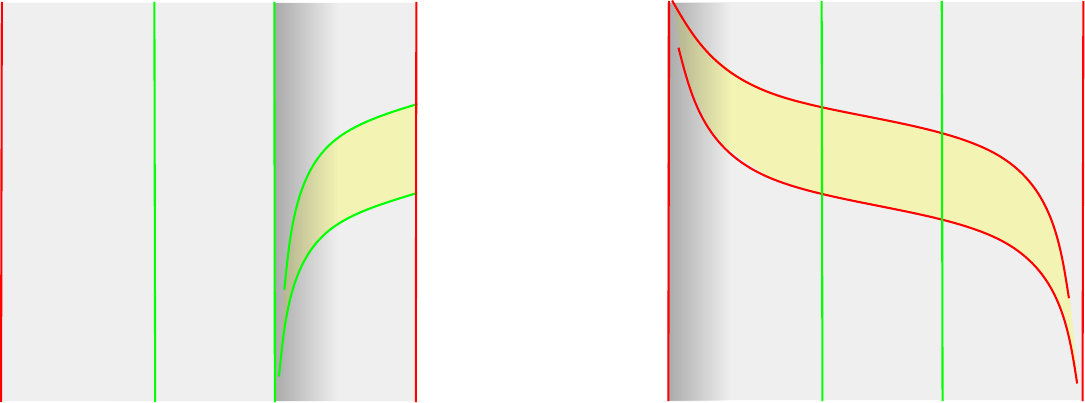}
    \caption{Side $s(c,C)$ (shaded) of a periodic boundary $c$ of a component $C$ contained in a maximal band $B$}
    \label{fig: side periodic}
\end{figure}

Let $\tilde \xi$ be the lift of $\xi$ in $\widetilde M$ commuting with $\pi_1(M)$.
It maps the periodic leaves of $(\widetilde \cF_{\cT,0}^s, \widetilde \cF_{\cT,0}^u)$, the trace of $(\widetilde \cF^s_0, \widetilde \cF^u_0)$ on $\widetilde \cT$, to periodic leaves of $(\widetilde \cF_{\cT,1}^s, \widetilde \cF_{\cT,1}^u)$, the trace of $(\widetilde \cF^s_1, \widetilde \cF^u_1)$ on $\widetilde \cT$. Hence, $\tilde \xi$ maps maximal periodic bands to maximal periodic bands (see Definition \ref{def: maximal band}).

\begin{claim} \label{claim: compatibility f and orbit equivalence} 
The map $\tilde \xi$ satisfies the following:
 \begin{enumerate}
     \item $\tilde \xi( c_0 ) = \tilde h^s (c_0)$ and $\tilde \xi( c'_0 ) = \tilde h^u (c'_0)$
     \item a maximal band $B$ contains $s(c_0, C_0)$ if and only if $\tilde \xi(B)$ contains $s(\tilde h^s (c_0), \tilde h^s (C_0))$
     \item a maximal band $B$ contains $s(c'_0, C'_0)$ if and only if $\tilde \xi(B)$ contains $s(\tilde h^u (c'_0), \tilde h^u (C'_0))$
 \end{enumerate}
\end{claim}

\begin{proof}
The first item comes from Claim \ref{claim: map f comaptible orbit equivalence} and the fact that $c_0$, $c'_0$ are leaves of $\widetilde \cK_{\cT,0}^s$ and $\widetilde \cK_{\cT,0}^u$ respectively.
For the second item, we know that $\tilde \xi$ maps $c_0$ to $\tilde h^s(c_0)$. 
As $\tilde h^s$ preserves the stable lamination, this leaf is a periodic boundary of the \cc{} $\tilde h^s(C_0)$ of $\widetilde \cT \ssm \cK_{\cT,1}^s$.
Set $c_1 := \tilde \xi (c_0)=\tilde h^s(c_0)$. This is a periodic boundary leaf of $C_1 := \tilde h^s(C_0)$.
Suppose $B$ is a maximal band that contains $s(c_0, C_0)$. 
Both $\tilde \xi$ and $\tilde h^s$ preserve the orientation on $\widetilde \cT$ and map the oriented leaf $c_0$ to the same oriented leaf $c_1$.
It follows that they map the side $s(c_0, C_0)$ to the same side $s(c_1, C_1)$, which shows the second item. The proof of the third item is identical.
\end{proof}

\begin{figure}[h]
\labellist
\small

\pinlabel $\textcolor{blue}{\tilde\phi_0}$ [r] at 8 237
\pinlabel $\textcolor{blue}{\tilde\phi_1}$ [r] at 333 237

\pinlabel $B_0$ [l] at 2 10
\pinlabel $B'_0$ [l] at 2 325

\pinlabel $B_1=\tilde\xi(B_0)$ [l] at 320 10
\pinlabel $B'_1=\tilde\xi(B'_0)$ [l] at 320 325

\pinlabel $C_0$ at 122 90
\pinlabel $C_0'=\tilde\Theta_0(C_0)$ at 125 421

\pinlabel $C_1=\tilde{h}^s(C_0)$ at 472 110
\pinlabel $C_1'=\tilde{h}^u(C'_0)$ at 428 421

\pinlabel $\textcolor{green}{c_0}$ [b] at 90 191
\pinlabel $\textcolor{green}{c_1}$ [b] at 408 191

\pinlabel $\textcolor{blue}{c_0'}$ [b] at 2 503
\pinlabel $\textcolor{blue}{c_1'}$ [b] at 320 503

\pinlabel $x$ [tl] at 172 115
\pinlabel $x'$ [l] at 63 431

\pinlabel $\tilde h^u$ [b] at 265 392
\pinlabel $\tilde \xi$ [b] at 265 357

\pinlabel $\tilde h^s$ [b] at 265 83
\pinlabel $\tilde \xi$ [b] at 265 49

\pinlabel $\tilde\Theta_0$ [l] at 124 255

\endlabellist

    \centering
    \vspace{5mm}
    \includegraphics[width=1\linewidth]{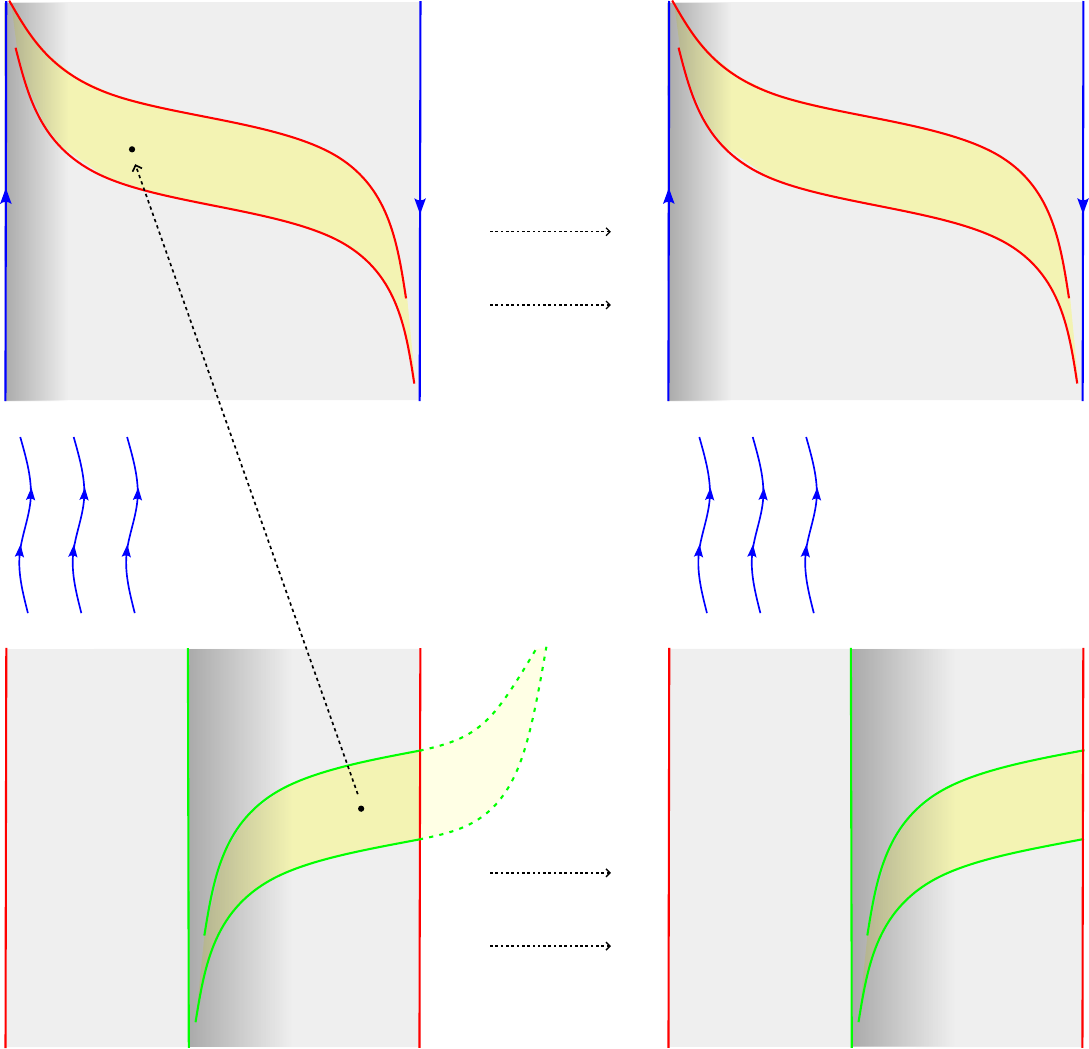}
    \caption{Relations between the maximal bands, components, and periodic boundary on the collection $\cT$. The sides $s(c,C)$ are shaded.}
    \label{fig: homeo band}
\end{figure}
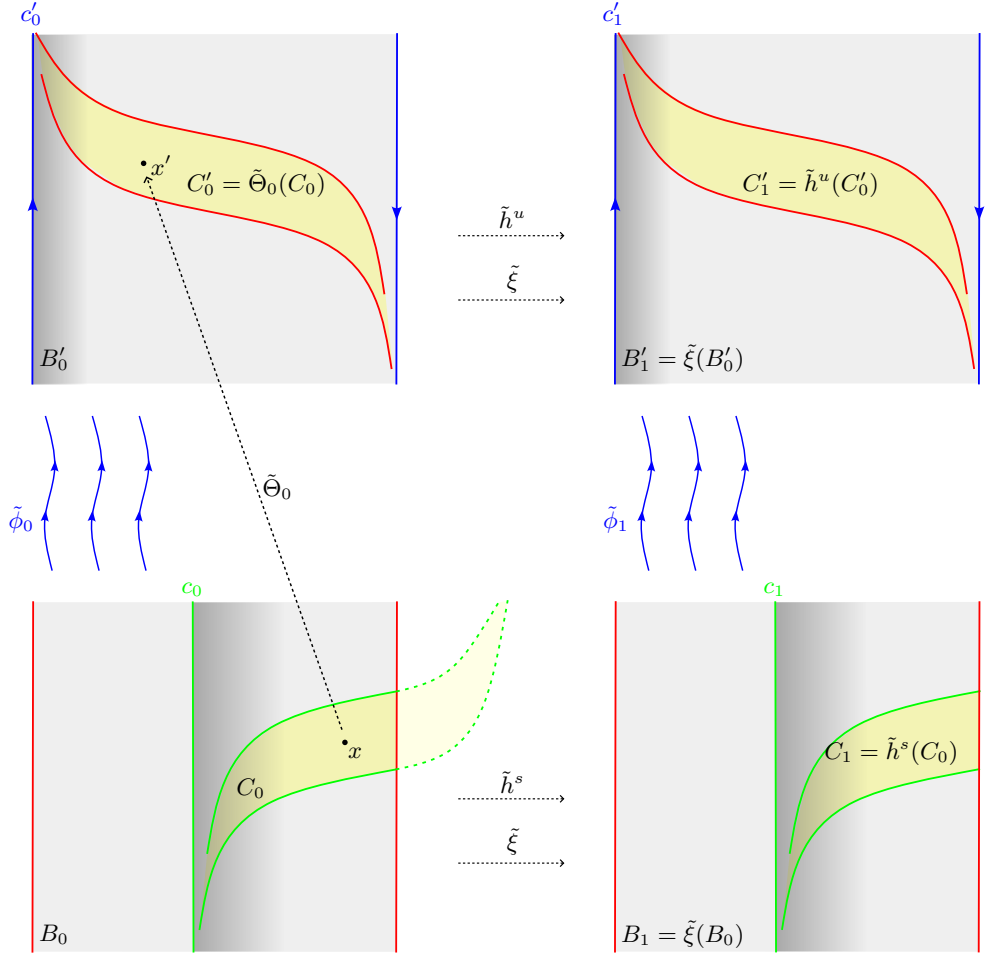

Set $c_1 = \tilde \xi (c_0)$ and $c'_1 = \tilde \xi (c'_0)$. Those are periodic boundary leaves of $C_1= \tilde h^s(C_0)$ and of $C_1'= \tilde h^u(C'_0)= \tilde \Theta_1 (C_1)$ respectively.
Set $B_1 = \tilde \xi(B_0)$ and $B_1' = \tilde \xi(B_0')$. Those are maximal bands which project to the maximal lines $\tilde \cL_1$ and $\tilde \cL_1'$ in the orbit space of $\phi_1$ (see construction of $\bar \xi$ in Lemma \ref{lem: isomorphism of chain}).
See Figure \ref{fig: homeo band}.

We must show that there exists an orbit of $\tilde \phi_1$ from $B_1$ to $B_1'$.
By Claim \ref{claim: compatibility f and orbit equivalence}, $B_1$ contains $s(c_1, C_1)$, and $B'_1$ contains $s(c'_1, C'_1)$.
Now, by Lemma \ref{lem: action lift return map transverse segment}, the image under the map $\Theta_1$ of a small transversal $\sigma^u$ contained in $B_1$ will accumulate to $c_1'$ on the side $s(c'_1, C'_1)$. In particular, one can find an orbit of $\tilde \phi_1$ from $B_1$ to $B'_1$.

We conclude that the pseudo-Anosov flows $\phi_0$ and $\phi_1$ satisfy the assumptions of Theorem \ref{thmintro: tool thm for free homotopy data}, therefore they have the same free homotopy data.

\begin{claim} 
If one of the flows is transitive, then they are orbit equivalent.
\end{claim}

\begin{proof}
Two transitive pseudo-Anosov flows with the same (equivalent) free homotopy data are orbit equivalent, unless there exists a periodic Seifert piece $P$ all of whose boundary surfaces are scalloped and such that the \emph{signs} of the associated trees of scalloped regions differ (see \cite[Theorem 1.3]{BFrM25} and \cite[Theorem 2.26]{BFM25})\footnote{For Anosov flows, one may describe this much more succinctly as the flows are orbit equivalent unless they differ by a periodic Seifert flip \cite[Theorem 1.1]{BFM25}.}. We refer to Section 6.2 of \cite{BFM25} for the precise definition of the sign associated to a scalloped periodic Seifert piece. Here it is enough to know that the sign is determined by the orientations of the periodic orbits contained in $P$.

Assume by contradiction that there is a scalloped periodic Seifert piece $P$ of the JSJ decomposition of $M$ on which the signs of $\phi_0$ and $\phi_1$ differ.

If the interior of $P$ is disjoint from the collection $\cT$, this is impossible: indeed, the restrictions of
$\phi_0$ and $\phi_1$ to $M\setminus \cT$ come from equivalent pseudo-Anosov pieces (in particular, they are isotopically equivalent on the restriction $M \setminus \cT$), so the signs must match.

So we may assume that $P$ meets $\cT$ along a component $T$.
By uniqueness of the JSJ decomposition, the incompressible torus $T$ is contained in $P$.
Inside a periodic Seifert piece, a quasi-transverse torus has the following property:
every orbit meeting $T$ either is contained in $T$, or exits the piece both in the past and in the future.

Since $\phi_0$ and $\phi_1$ have the same free homotopy data, their spines on $P$ are diffeomorphic, and may only differ by the orientation of matching periodic orbits in $P$. 
The periodic orbits on $T$ come from periodic orbits of the boundary of the pseudo-Anosov pieces, and for
two isotopic pseudo-Anosov pieces this data is fixed: In particular, the orientation of
the periodic orbits contained in a boundary torus is preserved, which implies that the orientation on the spines of $\phi_0$ and $\phi_1$ and thus the signs of $\phi_0$ and $\phi_1$ must be equal on $P$.
\end{proof}

This finishes the proof of Theorem \ref{thmintro: uniqueness of gluing}.
\end{proof}

We showed in the particular case of filling \paps{} that the bar code of a triple is equal to the complete bar code (Proposition \ref{prop: bar code and complete bar code for filling piece}).
In that case, we can replace complete bar code by bar code:

\begin{coro} \label{coro: uniqueness of filling gluing}
Two equivalent filling triples with the same bar code induce \pa{} flows with the same free homotopy data. 
If one of them is transitive, then they are isotopically equivalent.
\end{coro}

Putting together Proposition \ref{propintro: finiteness of bar code} and Theorem \ref{thmintro: uniqueness of gluing}, we obtain 

\begin{coro} \label{coro: finiteness of free homotopy data}
    For a fixed manifold $M$, there are finitely many free homotopy data of pseudo-Anosov flows constructed from equivalent \pa{} triples $(P,\phi,f)$ such that $P/f\simeq M$.
\end{coro}

\subsection{Comparison with previous gluing constructions} 
\label{subsec: equivalent triple and isotopic triple}

The notion of \emph{equivalent triple} that we decided to use here
(Definition~\ref{def: equivalent triple}) differs in subtle ways from the notions of
(strongly) isotopic building blocks and triples from
\cite[Definitions~1.27 and~1.35]{pauletAnosovFlowsDimension2025},
and of strongly isotopic hyperbolic plugs from
\cite{beguinBuildingAnosovFlows2017}.
We chose to work with orbit equivalence rather than isotopy, as in \cite{beguinBuildingAnosovFlows2017,pauletAnosovFlowsDimension2025}, because this is the relation that appears naturally in Section~\ref{sec: free data block} and in the finiteness arguments of
Section~\ref{sec: block gluing} making for a more straightforward definition.

Here, we will quickly point out the differences between the different notions, and explain how Theorem \ref{thmintro: uniqueness of gluing} can be adapted to these different notions.

First, consider $(P,X)$ and $(P',X')$ two building blocks in the sense of
\cite[Definition~1.2]{pauletAnosovFlowsDimension2025}, and assume that they are
isotopic in the sense of \cite[Definition~1.27]{pauletAnosovFlowsDimension2025}.
Denote by $\cO$ and $\cO'$ the collections of boundary periodic orbits of $(P,X)$ and
$(P',X')$, and let
\[
\hat P:=P\setminus \cO, \qquad \hat P':=P'\setminus \cO'
\]
with $\hat X$ and $\hat X'$ the restrictions of the vector fields.
Then \cite[Proposition~1.30]{pauletAnosovFlowsDimension2025} gives an orbit
equivalence
$h\colon \hat P \to \hat P'$ between the flow of $\hat X$ and $\hat X'$.
In general, however, this orbit equivalence does not extend to the whole blocks.

Conversely, if $(P,\varphi,f)$ and $(P',\varphi',f')$ are two pseudo-Anosov triples in our sense (Definition \ref{def: gluing map}) whose
pieces are orbit equivalent, they need not be isotopic in the sense of
\cite[Definition~1.27]{pauletAnosovFlowsDimension2025}. Nevertheless, after replacing
one of the triples by a smoothly orbit equivalent representative, one may arrange
that the two triples become isotopic in that sense.

Despite these differences, a version of Theorem \ref{thmintro: uniqueness of gluing} can also be formulated using isotopic
triples instead of equivalent triples. Indeed, our arguments above only use the following key elements: We need to have 
the same free homotopy data for periodic orbits inside the blocks,
an orbit equivalence on the entering and exiting components of the boundary of the blocks, and
an isotopy of the gluing maps on the boundary. Each of these properties is also implied by the isotopy framework.
For this reason, the arguments adapt with no essential change to that setting, and one gets Theorem \ref{thmintro: uniqueness of gluing} replacing equivalent triples by isotopic triples.

In particular, one can recover uniqueness results for the gluing procedures
considered in \cite{pauletAnosovFlowsDimension2025} and
\cite{beguinBuildingAnosovFlows2017}: 
If two isotopic triples have the same bar codes (in the natural sense
obtained by comparing compact leaves through the isotopy), then they are strongly
isotopic in the sense of \cite{pauletAnosovFlowsDimension2025} and
\cite{beguinBuildingAnosovFlows2017}.
In the transverse boundary case, this gives back the
uniqueness theorem for Béguin--Bonatti--Yu blocks
\cite[Theorem~1.2]{beguinUniquenessTheoremTransitive2023}. 
For quasi-transverse boundary, we obtain the following analog of \cite[Theorem 1.5]{beguinUniquenessTheoremTransitive2023}.
\begin{coro}
    Suppose that $(P,X,f)$ and $(P',X',f')$ are strongly isotopic triples in the sense of \cite{pauletAnosovFlowsDimension2025} inducing transitive Anosov flows on $P/f$ and $P'/f'$. Then they are orbit equivalent.
\end{coro}

\section{Finiteness on manifolds with skew hyperbolic pieces}\label{sec: finiteness flow}

In this section, we will prove Theorems \ref{thmintro: finiteness orbit equivalence} and \ref{thmintro: finiteness extra conditions} from the introduction.
We start by giving the precise definition of what we call a skew piece, and refer to \cite{barthelmeSkewPiecesPseudoAnosov} for a list of equivalent characterizations of it.

\begin{defi}\label{def: skew pap}
    Let $(P,\varphi)$ be a \pap{}.  
    We say that $(P,\varphi)$ is a \emph{skew} piece if it is isotopic to a \pap{} $(P',\varphi')$ obtained by cutting a skew Anosov flow along a good collection of quasi-transverse tori.
\end{defi}

We refer to Section~\ref{subsec: pap} for the general definitions and notation concerning \paps{} and good collections.

Recall the statement of Theorem \ref{thmintro: finiteness orbit equivalence} that we are aiming to prove:
\finitenessgeneral*

As before, the transitivity assumption in Theorem \ref{thmintro: finiteness orbit equivalence} is only in order to be able to apply \cite{BFrM25} and deduce orbit equivalence from having the same free homotopy data, so what we actually will prove in this section is the following:

\begin{thm}\label{thm: finiteness free homotopy data}
    Let $M$ be a compact $3$-manifold and $\{\phi_i\}$ a family of smooth pseudo-Anosov flows on $M$.
    Assume that each $\phi_i$ is skew on each hyperbolic piece of the JSJ decomposition of $M$.
    
    Then there exists a finite cover $\hat M$ and an integer $n$, both depending only on the topology of $M$, such that the free homotopy data $\cP(\hat\phi_i)$ of the lifted flows $\{\hat \phi_i\}$ belong to at most $n$ distinct equivalent classes of free homotopy data in $\hat M$.
\end{thm}

Here, we need to assume that the flows are \textit{smooth} (see Section \ref{sec: preli; subsec: pa flows}) to be able to use the construction in \cite{pauletAnosovFlowsDimension2025}. That smoothness assumption is not needed for transitive pseudo-Anosov flows, thanks to the result of Shannon \cite{Sha25} and Agol--Tsang \cite{AT24}, which gives that transitive pseudo-Anosov flows are always orbit equivalent to smooth ones.

\medskip

We now work toward the proof of Theorem~\ref{thm: finiteness free homotopy data}.  

Let $\{\phi_i\}_{i\in\N}$ be an infinite family of smooth pseudo-Anosov flows on $M$ such that, on every hyperbolic JSJ piece of $M$, each $\phi_i$ is skew. By Proposition~\ref{prop: cover embedded jsj tori}, there exists a  finite cover $\hat M$ of $M$ such that
\begin{itemize}
        \item $\hat M$ is orientable;
        \item every Seifert piece of the JSJ-decomposition of $\hat M$ has orientable base and orientable fibers;
        \item for each $i$, the lifted flow $\hat \phi_i$ admits a modified JSJ-decomposition with \emph{embedded} JSJ tori.
\end{itemize}
Thus, to prove Theorem \ref{thm: finiteness free homotopy data}, we need to show that there can be only finitely many distinct free homotopy data among the flows $\hat\phi_i$ on $\hat M$.
Our proof will be by contradiction. In order to avoid the use of hat decorations throughout (and at the same time, show Theorem \ref{thmintro: finiteness extra conditions}), we now assume that $\hat M= M$ and $\hat\phi_i = \phi_i$. Up to taking a subsequence of the $\phi_i$, we will now work with the following assumptions:
\begin{convention} \label{conv: finite cover}
For the rest of the proof of Theorem \ref{thm: finiteness free homotopy data}, we suppose that $M$ is a closed $3$-manifold carrying infinitely many pseudo-Anosov flows $\{\phi_i\}_{i\in\N}$ that satisfy:
 \begin{enumerate}[label=(\roman*)]
     \item $M$ is orientable;
     \item Each Seifert piece $P$ of the JSJ decomposition of $M$ has base an orientable surface and orientable fibers;
     \item Each $\phi_i$ is smooth;
     \item Each $\phi_i$ is skew on each hyperbolic piece of the JSJ decomposition of $M$;
     \item Each torus $T$ of the JSJ decomposition is isotopic to an \emph{embedded} quasi-transverse torus for $\phi_i$;
     \item There is a fixed collection of JSJ pieces, denoted $P_{\per}$, such that, for each $i$, the flow $\phi_i$ is periodic on all the pieces in $P_\per$;
     \item The free homotopy data $\cP(\phi_i), i \in \N$ of the flows are all distinct.
 \end{enumerate}
\end{convention}

Our proof will now consist of the following steps: First, we will show that, up to taking a subsequence, all the $\phi_i$ must be orbit equivalent when restricted to the complement of the periodic pieces. Then, we will show that, again up to passing to a subsequence, the flows are also all orbit equivalent in each periodic Seifert piece. Putting this together, we will finally obtain a contradiction with Theorem \ref{thmintro: uniqueness of gluing}. If we further assume that all the flows are transitive, then Theorem \ref{thmintro: uniqueness of gluing} gives us Theorem \ref{thmintro: finiteness orbit equivalence}.

\subsection{Finiteness of the skew pieces}

Let $P$ be a connected component of $M \setminus P_\per$. Note that $P$ is a union of JSJ pieces and we denote by $P_i$ the corresponding union of modified JSJ pieces with respect to the flow $\phi_i$. Let $\varphi_i = \phi_i|_{P_i}$, so that $(P_i,\varphi_i)$ is a \pap{}.

Recall that by Convention \ref{conv: finite cover} each boundary torus of each $P_i$ is embedded, while the collection of all the boundary tori may a priori fail to be embedded. This potential failure would be problematic for us, as we want to be able to use the gluing results of \cite{pauletAnosovFlowsDimension2025}, which require $P_i$ to be a manifold, which is equivalent in our setting to having the union of the boundary tori of $P_i$ be embedded. Our first lemma shows that $P_i$ is indeed a manifold, and that we can use the results of \cite{pauletAnosovFlowsDimension2025}.

\begin{lem}\label{lem: Pi building block}
    For each $i$, $P_i$ is a manifold and the pseudo-Anosov piece $(P_i, \varphi_i)$ is a building block in the sense of \cite[Definition 1.2]{pauletAnosovFlowsDimension2025}.
\end{lem}

\begin{proof}
First note that if $P_\per$ is empty, then $P_i=M$ and there is nothing to prove.

To ease notation, we will drop the $i$ subscripts in this proof and just set $P=P_i$ and $\varphi=\varphi_i$.
The flows are $\cC^1$ by assumption, hence the maximal invariant set admits a hyperbolic splitting.
To show that it is a building block, we need to show that the boundary of $P$ is embedded.
By Convention~\ref{conv: finite cover}, every modified JSJ torus of $\phi$ is embedded in $M$.
Hence each boundary component of $P$ is an embedded torus.
To prove that $P$ is a manifold, it only remains to show that two boundary components of $P$
cannot intersect along a periodic orbit, i.e., that no corner circle occurs.
Suppose for contradiction that there is a corner periodic orbit $O$ in $P$.
Then there exist two distinct boundary tori $T_1,T_2\subset \partial P$ meeting along $O$.
Since $P$ is obtained from the complement of the periodic JSJ pieces of $M$,
the orbit $O$ is contained in some JSJ piece $Q$ of $M$ adjacent to both $T_1$ and $T_2$, and $Q$ is either hyperbolic or a free Seifert piece.

We first rule out the hyperbolic case.
If $Q$ is hyperbolic, the common orbit $O$ would give an essential annulus in $Q$
with boundary on the boundary tori corresponding to $T_1$ and $T_2$.
This contradicts the fact that a hyperbolic JSJ piece is acylindrical.

We now rule out the free Seifert case.
Let $h\in \pi_1(Q)$ be the element represented by the periodic orbit $O$.
Since $O\subset T_1\cap T_2$, the element $h$ belongs to both boundary torus subgroups
$\pi_1(T_1)\cong \mathbb Z^2$ and
$\pi_1(T_2)\cong \mathbb Z^2$
inside $\pi_1(Q)$.
For an orientable Seifert piece, the intersection of two distinct boundary torus subgroups
is exactly the infinite cyclic subgroup generated by the regular fiber.
Therefore $h$ is a power of the fiber class.
This is impossible because $Q$ is by assumption a {free} Seifert piece.

We conclude that no corner periodic orbit can occur in $P$.
Therefore the boundary components of $P$ are pairwise disjoint embedded tori, so $P$ is a manifold, and $(P,\varphi)$ is a building block in the sense of \cite{pauletAnosovFlowsDimension2025}.
\end{proof}

Next we want to show that $(P_i, \varphi_i)$ is a skew piece. Note that this is not immediate from our assumption that all the modified JSJ pieces are either periodic Seifert or skew: $P_i$ may be a union of JSJ pieces, so we need to show that a union of skew pieces is a skew piece.

\begin{lem}\label{lem: Pi skew}
For each $i$, the pseudo-Anosov piece $(P_i, \varphi_i)$ is skew in the sense of Definition \ref{def: skew pap}.
\end{lem}

\begin{proof}
As in the preceding lemma, we set $P=P_i$ and $\varphi=\varphi_i$ to simplify notations. 
If $P$ is a single JSJ piece, then there is nothing to prove. Otherwise, $P$ is a union of hyperbolic skew pieces and free Seifert pieces glued along incompressible tori.
Let $(P', \varphi')$ be another copy of $(P,\varphi)$.
We know that $(P,\varphi)$ and $(P', \varphi')$ are building blocks in the sense of \cite{pauletAnosovFlowsDimension2025} by Lemma \ref{lem: Pi building block}.
Each boundary torus is adjacent to a skew hyperbolic piece or a free Seifert piece. In particular, the boundary lamination is elementary: There are no compact leaves other than the boundary periodic orbits; there is an even number of such boundary orbits with alternating orientations. Thus the boundary lamination is filling and alternating elementary (see \cite[Section 11.4]{pauletAnosovFlowsDimension2025}).
There is an obvious pairing of boundary components between $(P,\varphi)$ and its copy $(P',\varphi')$ with the same number of periodic orbits.
The block $(P, \varphi)$ satisfies the same properties as those of a \textit{skewed $\R$-covered block} in \cite[Lemma 11.4]{pauletAnosovFlowsDimension2025}. Therefore, we can apply \cite[Lemma 11.4]{pauletAnosovFlowsDimension2025} to the union of $(P,\varphi)$ and its copy $(P',\varphi')$.
We can then apply \cite[Theorem 1]{pauletAnosovFlowsDimension2025} and get that there is a piece $(P_1, \varphi_1)$ \textit{isotopic} (see \cite[Definition 1.27]{pauletAnosovFlowsDimension2025}) to $(P, \varphi)$ and a gluing map $f$ gluing $(P_1, \varphi_1)$ to its copy $(P'_1, \varphi'_1)$, and inducing an Anosov flow $\phi'$ on a closed manifold $M'$ diffeomorphic to the double of $P$ glued along $f$.
Up to replacing the boundary $\partial P$ in $M$ by an isotopic good collection of quasi-transverse tori, we can assume that the piece $(P_1, \varphi_1)$ for the gluing is orbit equivalent to the initial piece $(P, \varphi)$ (see \cite[Proposition 1.30]{pauletAnosovFlowsDimension2025}). So without loss of generality, we may assume that $P_1=P$, $\varphi_1 = \varphi$.

The fact that the Anosov flow $\phi'$ on $M'$ is skew is proved in \cite{barthelmeSkewPiecesPseudoAnosov}, for completeness, we sketch the argument, which is only a very slight modification of the proof of Theorem 5.3.2 of \cite{barthelmePseudoAnosovFlowsPlane2025}:

By construction, $(M',\phi')$ is built by gluing skew pieces together. Call $Q_j$, $j=1, \dots, k$, the pieces. Then each $\pi_1(Q_i)$ (seen as a subgroup of $\pi_1(M')$) preserves an axis $A_i\simeq \R$ in the, say stable, leaf space $\Lambda(\cQ^s)$ of $\phi'$. (This axis comes from the semi-conjugacy between $\phi'|_{Q_j}$ and a piece of a skew flow.)

We claim that $A_i=A_j$ for all $i,j$: By connectedness, it is enough to prove it for two adjacent pieces $Q_i,Q_j$. Such adjacent pieces share a common boundary torus, so there exists a common subgroup $\pi_1(T)$ of both $\pi_1(Q_i)$ and $\pi_1(Q_j)$. Hence $\pi_1(T)$ preserves both $A_i$ and $A_j$. Now $\pi_1(T)$ also contains an element $h$ acting freely on the orbit space of $\phi'$, and by uniqueness of axis of free elements (see, e.g., \cite[Section 3.3]{barthelmePseudoAnosovFlowsPlane2025}) we must then have $A_i=A_j$ as claimed. 

Call $A = A_j$ for all $j$. Now we claim that $\pi_1(M')$ preserves $A$: $\pi_1(M')$ is a graph of groups, with vertex groups being the $\pi_1(Q_j)$. If there are no loops in the graph, then $\pi_1(M')$ is just as amalgamated free product, and the result is immediate. Now each loop in the graph gives an additional generator, but each such element normalizes $\pi_1(Q_i)$, so also preserves $A$. Hence, we have that $\Lambda(\cQ^s) = A \simeq \R$, so $\phi'$ is $\R$-covered, and since it is not a suspension flow, it is skew, which ends the proof. \qedhere

\end{proof}

\begin{defi}[Boundary data] \label{def: boundary data}
    Let $(P, \varphi)$ be a \pap{} cut in a \paf{} $(M, \phi)$.
    We call the \emph{boundary data} of $(P,\varphi)$ the free homotopy class in $M$ of the boundary periodic orbits of $\phi$ on $\partial P$, on each connected component of $\partial P$.
\end{defi}

We now prove:
\begin{prop}\label{prop: finite on skew}
    There are only finitely many distinct orbit equivalence classes of \paps{} $(P_i,\varphi_i)$ on the complement $P$ of periodic Seifert pieces and the orbit equivalences are isotopic to the identity on the boundary.
\end{prop}

Recall that all the pieces $P_i$ are isotopic to $P$ in the manifold $M$ so it makes sense to say that an orbit equivalence $P_i \to P_j$ is isotopic to the identity on the boundary.

\begin{proof}
Suppose first that $P_\per$ is empty. So $P_i = M$ for all $i$, and by Lemma \ref{lem: Pi skew}, all the $\phi_i$ are skew flows on $M$. Since there are only finitely many distinct orbit equivalence classes of skew Anosov flows in a fixed manifold (\cite[Corollary~C]{martySkewedAnosovFlows2023}), the proposition follows directly in that case.

So we now assume that $P_\per$ is non-empty, and we suppose, by contradiction, that the \paps{} $(P_i,\varphi_i)$ give rise to infinitely many distinct orbit equivalence classes.
    The boundary $\partial P$ has finitely many components $T^1, \dots, T^k$ and for each $i$ there is a corresponding torus $T_i^k\in \partial P_i$, so that for each fixed $k$, $T_i^k$ and $T_j^k$ are isotopic to the same JSJ torus $T^k$ in $M$.

We  first claim that all pieces have the same boundary data on the component corresponding to $T^k$:
Each such component is the boundary torus of the same periodic JSJ piece of $M$, and the periodic orbits of $\phi_i$ contained in $T^k_i$ are freely homotopic to the fiber of this periodic piece.  
        This fiber class is fixed by the topology of $M$, hence the corresponding class is the same for every $\phi_i$. Therefore, the boundary data of each $(P_i,\varphi_i)$ are indeed the same.

    Let $(P_i',\varphi_i')$ denote a second copy of $(P_i,\varphi_i)$.  

     As in the proof of Lemma \ref{lem: Pi skew}, we can find a \pagp{} 
      $f_i \colon \partial P_i \longrightarrow \partial P_i'$
        for a \pap{} equivalent to $(P_i \cup P_i',\, \varphi_i \cup \varphi_i')$ inducing a skew Anosov flow $\psi_i$ on the resulting manifolds
        $N_i \simeq (P_i \cup P_i') / f_i$.
        
        Let us show that each $N_i$ is diffeomorphic to $N_0$.
        Let $c_i$ be the free homotopy class represented by a boundary periodic orbit of $\varphi_i$ in $T_i$, and let $c_i'$ be the corresponding class for the copy $\varphi_i'$ on the associated boundary torus $T_i'$. By equality of the boundary data, these classes are independent of $i$.
        Any \pagp{} $f_i$ maps $T_i$ to $T_i'$ and sends the class $c_i=c_0$ to $c_i'=c_0'$.  
        It follows that two such gluings differ, up to isotopy, by a power of a Dehn twist along the periodic orbit of $\varphi_i'$ in $\partial P_i'$.
        By \cite[Lemma~11.6]{pauletAnosovFlowsDimension2025}, we can compose $f_i$ with an appropriate Dehn twist so that all gluing maps $f_i$ lie in the same isotopy class.
        In particular, the resulting manifolds $N_i$ are all diffeomorphic.
    
Using again that there are only finitely many skew Anosov flows on a given manifold up to orbit equivalence (\cite[Corollary~C]{martySkewedAnosovFlows2023}), we deduce that, after passing to an infinite subsequence, we may assume that the flows $\psi_i$ on $N_i$ are all orbit equivalent.
    Moreover, we can choose the isotopy class of each $f_i$ so that the image $\pi_i(\partial P_i)\subset N_i$ is a union of JSJ tori of $N_i$: Indeed, composing $f_i$ with an appropriate Dehn twist along the boundary periodic orbit ensures that $f_i$ does not identify Seifert fibers across the gluing.  

    Any orbit equivalence between $\psi_i$ and $\psi_j$ preserves the JSJ-decomposition of $N_i$ and $N_j$, and, up to composing with an isotopy along flow lines of $\psi_i$, we may arrange that the orbit equivalence maps the submanifold $\pi_i(P_i)\subset N_i$ exactly onto $\pi_j(P_j)\subset N_j$.
    In particular, this implies that the corresponding \paps{} $(P_i,\varphi_i)$ are all orbit equivalent.

    Now consider $H\colon P_i \to P_j$ be an orbit equivalence between $\psi_i$ and $\psi_j$. After composing by an isotopy we can assume that $P_i = P_j =P$ is the same manifold.
    Since the boundary data of the corresponding pieces
are the same, the restriction of $H$ to the boundary $\partial P$ preserves the class (in $\pi_1(\partial P)$) represented by the tangent periodic orbits, hence it is a power of a Dehn twist on each boundary torus along the corresponding slope.
Applying \cite[Theorem~1]{mcculloughHomeomorphismsWhichAre2006a} we deduce that two curves on distinct boundary components representing periodic orbits must bound an incompressible annulus in $P$.
This is impossible when $P$ is atoroidal.
If the piece is Seifert, such an annulus is isotopic to a vertical annulus (see, e.g., \cite{hatcherNotesBasic3Manifolda}), hence the periodic slope coincides with the Seifert-fiber slope and the piece is periodic. This is a contradiction and we conclude that $H_i|_{\partial P}$ is isotopic to the identity.
\end{proof}

\subsection{Finiteness of periodic Seifert pieces}

In this subsection, we fix a periodic Seifert piece of $M$, that we call $Q$, and denote by $(Q_i,\varphi_i)$
the modified pseudo-Anosov piece associated to $Q$ for the flow $\varphi_i$.
By Convention~\ref{conv: finite cover}, $Q$ is Seifert fibered over an orientable compact surface
$\Sigma$ with orientable fiber, and all modified JSJ tori of $\varphi_i$ are embedded.

\begin{prop} \label{prop: finite on one periodic piece}
There are only finitely many orbit equivalence classes of pseudo-Anosov pieces
$(Q_i,\varphi_i)$ on each periodic Seifert piece $Q$ and the orbit equivalences are isotopic to the identity on the boundary.

\end{prop}

\begin{proof}
Let $p_i \colon Q_i \to \Sigma_i$ be the Seifert fibration of the modified piece $Q_i$, where $\Sigma_i$ is an orientable surface.
Let $Z_i$ be the spine of $\phi_i$ in $Q_i$ and $\cG_i:=p_i(Z_i)\subset \Sigma_i$ the associated fatgraph (see Section \ref{subsec: spine}).

Recall that by Lemma~\ref{lem:finite_number_fat_graph}, up to diffeomorphisms, there are only finitely many fatgraphs with no vertices of valence $2$ in $\Sigma$.

The fatgraphs $\cG_i$ that we consider here may have vertices of valence $2$. However, such vertices are necessarily on the boundary of the modified surface $\Sigma_i$, as any vertex in the interior will have valence at least $4$.

Since the modified JSJ tori between two periodic Seifert pieces are transverse (see \cite[Proposition 3.4]{barbotClassificationRigidityTotally2015}), any vertex of valence $<4$ corresponds to periodic orbits on quasi-transverse boundary tori, which are adjacent to the non-periodic pieces, i.e., the skew piece $P = M \ssm P_\per$.

The number of periodic orbits on such a boundary torus is determined by the orbit equivalence class of the flow $\varphi_i$ on the modified skew piece $P_i$ associated to $P = M \ssm P_\per$. By the finiteness of the number of orbit-equivalent classes of such pieces of flows (Proposition \ref{prop: finite on skew}), we get that there is a uniformly bounded number of boundary periodic orbits. In particular there is a uniformly bounded number of vertices of valence $2$ in the associated fatgraphs.

As a consequence, we get that there are only finitely many diffeomorphism classes of possible fatgraphs in $\Sigma$ with the boundary data fixed. So, by Theorem \ref{thm_same_spine_implies_same_flow}, there are only finitely many distinct orbit equivalence classes of pseudo-Anosov pieces $(Q_i,\varphi_i)$.

Now consider $H\colon Q_i \to Q_j$ an orbit equivalence between $\varphi_i$ and $\varphi_j$. As in the proof of Proposition \ref{prop: finite on skew}, after composing by an isotopy we can assume that $Q_i = Q_j =Q$ is the same manifold. Using again \cite[Theorem~1]{mcculloughHomeomorphismsWhichAre2006a}, we know that if $H|_{\partial Q}$ is not the identity, then by uniqueness of the Seifert fibration (see Case 3 of the proof of Proposition \ref{prop: finiteness jsj}) it must correspond to Dehn twists on $\partial Q$ in the direction of the Seifert fibration of $Q$ (since an annulus joining two essential closed curves on distinct components of $\partial Q$ must be isotopic to a vertical annulus).

By \cite[Theorem~2]{mcculloughHomeomorphismsWhichAre2006a}
there exists
a composition
$D\colon Q\longrightarrow Q$
of three-dimensional Dehn twists along a disjoint collection of annuli whose
restriction to $\partial Q$ is isotopic to $H|_{\partial Q}$.
The twisting curves are Seifert-fiber slopes, so the annuli can be chosen vertical.
Therefore, the map $D$ can be represented within the same isotopy classes by a self-orbit
equivalence of $\varphi_i$ in $Q$ by Lemma \ref{lem: vertical annular twists orbit equivalences}.
Thus, up to composing $H$ with $D^\inv$, we get an orbit equivalence between $\varphi_i$ and $\varphi_j$
which is isotopic to the identity on the boundary.
\qedhere

\end{proof}

\subsection{Proof of Theorems \ref{thm: finiteness free homotopy data} and \ref{thmintro: finiteness orbit equivalence}}

We can now finish the proofs of Theorem \ref{thm: finiteness free homotopy data} and Theorem \ref{thmintro: finiteness orbit equivalence}.

Recall that we assumed the existence of infinitely many flows $\phi_i$ on $M$ satisfying the assumptions of Convention \ref{conv: finite cover}. Up to taking a subsequence, we may further assume, thanks to Proposition \ref{prop: finite on one periodic piece} and Proposition \ref{prop: finite on skew} that, on the union of skew pieces, the flows $(P_i, \varphi_i)$ are all orbit equivalent, and on each periodic Seifert piece $Q^j$, the flows $(Q^j_i, \varphi^j_i)$ are also orbit equivalent. Moreover, we can assume that $\{Q^j\}_j$ is a nonempty collection, as otherwise Proposition \ref{prop: finite on skew} would already yield the result.\footnote{In fact, we could also assume that there exists at least one skew piece, as otherwise we would have that the flows are totally periodic and hence the result already follows from \cite{barbotClassificationRigidityTotally2015} and Theorem \ref{thm_same_spine_implies_same_flow}. However, we do not need that here. }

Denote by $(N_i, \psi_i)$ the disjoint union of all $P^j_i$ and $Q^j_i$ together with the flow $\varphi_i^j$ on each piece. The flows $(N_i, \psi_i)$ are all orbit equivalent pseudo-Anosov pieces.
Recall that all the pieces $N_i$ are isotopic in the closed manifold $M$ (see Theorem \ref{thm: modified JSJ} and Proposition \ref{prop: cover embedded jsj tori}). 
Up to performing an isotopy on $(N_i, \psi_i)$ we can consider pieces on the same manifold. We denote those pieces $(N, \psi_i)$.
Let $f_i \colon \partial N \to \partial N$ be the pseudo-Anosov gluing map which recovers the pseudo-Anosov flow $\phi_i$ on $M$.
Denote by $H_i\colon N \to N$ the orbit equivalence between $\psi_0$ and $\psi_i$. 
To show that the triples $(N, \psi_i, f_i)$ are equivalent we need to show that the conjugate $(H_i)_* f_0$ is isotopic to $f_i$ for each $i$.
This is satisfied because $H_i$ is isotopic to the identity on the boundary by Proposition \ref{prop: finite on one periodic piece} and Proposition \ref{prop: finite on skew}.
We now apply Theorem \ref{thmintro: uniqueness of gluing} which implies that these triples only induce a finite number of distinct sets of free homotopy data, and, if the flows are assumed to be transitive, a finite number of orbit equivalence classes, leading to a contradiction.

\bibliographystyle{alpha}
\bibliography{biblio}

\end{document}